\documentclass[12pt]{report}
\usepackage{amssymb,amsthm,uwthesis,latexsym,amsmath}
\usepackage{xypic}
\usepackage{mdwlist, subfigure, graphicx}
\graphicspath{{figs/}}
\input xy
\xyoption{all}
\graphicspath{{figs/}}

\newcommand{\be}{\begin{equation}}
\newcommand{\ee}{\end{equation}}
\newcommand{\bfg}{\begin{figure}}
\newcommand{\efg}{\end{figure}}
\newcommand{\dlim}{\displaystyle\lim_{n\rightarrow \infty}}

\newcommand{\nn}{\nonumber}
\providecommand{\abs}[1]{\vert#1\vert}
\providecommand{\norm}[1]{\Vert#1\Vert}

\newcommand{\fl}[1]{\lfloor{#1}\rfloor}

\def\cD{\mathcal{D}}
\def\cE{\mathcal{E}}

\def\cH{\mathcal{H}}
\def\cM{\mathcal{M}}
\def\cN{\mathcal{N}}

\def\cW{\mathcal{W}}
\def\cL{\mathcal{L}}

\def\bE{\mathbb{E}}
\def\bN{\mathbb{N}}
\def\bP{\mathbb{P}}
\def\bQ{\mathbb{Q}}
\def\bR{\mathbb{R}}
\def\bV{\mathbb{V}}
\def\bZ{\mathbb{Z}}

\def\bx{\textbf{x}}
\def\r{\rho}
\def\l{\lambda}

\def\dks{d_{\xi}}
\def\hks{h_{\xi}}

\def\e{\varepsilon}
\def\xvec{\mathbf{x}}
\def\yvec{\mathbf{y}}
\def\vvec{\mathbf{v}}
\def\uvec{\mathbf{u}}
\def\avec{\mathbf{a}}

\def\om{\omega}

\def\e{\varepsilon}

\def\l{\ell}
\def\zsq{Z^{\square}(m,n)}

\def\rf{J}
\def\lmgf{M}

\DeclareMathOperator\dtot{tot}
\DeclareMathOperator\inter{int}

 \def\wt{\widetilde}   

\begin{document}
\title{\sc{Positive and Zero Temperature Polymer Models}}
\author{Nicos Georgiou}
\degree{Doctor of Philosophy}
\dept{Mathematics}
\thesistype{dissertation}
\beforepreface
\prefacesection{Abstract}
We present results about large deviations  and laws of large numbers for various polymer related quantities.
 
In a completely general setting and strictly positive temperature, 
we present results about large deviations for directed polymers in 
random environment. We prove quenched large deviations 
(and compute the rate functions explicitly) for the exit point of the polymer chain and 
the polymer chain itself.

We also prove existence of the upper tail large deviation rate function for the logarithm 
of the partition function. In the case where the environment weights have certain 
log-gamma distributions the computations are tractable and allow us to compute the rate 
function explicitly. 

At zero temperature, the polymer model is now called a last passage model. With a particular choice of random weights, the last passage model 
has an equivalent representation as a particle system called Totally Asymmetric Simple Exclusion Process (TASEP). 
We prove a hydrodynamic limit for the macroscopic particle density and current for TASEP
with spatially inhomogeneous jump rates 
given by a speed function that may admit discontinuities. The limiting
density profiles are described with a variational formula.
This formula enables us to compute explicit density profiles 
even though we have no information about the invariant distributions
of the process.  
In the case of a two-phase flux  for which a suitable 
 p.d.e.\ theory has been developed 
we also observe that the limit profiles are 
entropy solutions of the corresponding scalar conservation law
with a discontinuous speed function. 
\prefacesection{Acknowledgements}
I know that many people consider me to be a storyteller.  Many times during
the past five years people would just pop into my office to chat and have a
laugh, ask me to go to the terrace to hang out, or to make plans to go out for
dinner later.  The end result usually involved me telling stories.  Some
were sad, shocking, or blatantly funny.  To be fair, when I tell a
story there is always a question of reality vs artistic liberty!  But by now
people have favorite stories that they start referencing and asking me to
tell again - such as ``The one where..." or ``The one with ...".  I dedicate
this section to those who helped me create all the stories during my years
in Madison.  So even though I am going to be somewhat vague in what follows,
I hope that they will each understand their part.


First and most importantly I would like to thank my advisor (and hopefully
by now my friend), Timo Sepp{\"a}l{\"a}inen. Naturally, one's advisor is the
source of many tales and I am glad to say that in my case they are all
good.  Timo is the reason I can tell ``The one where Nicos found an
advisor", ``The one when Nicos published his first paper" , ``The one with a
flood of e-mails", and ``The one with the Sensei".  In fact, Timo is the
only reason I am able to write a thesis.  During the stressful time between
passing my quals and making sure I could actually do research by myself, he
was the only light in an otherwise dark world.  Always patient (ridiculously
so) and careful, he taught me so many things about so many things (math and otherwise) that, if listed, would be at least as long as this thesis. I am
especially grateful for ``The one when Nicos was not once called stupid..."
and ``The one where people wanted to work with Nicos' advisor".


Also, I would like to thank the other faculty members of the Probability
group: Tom Kurtz, David Griffeath, Benedek Valk{\'o} and David Anderson.  All of
the them are responsible for stories like `` The one with your professor in
[such and such] class".  Having lunches and dinners with you guys was most
often the highlight of my week.  My gratitude extends even further to Tom
Kurtz for giving me a Research Assistanship during my last semester from his NSF grant DMS-0805793. 
As I recall, ``The one where Nicos got money from
someone that wasn't his advisor," made many graduate students
thinking about switching to probability. 


Naturally, I cannot forget the student members of the probability group.  To
the original seven - Ankit, Arnab, Hao, Hye-Won, Mathew, Rohini and Sabrina
- thank you for ``The one where Nicos was convinced to do Probability",
``The first one in Evanston", and ``The magnificent 8".  It was very nice to
see a group dynamic as clear and refreshing as you guys made the probability
group.

This is a good place to thank my former office-mate Annette.  She was the
one who pointed out the obvious and convinced me that this group of people
liked me and would be happy to be my mathematical siblings and cousins.

I have tried to carry on the `closeness' of the probability students as older
ones graduated and younger ones joined.  Thanks go to my younger
(mathematical) cousins Diane and Maso for ``The one with the reading course"
and ``The one with the practice talks".  Hopefully I made the group as warm
and fuzzy for you as the others did for me.


No one needs to walk alone in this world-especially during a Ph.D. program.
I was very lucky with my inner circle of friends. Firstly, to my roommate
Andrea, thanks for ``The one where Andrea asked me to be his roommate"
(timeless classic), ``The one with the 14 hour sleep" and ``The one with the
cleaning" as well as all of the episodes of the sitcom I am going to write
about us.  Having someone to talk to after a long a day and just sitting
around watching your futile attempts to convince me that $c \neq \om_1$ was
a great source of stress relief.

To Achilles and Kostas (and their parents Alex and Mariam), thank you for
awesome experiences like ``The one with the gym semester" and ``The one with
soda and salt over Easter".  Truly, you were a substitute family for me in a
foreign country.  Living together in the same building was, in all honesty,
the best idea we ever had!


My wing-men, Dan and Johana are responsible for ``The first free
Valentine's day", and ``The one with the weird bar-hopping".  They really
showed me that my friends are awesome and kind, as well as giving me a
renewed trust in people.  If Dan is reading this: please get a grip!  Finally,
Sarah, thanks for ``The one with Kongregate", ``The one with the olive
branch", and various others.  You were truly the best office-mate one can
hope for - especially considering that our desk arrangement requires our desk chairs to occupy the same space.  You are probably the only person who saw all my weird mood swings
and always gave me rational advice.  I am extremely lucky to call you my
friend.


Finally, I would like to thank my family and friends in Cyprus.  To my
parents Andri and Christos, and grandparents Rodou and Giwrgos, thanks for
always supporting me and feeling proud of me.  Even though you have never really
understood what it is that I actually do, you always understood that it is
rare for someone like me.  I would also like to thank them for abandoning
all of their weird schemes that involved finding me a wife.  Thanks also to
my sister Eleni and her husband Simos for giving my parents several
opportunities (including the forth-coming Gandalf and/or Xena) to dote over someone else
for a change.

To my friends back in Cyprus - Stella, Nia, Theodoros, Dafni, Fanos, Fanis
and Loizos - thank you for always making me feel wanted.  It has been
somewhere between five and ten years since we've lived in the same country,
but somehow I am convinced that even if a hundred years pass we will still
be friends and drive each other crazy.  Even if it is going to be through
Skype!  Any possible test for true friendship you passed with flying colors,
so a great THANK YOU might not be enough. But, since you might never read this 
you should believe whatever I tell you - that I wrote pages and pages thanking you.
;)

\phantom{
xxxxxxxxxxxxxxxxxxxxxxxxxx} \,\,Nicos 

\listoffigures
%

\afterpreface


\chapter{Introduction}

\section{Polymers at finite temperature and corner growth models}

We begin by presenting the two main models that are discussed in this dissertation. After the two models are introduced we offer a connection between the two of them (namely one can be viewed as a limiting case of the other). The two remaining sections of the chapter can be viewed as an informal introduction to the material that follows: Some basic definitions, discussion on classical results and an idea of the kind of questions we are asking. At the end of the chapter we describe the organization of the thesis.

\subsection{General polymer models} 

A directed polymer in random environment  is a random walk path that interacts with a random environment. The polymer chains live in $\bZ^{d}\times\bZ_+$, where the last coordinate denotes time.  The space of environments is denoted by  $\Omega = \{ \om(\uvec, n): \uvec \in \bZ^{d}, n\in \bZ_+ \}$  and is equipped with a probability measure $\bP$, so that under $\bP$ the random variables $ \om(\uvec, n)$ are i.i.d.\ for all $\uvec$, $n$. 
 
The two models under consideration are \textit{directed polymers with free endpoints} and \textit{directed polymers with constrained endpoints}. Here, \textit{directed} means that the last coordinate is always increasing by $1$ at each time step. This allows for the polymer chain in $d+1$ dimensions to be viewed as the path of a $d$-dimensional nearest neighbor random walk.

 We assume that at time $t=0$, the starting point of the polymer chain (the random walk) is anchored at $\mathbf0 \in Z^d$. For each $m \in \bN$, define the set of possible endpoints for the polymer to be  $\cE(m)$. For a fixed $(\uvec,m)$ in  $\mathcal{E}(m)$, the set of all polymer chains starting from $(\mathbf0,0)=x_0$   and ending at $(\uvec,m)=x_m$  is 
\begin{align}
\mathcal{R}(\uvec, m)=\big\{ x_{0,m}:& x_{0,m}=(x_0, x_1, \,... \, , x_m),\, x_k - x_{k-1}= (\pm e_i, 1) \label{cm} \\
&\text{ for some }1\le i \le d, \text{ for all } 1\le k\le m,\, (\mathbf0,0)=x_0,\,(\uvec,m)=x_m\big\}\notag,
\end{align}
 where $e_i, 1\le i \le d$ is the standard basis of $\bR^d$.  

The point-to-point partition function is defined by
\be
Z^{\beta}(\uvec, m)= \sum_{x_{0,m} \in \mathcal{R}(\uvec, m) }\prod_{j=1}^m e^{\beta \om(x_j)},
\label{parf0}
\ee
and the total (or the point-to-line) partition function can be defined by
\be
Z^{\beta}_m= \sum_{\uvec \in \cE(m)}\sum_{x_{0,m} \in \mathcal{R}(\uvec, m) }\prod_{j=1}^m e^{\beta \om(x_j)},
\label{parf00}
\ee

The parameter $\beta$ is what is known in the literature as the \textit{inverse temperature} and is assumed without loss of generality to be positive. 

Under a fixed environment $\om$, the polymer chain $x_{0,m}$ is selected according to the quenched probability measures
 \be
Q^{\omega, \beta}_{\uvec, m}\{ x_{0,m}\} = (Z^{\beta}(\uvec, m))^{-1}\prod_{j=1}^m e^{\beta \om(x_j)},
\label{qpp}
\ee
and
\be
Q^{\omega, \beta}_{m}\{ x_{0,m}\} = (Z^{\beta}(m))^{-1}\prod_{j=1}^m e^{\beta \om(x_j)},
\label{tpp}
\ee
respectively for each of the models described above.

Let us momentarily restrict our attention to dimension $1+1$. The polymer chain starts by being anchored at $(0,0)$ and under a fixed environment, at time $n$ the chain is chosen according to the measures  $Q^{\omega, \beta}$. The chain lives inside the cone $\{\uvec=(u_1, u_2) \in \bZ\times \bZ_+ : u_2 \ge |u_1| \}$. We rotate the picture clockwise by $45$ degrees. The polymer chain now becomes a path  $\tilde{x}_{0,m}= \{ (0,0) = x_0, x_1, \cdots, x_m \}$ in the first quadrant $\bZ_+^2$ where the differences $x_k - x_{k-1}$  are a standard basis vector. Such a path we call an \textit{up-right path}. 

The advantage of this viewpoint resides in the point-to-point model: For any vector $\uvec \in \bZ_+^2$ we specify as an endpoint, it is guaranteed that exponentially many paths start at $(0,0)$ and end at $\uvec$, i.e. $\uvec$ is an admissible endpoint. Unfortunately we lose information about the time (now the time axis is the main diagonal) but this does not affect the analysis and the limits of the main theorems. 

From this point onwards, all models in this dissertation are about paths that live in the first quadrant and are up-right paths. We denote the set of up-right paths from $(0,0)$ to $(\uvec)$ by $\Pi(\uvec)$. More precise details are available in the following chapters. 

\subsection{Corner growth model}

The corner growth model in two dimensions is a model of a randomly growing cluster that over time covers larger and larger portions of the first quadrant $\bZ^2_+$. At the outset, each coordinate $(i,j)$ is given a weight $\om(i,j)$. In the language of the previous section, we have now a fixed environment and we want to study the evolution of the random cluster according to the following rules.

The general rule is that $\om(i,j)$ is the time it takes for the random evolution to occupy site
$(i, j)$. This can only be done only after its two neighbors to the left and below are either occupied or lie outside $\bZ^2_{+}$. At the boundaries the rule is that point $(0, 0)$ needs no occupied neighbors to start, points $(1, j)$ on
the left boundary wait only for the neighbor below to be occupied, and points $(i, 1)$ on the bottom
boundary wait only for the left neighbor to be occupied (see Fig. \ref{cluster}). 

\begin{figure}
\begin{center}
\includegraphics[ scale=1.3, trim= 130 350  50 350 ]{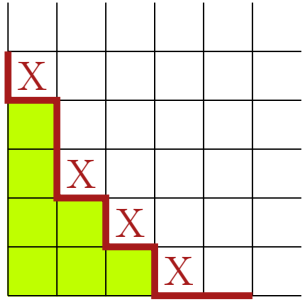}
\end{center}
\caption[Corner growth of a random cluster]{The shaded squares represent sites that are occupied at this particular time. The thick line represents the boundary of the cluster and the squares marked with $X$ are the `` growth sites ". The time it takes each growth site to be infected (shaded) is $\om_{0,4}, \om_{1,2}, \om_{2,1}, \om_{3,0}$ respectively. These times start elapsing as soon as both the left and lower neighbors become shaded.}
\label{cluster}
\end{figure}

The quantity of interest is that of the \textit{last passage time}. Given a site $(i,j)$ the last passage time, $T(i,j)$, is the time when the site $(i,j)$ becomes part of the evolving cluster. Using the evolution described above, it can be recursively computed (with appropriate boundary conditions) to be
\be
T(i,j) = T(i-1, j)\vee T(i, j-1) + \om_{i,j}.
\label{lptrec}
\ee
Equation \eqref{lptrec} says that $(i,j)$ becomes part of the cluster only after both $(i-1, j)$ and  $(i, j-1)$ are part of the cluster, and after that happens, time  $\om_{i,j}$ elapses. Assume without loss that $T(0,0) = \om(0,0)=0$. An easy induction argument then yields that 
\be
T(i,j) = \max_{x_{0, i+j} \in \Pi(i,j)} \sum_{k=1}^{i+j} \om_{x_k}.
\label{LPT}
\ee

In the case where the $\omega$  weights are continuous and an endpoint $\uvec = (i,j)$ is specified, there is a unique path $x_{0,i+j}$ for $\bP$- a.e. $\om$ that attains the last passage time. We call that the \textit{maximal path} and denote it by $x^{\text{max}}_{0, i+j}$. In this respect, we can define a (degenerate) quenched measure on the paths,
\be
\widetilde{Q}^{\om}_{\uvec, i+j}\{ x_{0, i+j}\} = \delta_{x^{\text{max}}_{0, i+j}} = \begin{cases}
                                                                                                                                             1, &  x_{0, i+j}=x^{\text{max}}_{0, i+j}\\
                                                                                                                                             0, & \text{otherwise.}
                                                                                                                                             \end{cases}
\label{cgmqm}
\ee
 If the $\om$ weights have discrete distributions, is possible that there are more than one maximal path. When that happens, the quenched probability measure on the paths is the uniform measure on maximal paths. 
 
\subsection{Connection of the two models}

Now we are ready to justify the title of this dissertation.
The connection between the polymer models and the last passage time models comes via the parameter $\beta$. As we let $\beta$ tend to $\infty$, under a fixed environment, the quenched probability measure $Q^{\om, \beta}$ converges weakly to a delta mass, given by \eqref{cgmqm}. This becomes precise in the context of the next proposition.

\begin{proposition}
Let $\uvec \in \bZ^2_+$, $\uvec=(i,j)$ and let $m=i+j$. Fix an environment $\omega$. Then, the probability measures defined by \eqref{qpp} converge weakly
\be
Q^{\beta, \om}_{\uvec, m}\{ x_{0, m}\}\Longrightarrow \widetilde{Q}^{\om}_{\uvec, m}\{ x_{0, m}\}= \delta_{x^{\text{max}}_{0, i+j}} , \quad\text{as } \beta \to \infty.
\ee 
\end{proposition}  
\begin{proof}
Let $f$ a function on the space $\Pi(\uvec)$ of up-right paths. Then
\begin{align}
\int_{\Pi(\uvec)}f \, dQ^{\beta, \om}_{\uvec, m} &= \sum_{x_{0, m}\in \Pi(\uvec)}f(x_{0,m})Q^{\beta, \om}_{\uvec, m}\{x_{0,m}\}\notag \\  
&= \sum_{x_{0, m}\in \Pi(\uvec)}f(x_{0,m})Z^{\beta}({\uvec, m})^{-1}\prod_{j=1}^{m}e^{\beta \om(x_j)}\notag \\  
&= f( x_{0,m}^{\text{max}})Z^{\beta}({\uvec, m})^{-1}\prod_{j=1}^{m}e^{\beta \om(x^{\max}_j)} \notag\\
&\phantom{xxxxxxxxxx}+ \sum_{\substack{x_{0, m}\in \Pi(\uvec)\\ x_{0, m} \neq x_{0,m}^{\text{max}}}}f(x_{0,m})Z^{\beta}({\uvec, m})^{-1}\prod_{j=1}^{m}e^{\beta \om(x_j)}\notag \\  
&= f( x_{0,m}^{\text{max}})Z^{\beta}({\uvec, m})^{-1}\prod_{j=1}^{m}e^{\beta \om(x^{\max}_j)} \label{aaaaa}\\
&\phantom{xxxxxxxxxx}\times\Bigg( 1 + \sum_{\substack{x_{0, m}\in \Pi(\uvec)\\ x_{0, m} \neq x_{0,m}^{\text{max}}}}f(x_{0,m})e^{\sum_{j=1}^{m}\beta (\om(x_j) - \om(x^{\text{max}}_j))}\Bigg)\label{bbbbb}
\end{align}
As $\beta \to \infty $, \eqref{bbbbb} tends to $1$ because all the exponents are negative. We are going to show that $Z^{\beta}({\uvec, m})^{-1}\prod_{j=1}^{m}e^{\beta \om(x^{\max}_j)}$ tend to $1$ as $\beta \to \infty$. Then the proposition follows from \eqref{aaaaa} and $\eqref{bbbbb}$.

From the definitions, $Z^{\beta}({\uvec, m})^{-1}\prod_{j=1}^{m}e^{\beta \om(x^{\max}_j)} \le 1$. For a lower bound, observe that for any $\delta > 0$ and $\beta$ sufficiently large
\begin{align}
Z^{\beta}({\uvec, m})\Big(\prod_{j=1}^{m}e^{\beta \om(x^{\max}_j)}\Big)^{-1}&= \sum_{x_{0, m}\in \Pi(\uvec)}e^{\sum_{j=1}^{m}{\beta (\om(x_j) - \om(x^{\max}_j) )}}\notag\\ 
&=1 + \sum_{\substack{x_{0, m}\in \Pi(\uvec)\\ x_{0, m} \neq x_{0,m}^{\text{max}}}}e^{\sum_{j=1}^{m}{\beta (\om(x_j) - \om(x^{\max}_j) )}}\notag\\ 
&\le 1+\delta. \notag\quad\quad\quad\phantom{xxxxxxxxxxxxxxxxxxxx}\qedhere
\end{align} 
\end{proof}

\section{Large deviations}

The theory of large deviations is concerned with the study of rare (improbable) events. Assume you have a sequence $\{X_n\}_{n\in \bN}$ of real random variables on $\Omega$. A natural question that arises is to compute the limiting probability $\lim_{n\to \infty} \bP\{ X_n \in A \}$ where $A$ is a fixed measurable set on $\bR$. When the events $\mathcal{A}_n = \{ X_n \in A \}$ are ``rare'' the limiting probability can be $0$. This information, while not particularly helpful, leads to a deeper question: \textit{How fast} does it go to $0$? For example if the probabilities are now summable, one can apply the Borel-Cantelli lemma.

In a more general setting, instead of a sequence of random variables, one can use a sequence of probability measures $\mu_n$ on a Polish measurable space $\mathcal{X}$ (in the example above, $\mu_n = \bP\{ X_n \in \cdot \}$ and $\mathcal{X} =\bR$). The measures $\mu_n$ live in the space of probability measures on $\mathcal{X}$, $\mathcal{M}_1(\mathcal{X})$ and we say that they satisfy a large deviation principle if the following definition holds.    

\begin{definition}  Let $I : \mathcal{X} \mapsto [0,\infty]$ be a lower semicontinuous function
and $r_n \to \infty$ a sequence of positive constants. A sequence of probability
measures $\{\mu_n\}_n \subseteq \cM_1(\mathcal X)$ is said to satisfy a large deviation principle with
rate function $I$ and normalization $r_n$ if the following inequalities hold for all
closed $F \subseteq \mathcal X$ and all open $G \subseteq \mathcal X$:

\be
\varlimsup_{n\to \infty}\frac{1}{r_n}\log \mu_n(F)\le - \inf_{F} I;
\label{upra}
\ee

\be
\varliminf_{n\to \infty}\frac{1}{r_n}\log \mu_n(G)\ge - \inf_{G} I;
\label{lora}
\ee

\end{definition}

When the sets $\{I\le c \}$ are compact for all $c \in \bR$, we say $I$ is a \textit{tight rate function}. It is of interest to find explicit rate functions, since they offer an exact measurement of the rare event. For example, insurance companies can use that information to decide on a fair premium for the customer.  

\section{Interacting particle systems and hydrodynamic limits}  

In full generality, interacting particle systems consist of finitely or infinitely many particles that evolve in space and time according to given transition probabilities or rates, with some interaction rules imposed on the particles. A particle system that has been extensively studied, is the  \textit{Totally Asymmetric Simple Exclusion Process} (TASEP). It is intimately connected with the last passage time, assuming exponential weights. 

 Assume that particles occupy integer sites. There is at most one particle at each site (Simple).  Each particle attempts to move one unit to the right (Totally Asymmetric) with rate $1$ . The jump is suppressed with probability $1$ if the target site is occupied (Exclusion).

One-dimensional TASEP can be constructed graphically by assigning independent mean $1$ Poisson processes (called \textit{clocks}) on each integer site. The vertical direction now becomes time. As time progresses, a particle on site $i$ attempts to jump at the Poisson event times of the site it occupies, and the jump happens with probability $1$ as long as the exclusion rule is satisfied. With  probability 1, two adjacent Poisson processes cannot have simultaneous events, and for every time $t$ there are infinitely many Poisson processes with no events before $t$, so one can study the temporal evolution of the system up to time $t$ in a rectangle $[-M,N]\times[0,t]$ around the origin. 

 The coupling with the corner growth model follows if we run a TASEP starting from step initial conditions: particles start by occupying only the negative integers and are labeled so that particle $ i$ starts on site $-i$. Then the last passage time $T(m,n)$ is equal in distribution to the time it takes the $n$-th particle to reach site  $m-n$. 

  Consider
 a sequence of exclusion processes $\eta^n = (\eta^n_i(t): i \in \bZ, \, t\in\bR_+)$ 
  indexed by $n \in \bN$. For each $i$ and fixed $t$, $\eta^n_i(t) = 1$ if and only if  
there is a particle present at site $i$, at time $t$.
These processes are constructed on a common probability 
space that supports the initial configurations $\{\eta^n(0)\}$ 
and the Poisson clocks of each process.  The clocks of
process $\eta^n$ are assumed to be independent of its initial state $\eta^n(0)$.

Starting from arbitrary particle initial conditions and for fixed time $t$ define the sequence of \textit{occupation measures}
\be
\mu^n_{t}([a,b])= \frac{1}{n}\sum_{i=\fl{na}+1}^{\fl{nb}}\eta^n_i(nt) = \frac{1}{n}\sharp\{ \text{ particles in } [na, nb] \text{ at time } nt\}.
\label{occmeas}
\ee

 Under some regularity assumptions on the initial conditions,  it is known (e.g. \cite{sepp99K}) that this sequence of measures converges weakly for all $t$ 
\be
\lim_{n\rightarrow \infty} \mu^n_{t}([a,b]) = \int_{a}^{b}\rho(x,t)\,dx,
\label{hydro}
\ee
 where $\rho(x,t)$ (called \textit{the particle density function})  is the unique entropy solution to the scalar conservation law
\be
\partial_t(\rho(x,t))+ \partial_x F(x,\rho(x,t))=0.
\label{scl}
\ee
The initial conditions of the particles correspond to the initial conditions required for uniqueness of the weak solution in \eqref{scl} and $F$ is the particle flux of TASEP, given by $F(x,\rho) = \rho(1-\rho)$. Results of this type go by the name of \textit{hydrodynamic limit} and there are many known generalizations.

\section{Motivation}

Upper tail large deviations for the last passage time have been computed in the case of geometric and exponential weights (see \cite{joha, sepp-large-deviations}). In \cite{sepp-large-deviations}, the rate function was computed via the height function and information about equilibrium distributions for TASEP particles. The equilibrium distributions for TASEP is a result of \textit{Burke's theorem} for M/M/1 queues and they can be interpreted as appropriate boundary weights in the corner growth model. This (Burke) property is ``transferable'' to the log-gamma polymer model. The model was introduced in \cite{sepp-poly}. The model is a $1+1$ dimensional polymer model at temperature $\beta =1$ where for a fixed $\mu > 0$, the weights 
\[\om(i,j) = -\log Y_{i,j}, \quad \text{ and } Y_{i,j}\sim \text{Gamma}(\mu) \]
for $(i,j)\in \bN^2$.

\section{Organization}

 We start (Chapter 2) by describing general polymer models. In Chapter 2 we show some general facts about the partition function and proceed by showing quenched large deviation results for the polymer chains under the quenched measures \eqref{qpp}, \eqref{tpp}. In a completely general setting we show some quantitative properties of the rate functions (existence, continuity, large $\beta$ behavior).

 In Chapter 3 we restrict to a specific $1+1$ dimensional model, the log-gamma model that satisfies a certain property that allows for explicit computations. We present large deviation  results about the logarithm of the partition functions.  

The following two chapters (Chapters 4 and 5) are concerned about hydrodynamic limits of exclusion processes and last passage time where the weights $\om$ are now exponential but with different parameters (so they are not identically distributed). The parameters are decided using a (possibly discontinuous) function $c(x)$. The connection with the particle process TASEP  leads to a further connection with scalar conservation laws with discontinuous coefficients. 



\chapter{Generalities about Polymer Models}\label{generalpolymer}

\section{Introduction}

\subsection{Directed polymers with constrained endpoint in a rectangle} 

For $\uvec\in\bZ^d_+$ define the set of directed polymer chains from $\vvec$ to $\uvec$ with $\| \uvec - \vvec\|_1= m$
\begin{align}
\Pi_{\vvec}(\uvec) &= \big\{ x_\centerdot = \{ \vvec=x_0, x_1,\dotsc, x_{m} = \uvec \}: x_k \in \bZ^d_+,\notag \\  
&\phantom{zxxxxxxxxxxxxxbbbbbbxxxxxxxxxx}x_{k+1} - x_k =\mathbf e_i \text{ for some }1\le i\le d \big\},
\end{align}
where $e_i $ is the $i$-th standard basis vector. 

The point-to-point partition function is in this case defined by    
\be
Z^{\beta}_{\vvec, \uvec}= \sum_{ x_\centerdot  \in \Pi_{\vvec}(\uvec) }\prod_{j=1}^{m} e^{\beta \om(x_j)} = \sum_{ x_\centerdot  \in \Pi_{\vvec}(\uvec) }e^{\beta \sum_{j=1}^m \om(x_j)} .
\label{parf}
\ee 

In the special case where $\vvec = \mathbf0$ we omit the index $\vvec$ from the above notation: The partition function is denoted by  $Z^{\beta}_{\uvec}$ and the set of polymer chains by $\Pi(\uvec)$. Observe that in the definitions given so far, the weight at the starting point is ignored. When that is not the case, we denote 
\[
Z^{\beta, \square}_{\vvec, \uvec}= e^{\beta \om(\vvec)}Z^{\beta}_{\vvec, \uvec}. 
 \]
Under a fixed environment $\om$, fixed endpoint $\uvec$ and fixed inverse temperature $\beta$, the \textit{quenched probability measure} $Q^{\om,\beta}_{\uvec}$ on paths with constrained endpoints, is defined by 
\be
Q^{\om,\beta}_{\uvec}(x_\centerdot ) = \big(Z^{\beta}_{\uvec}\big)^{-1} \prod_{j=1}^{m}e^{\beta\om(x_j)}.
\label{queme}
\ee

\subsection{Directed polymers with constrained endpoint in a rectangle} 

In this variation we fix the number of time steps $m$. Define  
the set of all admissible polymer chains starting from $\mathbf0$  to be
\begin{align}
\Pi_{\dtot}(m)&=\big\{ x_{0,m}: x_{0,m}=\big\{x_0,x_1,\ ,...\, x_m\big\}, \,x_{k+1} -x_{k}= e_i \label{cm2} \\
&\phantom{xxxxxxxxxxxxxxxxxxxxxxxxxxx}\text{ for some }1\le i \le d, \text{ for all } 1\le k\le m\big\}\notag
\end{align}
 where $e_i, 1\le i \le d$ is the standard basis of $\bR^d$.  

For each $m> 0$, the total partition function is defined by 
\be
Z^{\beta, \dtot}_m= \sum_{\uvec \in \bZ_+^d: \|\uvec\|_1=m} Z^{\beta}_{\uvec}.
\label{totparf}
\ee

The corresponding quenched probability measure  $Q^{\om,\beta}_{m}$ on paths in $\Pi_{\dtot}(m)$, is 
\be
Q^{\om,\beta}_{m}(x_{0, m}) = \big(Z^{\beta, \dtot}_m\big)^{-1} \prod_{j=1}^{m}e^{\beta\om(x_j)}.
\label{totqueme}
\ee

\begin{remark}
If the rectangle has dimension $2$, then both models describe the classical directed random polymers in random environment in dimension 1+1 described in the introduction (point-to-point and free endpoint respectively), where the picture is rotated by 45 degrees, so the polymer lives in the first quadrant.
\end{remark}

\subsection{Known results}

\paragraph{Concentration inequalities.} A problem in the area of random polymers in random environment is about the fluctuations (in particular the fluctuation exponent) of $\log Z^{\beta}_m$ that is conjectured to be $2/3$  in the physics  literature. Rigorous results about upper and lower bounds for fluctuations exponents for specific polymer models can be found in \cite{beze-tind-vien, come-yosh-05, meja-04, petermann, wuth-98aihp, wuth-98aop}. The only two cases where the conjectured value of $2/3$  is in fact verified is the log-gamma polymer \cite{sepp-poly} and the brownian polymer \cite{sepp-valk-poly}. 

In absence of exact variance bounds, information can be derived from concentration inequalities. In \cite{Carmona-Hu-Gaussian} an exponential of order $n^{1/3}$ concentration result has been  obtained for the partition function in a Gaussian environment and later this had been generalized for all weights under certain exponential moment assumptions in \cite{come-shig-yosh-03}. 

The sharpest concentration  inequalities are exponential of order $n$,  proven in \cite{Liu-Watbled-2009, watbled-arc} and a certain version used here about the point-to-point free energy in \cite{Comets-Gregorio-arc}.   

\paragraph{Large deviations.} The exponential order $n$ concentration inequality for the partition function is the correct one for the upper tail large deviations and is the correct normalization in order to get a non-trivial upper tail rate function. 

For the lower tail the behavior depends on the distribution of the weights. In \cite{Ben-Ari} three different normalization regimes for the lower tails are shown, depending on the distribution of the weights. It is also mentioned that a normalization $n^2$ is true if the weights are bounded. This is proved in the case where the weights are Gaussian or bounded in \cite{Carmona-Hu-Gaussian} where upper and lower normalizations are proven. 

In the last part of the chapter we prove quenched large deviations for the exit point of the polymer (in the free endpoint case) and quenched large deviations for the polymer path in both the free endpoint and constrained endpoint. The results are described in further detail in Section \ref{results}. 

\paragraph{Notation and conventions.} Throughout we use $\bN$ for positive integers, while $\bZ_{+}$ denotes the set of non-negative integer. Similarly, $\bR_{+}$ denotes the set of non-negative real numbers and $\bR^d_+$ is the set of all vectors with non-negative coordinates. All vectors $(v_1,v_2,\dotsc, v_d) \in \bR^d$ are denoted by boldface  notation $\vvec =  (v_1,v_2,\dotsc, v_d)$. The ordering $\vvec < \uvec$ means $v_1 \le u_1, v_2 \le u_2,\dotsc, v_d \le u_d$. The $d$-dimensional vector with entries equal to $1$ is denoted by $\mathbf{1}= (1,1,\dotsc, 1)$ and correspondingly, $\mathbf{0}= (0,0,\dotsc, 0)$. For the purposes of this dissertation we define for $\yvec = (y_1,y_2,\dotsc, y_d) \in \bR^d$, $n\in \bN$ the \textit{floor of a vector} $\fl{n\yvec} = (\fl{ny_1}, \fl{ny_2},\dotsc, \fl{ny_d})$.

\section{Definitions and Statement of Results}
\label{results}

Some technicalities before stating the general results. We need a technical assumption on the environment $\om$. Henceforth we assume 
\begin{assumption} 
There exists some $\xi> 0$ that depends on the distributions of the weights $\om$ such that 
\be
\bE\big(e^{\xi |\om(\uvec)|}\big)< \infty. 
\ee
\label{Cramerass}
\end{assumption}

This guarantees the existence of a large deviation rate function defined by
\be
I(r) = -\lim_{\e \rightarrow 0}\lim_{n\rightarrow \infty} n^{-1} \log \bP\Big\{ n^{-1}\sum_{i=1}^{n}\om(\uvec_i)   \in (r-\e, r+\e)\Big\}.
\label{c1}
\ee 
All results that follow are valid for $\beta < \xi$. In order to make the proofs cleaner we also assume that for all $\uvec \in \bZ^d_+$, 
\be
\bP\{ \om (\uvec) \ge 0 \} > 0.
\label{ndeg}
\ee

We start by proving some qualitative properties for the limits of the rectangle partition functions and we summarize them in the following two propositions. The superscript $\beta$ is considered fixed and henceforth omitted until the end of the chapter where it becomes relevant to take limits as $\beta \to \infty$.

\begin{proposition}
Let  $\yvec = (y_1,y_2,\dotsc, y_d) \in \bR_+^d$, $n\in \bN$ and let  $ Z_{\fl{n\yvec}}$ be defined by \eqref{parf}. Then the limit 
\be
\lim_{n\to \infty} n^{-1} \log Z_{\fl{n\yvec}} = p(\yvec) \quad \bP-a.s.
\ee
 Furthermore, there exists an event $\Omega_0 \subseteq \Omega $ of full probability such that the a.s.\ convergence happens simultaneously for all  $\yvec \in \bR_+^d$. The limiting value viewed as a function of $\yvec$ satisfies the following properties:
\begin{enumerate}
\item[(a)] $p(\cdot)$ is concave and continuous on $\bR_+^d$.
\item[(b)]  $p(c\yvec) = c p(\yvec)$ for $c>0$.
\item[(c)] For all $\yvec $  such that $\|\yvec\|_1=1$, $p(\yvec) \le p(d^{-1}\mathbf{1})$.
\end{enumerate}
\label{substd}
\end{proposition}

\begin{proposition}
Let  $n\in \bN, t>0$ and let  $ Z^{\dtot}_{\fl{nt}}$ be defined by \eqref{totparf}. Then the limit 
\be
\lim_{n\to \infty} n^{-1} \log Z^{\dtot}_{\fl{nt}} = \rho^{\dtot}(t) \quad \bP-a.s.
\ee
 Furthermore, 
\be
\rho^{\dtot}(t) = \sup_{\yvec: \|\yvec\|_1=t}p(\yvec) = p(d^{-1}t \mathbf{1}) 
\ee
\label{???}
\end{proposition}

These propositions suffice to guarantee the following general existence theorems. Then, some quantitative properties of the rate functions follow. 

\begin{theorem}
Under assumption \ref{Cramerass}, the following large deviations principles hold. 
\begin{enumerate}
\item  For $t>0$, $\uvec \in \bR_+^d\setminus\{\mathbf0\}$  and $r \in \bR$ there exists a nonnegative function $\rf$ that satisfies 
\be
 \rf_{\uvec}(r) = -\lim_{n\rightarrow\infty} n^{-1}\log \bP\{ \log Z_{\fl{n\uvec}} \ge nr\}.
\label{psidef}
\ee 
$ \rf$ is convex in the variable $(\uvec,r)$ and equals 0 on the set  $r \le p(\uvec)$. The rate function is continuous in $(\uvec, r)$ on the set $ \inter\{ (\uvec, r) : \uvec \in \bR^d_+\setminus\{\mathbf0\}, r \in \bR,  \rf_{\uvec}(r) < \infty \}\subseteq \bR^d_+\times \bR$.

\item For  $t>0$ and $r \in \bR$, there exists a nonnegative function $\rf$ that satisfies 
\be
 \rf_{ t}(r) = -\lim_{n\rightarrow\infty} n^{-1}\log \bP\{ \log Z^{\dtot}_{\fl{nt}} \ge nr\}.
\label{psidef2}
\ee
$ \rf$ is convex in the variable $(t,r)$ and equals 0 on the set  $r \le \rho^{\dtot}(t) $. The rate function is continuous in $(t, r)$ on the set $\inter\{ (t, r) :  \rf_{t}(r) < \infty \}$.

\end{enumerate}
\label{phidef}
\end{theorem}

Next, we show some quantitative properties of the rate functions described in the following three propositions. 

The first one is about the behavior of  the rate function $ \rf_{\uvec}(r)$ at the lower dimensional boundaries of $\bR^d_{+}$. For a vector $\uvec=(u_1, \ldots , u_d) \in \bR^d_+$ we denote by $\uvec_{1,k} = (u_1, \ldots, u_k,0,\ldots,0)$ the projection onto the k-dimensional  facet of the first quadrant. 

It is possible that $\uvec = \mathbf0$ so we need a definition for the rate function at $\mathbf0$. If we just change $\uvec$ with $\mathbf0$ in \eqref{psidef} the rate function becomes trivial (takes values $0$ and $\infty$ only). However, $\mathbf0$, is the macroscopic endpoint of the partition function. The definition of $\rf_{\mathbf0}$ should cover for example the rate function of a single random variable. It is consistent (in the sense that $\lim_{\uvec\to \mathbf0}\rf_{\uvec}(r) =\rf_{\mathbf0}(r)$  described in the following proposition) to define it the following way:
\be
\rf_{\mathbf0}(r)=\begin{cases}
                               0, & r<0, \\
                               x_\infty r, & r\ge 0
                               \end{cases}
\label{j0def}
\ee
where $x_\infty \in (0,\infty]$  is the 
maximal slope of the one-sided Cram{\'e}r rate function on the right of the zero for sums of i.i.d.~ $\om$ weights and is given by $x_{\infty} = \lim_{r \to \infty} r^{-1}I(r)$.

\begin{proposition}[Continuity at the boundaries]
Let $\uvec \in \bR^d_+$ and fix $r \in \bR$. Assume that $\rf_{\|\uvec\|_1\mathbf e_1}(r) < \infty$ in a neighborhood of $\|\uvec\|_1\mathbf e_1$. (This is the one sided Cram{\'e}r rate function.) Then, for any $k \ge 1$  and sequence $\{\uvec^{(m)}\}_{m\in \bN}$ where $\uvec^{(m)} \le \uvec$, $\lim_{m\to \infty} \uvec^{(m)}= \uvec_{1,k}$, 
\be
\lim_{m\to \infty}\rf_{\uvec^{(m)}}(r)=\rf_{\uvec_{1,k}}(r). 
\ee
In the special case where $\rf_{\mathbf0}(r)< \infty$, the function $\rf_{(\cdot)}(r)$  is continuous everywhere on $\bR^d_+$. 
\label{bdc}
\end{proposition}

\begin{proposition}[Unique zero]  \quad \phantom{ghoststeps}
\begin{enumerate}
\item The rate function $\rf_1(r)$ given by \eqref{psidef2} is strictly positive for $r > \rho^{\dtot}(1)$.
\item The rate function $\rf_{\uvec}(r)$ given by \eqref{psidef} is strictly positive for $r > p(\uvec)$.
\end{enumerate}
\label{u0}
\end{proposition}

As stated in the introduction, the polymer model starts behaving like the last passage time model when $\beta \to \infty$. 
For a fixed $\uvec$, define the last passage time  as in \eqref{LPT}.
Assumption \ref{Cramerass} along with the subadditive ergodic theorem, give the existence of a finite limiting last passage time constant \[ \lim_{n\to \infty}n^{-1} T(\fl{n\uvec}) = \psi(\uvec),\]
with arguments identical to those used in the proof of Proposition \ref{substd}. In turn, one can show the existence of an upper tail large deviation rate function for the last passage time, following the steps of the proof of Theorem \ref{phidef}. The rate function is given by \[ I^{\infty}_{\uvec}(r) = - \lim_{n\to \infty} n^{-1}\log \bP\{  T(\fl{n\uvec}) \ge nr \}. \]    

The rate functions also start behaving similarly as described in the following proposition.

\begin{proposition}[Large $\beta$ behavior]
Let $\uvec \in \bR^d_+$ and assume $I^{\infty}_{\uvec}(r)$ is left continuous at $r\in \bR$. Then
\be
\lim_{\beta\to \infty}\rf^{\beta}_{\uvec}(\beta r) =  I^{\infty}_{\uvec} (r).
\ee 
\label{lvbehavior}
\end{proposition}

We also prove quenched large deviations for the exit point for the free endpoint models, and the path of a polymer chain for both models.

\begin{theorem}[Exit point LDP] Let $\uvec  = (u_1,u_2,\dotsc, u_d)\in \bR^d_{+}$ with $\|\uvec\|_1=1$ and let $\yvec = (u_1,\dotsc, u_{d-1}) \in \bR^{d-1}_+ $. For $n \in \bN$ let $[n\uvec]= (\fl{n\yvec }, n - \|\fl{n\yvec }\|_1)$ and denote by $x_{n}$ the last point of the polymer chain $x_{0,n}$.  Then 
\be
p(d^{-1}\mathbf{1}) - p(\uvec) = - \lim_{n\rightarrow \infty}n^{-1}\log Q^{\om}_{n}\{x_{n} = [n\uvec ]\}.
\ee
\label{qend}
\end{theorem}
This readily leads to the following path large deviations.

\begin{theorem}[Path LDP]
Let $\gamma: [0,1]\to \bR_+^d$ be a coordinatewise nondecreasing Lipschitz curve with $\gamma(0)=0$. For   $\e >0$ let  $\mathcal{N}_{\e}(\gamma)$ denote a uniform  $\e$-neighborhood of $\gamma$. The following large deviation principles hold:
\begin{enumerate}
\item (Constrained endpoint.) Let $\uvec$ in $\bR_+^d$ and let $\gamma(1)=\uvec$. Then,
\be
 p({\uvec})  - \int_0^1 p({\gamma'(t)}) \,dt = -\lim_{\e\to 0} \lim_{n\rightarrow \infty}n^{-1}\log Q^{\om}_{\fl{n\uvec}}
\big\{x_{\centerdot} \in n\mathcal{N}_{\e }(\gamma)\big\}
\ee
\item (Free endpoint.) Let $\|\gamma(1)\|_1=1$ . Then,
\be
 p(d^{-1}{\mathbf{1}}) -\int_0^1  p(\gamma'(t)) \,dt = -\lim_{\e\to 0} \lim_{n\rightarrow \infty}n^{-1}\log Q_n^{\om}
\big\{x_{0,n} \in n\mathcal{N}_{\e }(\gamma)\big\}
\ee
\end{enumerate}
\label{qpath}
\end{theorem}

\begin{remark}
The theorem is true for any curve $\gamma$ that admits a Lipschitz parametrization. It is easy to check that the rate functions in \eqref{qpath} are independent of any $C^1$ parametrization $\phi: [0,1] \to [0,1]$. It follows from the $1$- homogeneity of $p$ and a change of variables in the integral.
\end{remark}

\section{Preliminaries}
In this section we record results that are used throughout. In particular, basic properties of the partition functions; their laws of large numbers and the proof of Propositions \ref{substd}, \ref{???}. For the continuity of the partition functions at the boundary and for the unique zeros of the rate functions, we need a concentration inequality which we state first. At the end of the section we prove an auxiliary lemma about upper-tail large deviations of a sum of independent random variables and conclude the section with a basic fact that we use at various instances without alerting the reader.

\begin{lemma}[Proposition 3.2.1 \cite{Comets-Gregorio-arc}]
Under assumption \ref{Cramerass}, there exist constants $c_1, c_2 \in (0, \infty)$ such that for all $\e >0$ sufficiently small
\be
\bP\{ |\log Z_{n} - \bE \log Z_n |>n\e \}\le 2 \exp\{ -c_1\e^2 n\}
\ee and
\be
\bP\{ |\log Z_{\fl{n\uvec}} - \bE \log Z_{\fl{n\uvec}} |>n\e \}\le 2 \exp\{ -c_2(\uvec)\e^2 n\}
\ee
\label{321} 
\end{lemma}
\begin{proof}
The proof is identical with the one in \cite{Comets-Gregorio-arc}, p.26, tailored to these particular rectangle models.
\end{proof}

For $\yvec \in \bZ_+^{d}$, define the coordinate shift operator $T_{\mathbf{y}}$ so that 
\be
Z_{\uvec,\vvec}\circ T_{\yvec} = Z_{\uvec+\yvec, \vvec+\yvec}
\label{logZsub}
\ee
and observe that the random variables  $ \log Z_{\uvec, \vvec}$ are superadditive, satisfying 
\begin{align}
 \log Z_{\uvec+\vvec} &\ge \log Z_{\uvec}  +  \log Z_{\uvec,\uvec+\vvec} \label{G-sub} \\
                                        &= \log Z_{\uvec}  +  \log Z_{\vvec}\circ T_{\uvec}.\notag
\end{align}

\begin{proof}[Proof of Proposition \ref{substd}]
Assume without loss of generality that $\bE(\om(\mathbf0)) >0$. This builds the monotonicity of the limiting free energy that will simplify the proof and makes the limiting free energy a positive function. The proof for concavity and homogeneity when $\yvec \in \bN^d$ and $c \in \bN$ follows from the subbadditive ergodic theorem. For the monotonicity of the limit, let $\uvec \le \vvec \in \bZ_+^2$ and for given $n$ fix a path $\pi_{n}$ from $n\uvec$  to $n\vvec$. Superadditivity gives
\be
\log Z_{n\vvec} \ge \log Z_{n\uvec} + \beta \sum_{x_j \in \pi_{n}}\om(x_j).
\label{pismono}
\ee 
Taking a limit along a suitable subsequence after dividing by $n$ gives that $p(\vvec) \ge p(\uvec) + \|\vvec-\uvec\|_1\bE(\om(\mathbf0))$, hence $p$ is monotone on integer vectors. 

For rational $\yvec \in \bQ^d$, find a positive integer $k$ so that $k\yvec \in \bN^d$ and define
\be
p(\yvec) = k^{-1}p(k\yvec).
\label{ratho}
\ee  
Observe that homogeneity for integer $ k $ gives that the value in \eqref{ratho} is independent of the choice of $k$. Homogeneity and concavity extend now to rational $c >0$ and rational $\yvec$ as follows.  

Let $n\in \bN$ and write it as $n = Mk +r$ with $r \in \{ 0,1,\dotsc, k-1 \}$. Then 
\be
Mk\yvec \le \fl{Mk\yvec + r\yvec} = \fl{n\yvec} \le (M+1)k\yvec.
\ee
Hence, by the superadditivity 
\be
 Z_{(M+1)k\yvec}\Big(Z_{\fl{n\yvec}, (M+1)k\yvec}\Big)^{-1}\ge Z_{\fl{n\yvec}}\ge Z_{Mk\yvec}\,Z_{Mk\yvec, \fl{n\yvec}}.
\label{sss}
\ee

We show the upper bound. The lower one follows in the same manner. Observe that $\| (M+1)k\yvec - n\yvec \|_1 \le k\| \yvec\|_1$, so as $n\to \infty$,  $n^{-1}\log Z_{\fl{n\yvec}, (M+1)k\yvec} \to 0$ almost surely. This, along with \eqref{sss}, gives
\be
p(\yvec)=\lim_{n\rightarrow \infty} \frac{ n-r+k }{n(n-r+k)} \log Z_{(n-r+k)\yvec} \ge  \varlimsup_{n\rightarrow \infty} \frac{1}{n} \log Z_{\fl{n\yvec}}.
\ee
Monotonicity follows as in the case with integer coordinates.
 
To extend to all vectors $\yvec$ in $\bR^d_+$, define
\be
p(\yvec) = \sup\big\{  p(\xvec_m) : \yvec \ge \xvec_m \in \bQ^d  \big\}
\label{ext}
\ee 
With this definition, homogeneity follows immediately for rational $c$. For any $c>0$ pick $c_1$,$ c_2$ $\in \bQ$, so that $c_1 <c < c_2$. Then
\be
c_1p(\yvec)=p(c_1\yvec)\le p(c\yvec) \le  p(c_2\yvec) = c_2p(\yvec).  
\ee
Then consider sequences $c_1^{m} \nearrow c$ and $c_2^{m} \searrow c$ to get homogeneity for all $c\in \bR_+$.

Superadditivity follows directly from \eqref{ext} and concavity follows from homogeneity and superadditivity. This in turn implies continuity of $p(\cdot)$ in the interior of $\bR^d_+$. 

Finally, for the most general form of the limit for $\yvec \in \bR^d_+$, pick rational vectors $\uvec_1 < \yvec < \uvec_2$ so that for a given $\e > 0$, $\| \uvec_i - \yvec \|_1 < \e$. By superadditivity, 
\be
\log Z_{\fl{n\uvec_1}} + \log Z_{\fl{n\uvec_1}, \fl{n\yvec}} \le \log Z_{\fl{n\yvec}} \le \log Z_{\fl{n\uvec_2}} - \log Z_{\fl{n\yvec}, \fl{n\uvec_2}}.
\label{bineq0}
\ee

We first bound $\log Z_{\fl{n\yvec}, \fl{n\uvec_2}}$ from below using a path $\pi_n$  with segments on the boundary of the rectangle, parallel to the axes (following the ordering of the axes). As $n$ becomes large, the weights on $\pi_n$ will satisfy a strong law of large numbers,
\begin{align}
\varliminf_{n\to \infty}n^{-1}\log Z_{\fl{n\yvec}, \fl{n\uvec_2}} \ge \lim_{n\to \infty} n^{-1}\sum_{j= 1}^{\fl{n\|\uvec_2 - \yvec\|_1}}\beta \om(\xvec_j) \ge 0.
\label{lll0}
\end{align}
Symmetric bounds hold for $Z_{\fl{n\uvec_1}, \fl{n\yvec}}$. These, along with \eqref{bineq0} give
\begin{align}
p(\uvec_1) \le \varliminf_{n\rightarrow \infty} n^{-1}\log Z_{\fl{n\yvec}} 
                                \le \varlimsup_{n\to \infty} n^{-1}\log Z_{\fl{n\yvec}} \le p(\uvec_2). 
\label{llll0}
\end{align}
Let $\uvec_1$ and $\uvec_2$ to converge appropriately on rational vectors to $\yvec$, use continuity of $p(\cdot)$ to finish the proof and verify the limit for all points.

\medskip
\textit{Continuity at the boundary.}  Let ${\bf e}_{k,d} = (0,0, \cdots, 0, \e_{k+1}, \cdots, \e_d) \in \bR^d_{+}$,  and let $\yvec = (y_1,\cdots, y_{k}, 0,\cdots, 0)$ with $y_i \neq 0$ if $i\le k$. Define $\yvec_{{\bf e}} = \yvec+ {\bf e}_{k,d} \in \text{int}({ \bR^d_{+}})$. Then,

\begin{align}
p(\yvec_{{\bf e}}) & = \lim_{n\to \infty} n^{-1}\log Z_{\fl{n\yvec_{{\bf e}}}}\notag \\
                               & \ge \lim_{n \to \infty } n^{-1}\log Z_{\fl{n\yvec}} +  \lim_{n \to \infty }\sum_{i=k+1}^{d} n^{-1} \sum_{j=1}^{\fl{n\e_i}}\beta \om(\uvec_j)\notag \\
                               & = p(\yvec) + \bE(\om(\mathbf 0))\beta \sum_{i=k+1}^{d}\e_i \quad \bP -\text{ a.s.}\label{boundcontup} \end{align}
                               
For a reverse inequality, given  $\uvec\in\bR_+^d$ define the
following set of paths
on $\bZ^{d-1}$: 
\[
\Lambda_{\fl{n\uvec}} = \Big\{ \pi = \{ \vvec^{i} \}_{ i=0}^{\fl{nu_d}+1}\in (\bZ^{d-1})^{\fl{nu_d}+2}: \mathbf0= \vvec^{0} \le \vvec^{1} \le \cdots\le\vvec^{\fl{nu_d}+1}=
 \fl{n\uvec_{d-1}}\Big\}.
\]
Decompose  the partition function according to the lattice points where it enters
a new level in the $\mathbf e_d$ direction   (including an irrelevant weight
$\om(0)=0$ at the origin): 
\be
Z_{\fl{n\uvec}} = \sum_{\pi \in \Lambda_{\fl{n\uvec}} } \prod_{i=0}^{\fl{nu_d}} 
Z^{\square}_{(\vvec^{i}, i ), (\vvec^{i+1}, i )}.   \label{partdeco}
\ee
 For $\pi \in \Lambda_{\fl{nu}}$ let $Z_\pi$ denote a summand in \eqref{partdeco}.   
\be\begin{aligned}
\log Z_{\pi} &= \sum_{i=0}^{\fl{nu_d}} \log Z^{\square}_{(\vvec^{i}, i ), (\vvec^{i+1}, i )}  \\
                   &= \sum_{i=0}^{\fl{nu_d}} \log Z_{(\vvec^{i}, i ), (\vvec^{i+1}, i )} + \beta\sum_{i=1}^{\fl{nu_d}}\om(\vvec^{i}, i )  \\
                   &\le \log \widetilde{Z}^{\pi}_{\fl{n\uvec_{1,d-1}}} + \beta\sum_{i=1}^{\fl{nu_d}}\om(\vvec^{i}, i ) \\
                   &\le \log \widetilde{Z}^{\pi}_{\fl{n\uvec_{1,d-1}} + \fl{nu_d\mathbf{e}_{d-1}}} . 
\end{aligned}\label{coupling}\ee
   $ \widetilde{Z}^{\pi}_{\fl{n\uvec_{d-1}} + \fl{nu_d\mathbf{e}_{d-1}}}$ above 
    is a partition function in $\bZ^{d-1}$ with weights coupled with the original 
weights  
so that we have the identities  $Z_{(\vvec^{i}, i ), (\vvec^{i+1}, i )}=
\widetilde{Z}^{\pi}_{ \vvec^{i}, \vvec^{i+1}}$ 
and 
$ \beta\sum_{i=1}^{\fl{nu_d}}\om(\vvec^{i}, i) = \log \widetilde{Z}^{\pi}_{ \fl{n\uvec_{1,d-1}}, \fl{n\uvec_{1,d-1}} +\fl{nu_d}\mathbf{e}_{d-1}}$. 

Let $M = |\Lambda_{\fl{n\uvec}}|$.  Counting the number of ways in which 
 the length from $0$ to $\fl{nu_i}$ can be decomposed into $\fl{nu_d}+1$
 segments gives 
\be 
M = \prod_{1\le i \le d-1}{ {\fl{nu_i} +\fl{nu_d}} \choose {\fl{nu_d}+1}}\label{M}.
\ee

By Stirling's formula,
\begin{align}
M &= \exp\Bigg\{n \sum_{i=1}^{d-1} \Big( (u_i + u_d)\log(u_i + u_d) - u_i \log u_i - u_d\log u_d  \Big) + o(n) \Bigg\} \notag \\
    &= \exp\big\{ n L_{\uvec_{1,d-1}}(u_d) +o(n) \big\} \label{pathcard}.
\end{align}

Let $\e > 0$ and fixed. Assume that $u_d = \e_d >  0$, positive but sufficiently small (depending on the value of $\e$).   
If $ \log Z_{\fl{n\uvec}} \ge np(\uvec_{1,d-1}) + n\e$ 
there must exist a summand in \eqref{partdeco} with total weight no less than  $M^{-1}e^{n(p(\uvec_{1,d-1})+ \e)}$. 
Using this, \eqref{coupling} and the concentration inequality in Lemma \ref{321} [Proposition 3.2.1 (b) in \cite{Comets-Gregorio-arc}] 
we compute
\begin{align}
\bP\big\{\log Z_{\fl{n\uvec}}&\ge np(\uvec_{1,d-1}) + n\e\big\}\notag\\
&\le M \bP\big\{ \log \widetilde{Z}^{\pi}_{\fl{n\uvec_{1,d-1}} + \fl{nu_d\mathbf{e}_{d-1}}}\ge np(\uvec_{1,d-1}) + n\e -\log M\big\} \notag \\
&\le M \bP\big\{ \log \widetilde{Z}^{\pi}_{\fl{n\uvec_{1,d-1}} + \fl{nu_d\mathbf{e}_{d-1}}}\ge np(\uvec_{1,d-1}+u_d\mathbf{e}_{d-1}) + n\e/2 -\log M\big\} \notag \\
&\le 2\exp\{ - n(c_1\e^2 -  2L_{\uvec_{1,d-1}}(\e_d)) +o(n) \}.\label{B-Csetup}
\end{align} 

For $ \e_d $ sufficiently small, $L_{\uvec_{1,d-1}}(\e_d)$ is smaller than $\e^2$. A Borel-Cantelli argument gives that    
\be
\lim_{n\to \infty}n^{-1} \log Z_{\fl{n\uvec}} \le p(\uvec_{1,d-1}) + \e, \quad \bP-a.s. \label{Continuityupbound0}
\ee

To finish the proof, consider $\yvec_{\mathbf{e}}$ as before. Iterating the arguments from above, starting from \eqref{partdeco}, we conclude  that 
\begin{align}
\bP\big\{\log Z_{\fl{n\yvec_{\mathbf{e}}}} &\ge np(\yvec) + n\e\} \notag\\
&\le 2^{d-k}\exp\Bigg\{ - n\Big(c_1\e^2 -  \sum_{j=k+1}^{d}L_{\yvec_{1, j-1}}(\e_{j}) + f(\e_d, \cdots, \e_{k+1})\Big)+o(n) \Bigg\}
\notag
\end{align}
where $f(\e_d, \cdots, \e_{k+1}) \to 0$ uniformly as the $\e_i \to 0$. By Borel-Cantelli we have 
\be
\lim_{n\to \infty}n^{-1} \log Z_{\fl{n\yvec_{\mathbf{e}}}} \le p(\yvec) + \e, \quad \bP-a.s. \label{Continuityupbound1}.
\ee

Equations \eqref{boundcontup} and \eqref{Continuityupbound1} give the continuity at the boundary by letting $\e$ tend to $0$.


 
\medskip

\textit{Convergence in all directions.} We now show that convergence happens $\bP$- a.s.\ simultaneously for all endpoints $\yvec \in \bR_+^d$. Note that there is a full probability event of $\Omega$ where the convergence is true for all 
$\yvec \in \bQ^d$. 
By (2),  it suffices to show that the conclusion is true for $\yvec$
 with $\|\yvec\|_1 =1$.

Let $ \delta,\e \in \bQ_+$  and assume $0<\delta< \e$. We define two partitions:  
 \[\pi^+_M(\e, \delta) = \{\vvec_1, \vvec_2,\dotsc, \vvec_M \}\] is a partition of the hyperplane $\|\yvec\|_1 = 1+\e, \,\, \yvec \in \bR^d_+$, so that  $\min_{i\neq j}\|\vvec_i-\vvec_j\|_1 \le \delta$ with $\vvec \in \bQ^d$. 
Also we project $\pi^+$ to get
\[ \pi^{-}_M(\e, \delta) = \{\vvec_1- 2\e\mathbf e_d, \vvec_2-2\e \mathbf e_d,\dotsc, \vvec_M-2\e\mathbf e_d\}, \]
a partition of the hyperplane $\|\yvec\|_1 = 1-\e$. Any point that falls outside the first quadrant we remove from $\pi^{-}_M(\e, \delta).$
Observe that for any $\yvec$ with $\|\yvec\|_1 =1$, 
there exist partition points $\vvec_j \in \pi^+_M(\e, \delta)$, $\uvec_i = \vvec_i- 2\e e_d \in \pi^{-}_M(\e, \delta)$
 with $\|\vvec_j - \yvec\|_1+ \|\uvec_i - \yvec\|_1 \le 4\e$, with the vector ordering 
$
\uvec_i < \yvec < \vvec_j.
$
For any $n\in \bN$ and choice of vectors $\uvec_j, \vvec_i$ fix a path $\pi^n_{\uvec_i, \vvec_j}$ from $\fl{n\uvec_i}$ to $\fl{n\vvec_j}$. Restrict the space of environments further by assuming the following law of large numbers:
\be
\lim_{n\to \infty}n^{-1}\sum_{x_k \in \pi^n_{\uvec_i, \vvec_j}} |\om(x_k)| = \|\vvec_j - \uvec_i\|\bE|\om(\mathbf0)|
\label{pathlln}
\ee

Then, following the calculations as in \eqref{bineq0},
\be
\log Z_{\fl{n\uvec_j}} + \log Z_{\fl{n\uvec_j}, \fl{n\yvec}} \le \log Z_{\fl{n\yvec}} \le \log Z_{\fl{n\vvec_i}} - \log Z_{\fl{n\yvec}, \fl{n\vvec_i}}
\label{bineq}
\ee
gives
\begin{align}
p(\uvec_i) -C\e &\le \varliminf_{n\to \infty} n^{-1}\log Z_{\fl{n\yvec}} 
                                                                 \le \varlimsup_{n\to \infty} n^{-1}\log Z_{\fl{n\yvec}} \le p(\vvec_j) + C\e.
\label{llll}
\end{align}
Let $\e\to 0$ along rationals and $\uvec_i$ and $\vvec_j$ to converge appropriately on rational vectors to $\yvec$. Continuity of $p(\cdot)$ on rationals gives the conclusion.

To prove (c) observe that for any vector $\yvec=(y_1, y_2,\dotsc, y_d)$ and a permutation on $d$ elements $\sigma \in  Sym(d)$, the following equalities hold:
 \be p(\yvec) = p(y_1,y_2,\dotsc, y_d) = p(y_{\sigma(1)},y_{\sigma(2)},\dotsc, y_{\sigma (d)})=p(\sigma(\yvec))\notag
\ee 
Restricted on the hyperplane $\|\yvec\|_1=1$, $p(\yvec)$ is concave, symmetric about $d^{-1}\mathbf{1}$. This implies (c).
\end{proof}

\begin{proof}[Proof of Proposition \ref{???}.] Without loss of generality, let $t=1$. 
Using the definition \eqref{totparf} we immediately get the lower bound
\be
p(d^{-1}\mathbf{1}) \le \rho^{\dtot}(1).
\ee
For the upper bound, let $\e>0$. Use the same partition $\pi^+_M(\e)$ as in the proof of the previous proposition and the last part of \eqref{bineq}, to estimate
\begin{align}
n^{-1}\log Z^{\dtot}_n&=n^{-1} \log \Big(\sum_{\yvec: \| \yvec \|_1=1} \log Z_{\fl{n\yvec}} \Big) \notag\\
                                        &\le n^{-1} \log \Big(n^d\max_{\yvec: \| \yvec \|_1=1} \log Z_{\fl{n\yvec}} \Big) \notag\\
                                        &\le n^{-1}d\log n +  n^{-1}\max_{1\le i \le M}\log Z_{\fl{n\vvec_i}} + n^{-1} \max_{\yvec \in Q_i}\sum_{j= 1}^{\fl{n\|\vvec_i - \yvec\|_1}}\beta|\om(u_j)|, \label{aarrgg}
\end{align}
where $Q_{i}=\{ \yvec: \|\yvec - \vvec_i\|_1 \le \e\}$. 
Take $n\rightarrow \infty$ in \eqref{aarrgg} to conclude that 
\be
\rho^{\dtot}(1) \le \sup_{\|\uvec\|_1=1+\e} p(\uvec) + C\e \le  p(d^{-1}1+\e, \cdots, 1+\e) + C\e,
\ee
where $C$ is the limiting last passage percolation constant if the path takes $2n\e$ steps. 
Let $\e$ tend to $0$ to get the conclusion.
\end{proof}

\begin{corollary}
There exists a full $\bP$-probability event so that 
\be
n^{-1}\log Z_{\fl{n\uvec}, \fl{n\vvec}} = p(\vvec - \uvec)
\ee 
simultaneously for all $(\uvec, \vvec) \in \bR^d_+\times \bR^d_+$ with $\uvec \le \vvec$
\label{lasterror}
\end{corollary}

\begin{proof}
From the previous proposition, the result is true if $\uvec= \mathbf0$. Then we can restrict further to all rational vectors $\uvec \in \bQ^d_+$ and use simialr approximations as before to get the corollary. 
\end{proof}

The remaining part of this section is about large deviations results.  


\begin{lemma}
Suppose that for each $n$, $L_n$ and $Z_n$ are independent random variables. Assume that the limits 
\begin{align}
{\lambda}(s) &= -\lim_{n\rightarrow \infty}n^{-1}\log \bP\{ L_n \ge ns \}, \label{cradev}\\
\phi(s) &= -\lim_{n\rightarrow \infty}n^{-1}\log \bP\{ Z_n \ge ns \}
\label{rightdev}
\end{align}
exist and are finite   for all $s\in \bR$. Assume   that $\lambda(a_\lambda)=\phi(a_\phi)=0$ for some 
  $a_\lambda, a_{\phi}\in \bR$.   Assume also that $\lambda$ is continuous.  
 Then for $r\in\bR$
\begin{align}
\lim_{n\rightarrow \infty} n^{-1}\log \bP\{ L_n + Z_n \ge nr \} =\begin{cases}
-\displaystyle \inf _{a_{\lambda}\le s \le r-a_{\phi}}\{ \phi(r-s)+  {\lambda}(s) \}, \,& r > a_{\phi}+a_{\lambda} \\
0, \,& r \le a_{\phi}+a_{\lambda}.
\label{sumsind0}
\end{cases}
\end{align}

\label{sumsind}
\end{lemma} 

\begin{proof}  
The lower bound $\ge$  follows from 
\be
\bP\{ L_n + Z_n \ge nr \} \ge  \bP\{ L_n \ge ns \} \bP\{ Z_n \ge n(r-s) \}. 
\nn\ee
Since an upper bound $0$ is obvious, it remains to show the upper bound
for the case $r > a_{\phi}+a_{\lambda}$.  
Take a finite  partition  $ a_{\lambda}= q_0 < \dotsm < q_m =  r -a_{\phi} $.
Then use a union bound and independence:  
\begin{align}
&\bP\{ L_n + Z_n \ge nr \}  \notag \\[3pt]
                &\le  \bP\{ L_n + Z_n \ge nr , \, L_n < n q_0 \}   \notag \\ 
                 &\quad \quad \quad \quad+\sum_{i=0}^{m-1}\bP\{ L_n + Z_n \ge nr, \,nq_i \le L_n \le nq_{i+1}\}  +\bP\{  L_n \ge nq_m\}\notag\\
&\le  \bP\{ Z_n  \ge  n (r-q_0) \} + \sum_{i=0}^{m-1}\bP\{ Z_n \ge n(r -q_{i+1})\}
\bP\{   L_n \ge nq_{i}\}+\bP\{ L_n \ge nq_m\}.\label{dec}\nn
\end{align}
 From this 
\be\begin{aligned}
&\varlimsup_{ n \rightarrow \infty } n^{-1} \log \bP\{  L_n + Z_n \ge nr \}\\
&\qquad \le - \min\Big\{ \phi(r-q_0), \, \min_{0\le i \le m-1}[\phi(r-q_{i+1})+ {\lambda}(q_{i})],\, \lambda(q_m)\Big\}. 
\end{aligned}\label{oops}\nn
\ee
Note that $\lambda(q_0)=\phi(r-q_m)=0$, refine the partition  and use the 
  continuity of $\lambda$.   \end{proof}

\section{Existence of the rate functions}
\label{Ratefuncex}
First a general lemma about existence of limits of almost superadditive sequences.
\begin{lemma}
Let $\{a_n\}_{n\in \bN}$ a sequence such that $a_{n+m}\ge a_{n} + a_{m}+c_{n,m}$ where $|c_{n,m}| < B$. Then the limit 
$\lim_{n\to \infty} n^{-1} a_n$ exists (and is potentially infinite). 
\end{lemma}

\begin{proof}
Let $\gamma = \varlimsup_{n\to \infty}n^{-1}a_n$ and assume first that $\gamma< \infty$. Specify an $\e>0$ and let $N_0 \in \bN$ such that $N_0^{-1}a_{N_0} > \gamma -\e$ and $N_0^{-1} B<\e$. We can write any $n\in \bN$ as $n = k N_0 + r$. Then
\begin{align}
n^{-1}a_n &\ge n^{-1}a_{kN_0} + n^{-1}a_r - n^{-1}B\notag\\
                  &\ge (n-r)n^{-1}( N_0^{-1}a_{N_0} -  N_0^{-1}B)+ n^{-1}a_r - n^{-1}B\notag\\
                  &\ge (n-r)n^{-1}( \gamma -  2\e)+ n^{-1}a_r - n^{-1}B.\notag
\end{align}
Take $\varliminf_{n\to \infty}$ on both sides and let $\e\to 0$.
If $\gamma = \infty$, pick any large constant $C$ and let $N_0$ large so that  $N_0^{-1}a_{N_0} > C$. Identical steps show that the limit is infinity.
\end{proof}
\subsection{The fixed endpoint case }
 
\begin{proof}[Proof of Theorem \ref{phidef}- Existence] 
For $m,n\in\bR_+$ let $\xvec_{m,n} \in \{0,1\}^{d}$ so that
$\fl{(m+n)\uvec}= \fl{m\uvec}+\fl{n\uvec} + \xvec_{m,n}.$
By superadditivity, independence and shift invariance
\be\begin{aligned}
\bP\{ \log &Z_{\fl{(m+n)\uvec}} \ge {(m+n)r}\}   \\
                                      & \ge  \bP\{ \log Z_{\fl{m\uvec}}\ge mr\} \bP\{ \log Z_{\fl{n\uvec}} \ge nr\}
    \bP\{ \log Z_{\xvec_{m,n}} \ge 0\}.   
\end{aligned}\label{temp18}\ee
By assumption  \eqref{ndeg}  there is a uniform lower bound
$\bP\{ \log Z_{\xvec_{m,n}} \ge 0\}\ge \rho>0$.  
  Thus  $t(n)= \log \bP\{  \log Z_{\fl{n\uvec}}\ge nr \} $ 
is  superadditive with a small uniformly bounded correction.  
Similar reasoning shows that either $t(n)=-\infty$ for all $n$ or then
$t(n)>-\infty$ for all $n\ge n_0$.  
Consequently by superadditivity  the rate function
\begin{align}
 \rf_{\uvec}(r) &= -\lim_{n\rightarrow \infty} n^{-1}\log \bP\{  \log Z_{\fl{n\uvec}} \ge nr\} \label{Fekete}
               \end{align}
exists for $\uvec = (u_1,\dotsc,u_d) \in \bR_+^d$  and $r\in \bR$. The limit in \eqref{Fekete} holds also as $n\rightarrow \infty$ through   real values, not just integers.

Similarly we get  convexity of $ \rf$ 
in $(\uvec, r)$.  Let $\lambda \in (0,1)$ and assume $(\uvec, r) =\lambda(\uvec_1,  r_1)+(1-\lambda)(\uvec_2,  r_2).$ Then 
\begin{align*}
n^{-1}\log\bP\{ \log Z_{\fl{n\uvec}} &\ge nr \}    
                                     \ge  \lambda (\lambda n)^{-1} \log \bP\{  \log Z_{\fl{n\lambda \uvec_1}} \ge n\lambda r _1\} \\[3pt]
                                     &+ (1-\lambda) ((1-\lambda) n)^{-1} 
                                        \log \bP\{  \log Z_{\fl{n(1-\lambda) \uvec_2}}  \ge n(1-\lambda) r _2\}+o(1) \notag 
\end{align*} 
and letting  $n \rightarrow \infty$ gives 
\be
 \rf_{\uvec}(r) \le \lambda  \rf_{\uvec_1}(r_1) + (1- \lambda)  \rf_{\uvec_2}(r_2).
\ee 

Similar arguments give existence and convexity of the rate function
\be
\rf_t(r) = -\lim_{n\rightarrow \infty}\log \bP\{ \log Z^{\dtot}_{\fl{nt}}\ge r \}.
\ee
\end{proof}
\section{Behavior of the rate functions}

We first need to show that the rate function, whose existence was established in the previous section, is not trivial. We establish nontrivial bounds and the continuity of the rate function on the boundary.

\begin{proof}[Proof of Proposition \ref{bdc}.]
 Let 
$\uvec=(u_1,u_2,\dotsc, u_d)$ and let $\uvec_{1,k}$ $= (u_1,u_2,\dotsc, u_k,0,\ldots,0) $
where $k\in\{ 1,2,\dotsc, d-1\}$, $\uvec$ $\in \bR^d_+$. We start with an upper bound for $\rf$. Fix $r \in \bR$. Superadditivity of the partition functions gives 
\begin{align}
 \rf_{\uvec}(r) &\le -\lim_{n\rightarrow\infty}n^{-1}\log \bP\Bigg\{  \log Z_{\fl{n\uvec_{1,k}}} + \beta\sum_{i=k+1}^{d}\sum_{j=1}^{\fl{nu_i}}\om(x_{i,j}) \ge nr \Bigg\} \notag \\
                     &\le \rf_{\uvec_{1,k}}(r) -\sum_{i=k+1}^d \lim_{n\rightarrow \infty}n^{-1} \log \bP\Bigg\{\fl{nu_i}^{-1} \sum_{j=1}^{\fl{nu_i}}\om(x_{i,j}) \ge 0 \Bigg\} \label{ustep1} \\ 
                     &=\rf_{\uvec_{1,k}}(r) + \sum_{i=k+1}^d u_i \inf_{x>0}I(x) \label{ustep11},
\end{align}
since the sum inside the braces of \eqref{ustep1} is a sum of i.i.d.\ $\om(x_j)$ r.v.\ for which the Cram{\'e}r rate function exists by Assumption \ref{Cramerass}. 
This bound holds as long as $u_1 \neq 0$.

For the lower bound, we use the same decomposition as in the continuity of the limiting free energy. Recall \eqref{pathcard} and define  for $2\le i \le d$ 
\be
F_{i-1}(\uvec) =  \sum_{j=1}^{i-1} \Big( (u_j + u_i)\log(u_j + u_i) - u_j \log u_j - u_i\log u_i\Big).
\label{patheses}
\ee

With this notation, $M$ in \eqref{pathcard} is given, using a Stirling approximation by $M= \exp\{nF_{d-1}(\uvec) + o(n)\}$ $\le \exp\{nF_{d-1}(\uvec) + n\e_{d-1}\} $. The error $\epsilon_{d-1}$ can be as small as possible for $n$ sufficiently large. The functions $F_{i-1}$ are jointly continuous in $u_1, \dots, u_{i-1}$ and as long as $u_{i-1}$ (and possibly more of the $u_j$'s)  tend to $0$, so is the function. 
 
We use this to bound $\rf$ from below with a standard union bound:
\begin{align}
& \rf_{\uvec}(r) \ge- \lim_{n\rightarrow \infty}n^{-1}
\log \sum_{\pi \in \Lambda_{\fl{n\uvec}}} \bP\bigl\{ \log Z_{\pi} \ge nr - \log M  \bigr\}\notag \\
&\stackrel{\eqref{coupling}}{\ge}- \lim_{n\rightarrow \infty}n^{-1}\log \sum_{\pi \in \Lambda_{\fl{n\uvec}} }
 \bP\bigg\{ \log \widetilde{Z}^{\pi}_{\fl{n\uvec_{1,d-1}} + \fl{nu_d\mathbf{e}_{d-1}}} \ge nr - \log M  \bigg\}\notag \\
&\ge- \lim_{n\rightarrow \infty}\biggl( \frac{\log M}n + n^{-1}\log
 \bP\bigg\{ \log  \widetilde{Z}^{\pi}_{\fl{n\uvec_{1,d-1}} + \fl{nu_d\mathbf{e}_{d-1}}} \ge nr - nF_{d-1}(\uvec) -n\epsilon_{d-1}  \bigg\} \biggr)\notag \\
 &=\rf_{\uvec_{1,d-1}+ u_d\mathbf{e}_{d-1}}(r-F_{d-1}(\uvec)-\epsilon_{d-1}) - F_{d-1}(\uvec). \notag
\end{align}
In the last step above a little correction as in \eqref{temp18} replaces 
$\fl{n\uvec_{1,d-1}} + \fl{nu_d\mathbf{e}_{d-1}}$ with 
$\fl{n\uvec_{1,d-1} + nu_d\mathbf{e}_{d-1}}$. 

Set \be
\widetilde{\uvec}_{1, i-1}= \uvec_{1, i-1}+\sum_{k=i}^{ d} u_k\mathbf{e}_{i-1}.
\ee
Proceeding inductively, we get the lower bound
\be
 \rf_{\uvec}(r) \ge \rf_{\widetilde{\uvec}_{1,k}}\Big(r-\sum_{k+1\le i \le d}\big(F_{i-1}(\widetilde{\uvec}_{1,i})-\epsilon_{i-1}\big)\Big) -\sum_{k+1\le i \le d}F_{i-1}(\widetilde{\uvec}_{1,i}), \label{Psi lower bound}
\ee
where the $\epsilon_i$'s are quantities that come from the errors from repeated use of Stirling's formula and go uniformly to $0$. 
Joint continuity of the function $F(\uvec)$ translates to joint continuity of the lower bound in \eqref{Psi lower bound}.

Equations \eqref{ustep11} and \eqref{Psi lower bound} suffice to give continuity of the rate functions on the boundary of $\bR^d_+$. Let $1\le k \le d$ and assume without loss of generality that the coordinates that converge to $0$ are the last $d-k$, so that \[\uvec^{(m)}\stackrel{m\to \infty}{\longrightarrow} (u_1, u_2,\hdots, u_k, 0,0,\hdots,0 ) \in \bR^d_+,\] with $u_i > 0$. Then both bounds  \eqref{ustep11} and \eqref{Psi lower bound} come together and give
\be
\lim_{m\to \infty}\rf_{\uvec^{(m)}}(r) = \rf_{(u_1,\hdots, u_k,0,\ldots, 0)}(r).
\ee
When \eqref{Psi lower bound} is used, it is under the assumption that $u_1, \ldots, u_k >0$. Then, the rate function $\rf_{\widetilde{\uvec}_{1,k}}(r)$ is continuous in $(\widetilde{\uvec}_{1,k}, r)$ by Theorem \ref{phidef} applied when we are restricted on the $k$-dimensional facet of $\bR^d_+$ and the finiteness assumption of the rate function. 
 
Assume now that $\rf_{\mathbf0}(r) < \infty$ and $\uvec^{(m)}\to \mathbf0 $, $\uvec^{(m)} \in \inter\bR^d_+$.  For an upper bound at the origin, one can repeat the steps \eqref{ustep1},\eqref{ustep11}. The partition function now is just a sum of i.i.d.~ $\om$ variables. Then,
\be
\rf_{\uvec^{(m)}}(r)\le u^{(m)}_1\rf_{\mathbf e_1}(r/u^{(m)}_1) + \sum_{i=2}^d u^{(m)}_i \inf_{x>0}I(x). \label{ustep111} 
\ee 

This, along with \eqref{Psi lower bound} give that for any $\delta > 0$
\begin{align}
\varliminf_{m\to \infty}u^{(m)}_1 \rf_{\mathbf e_1}((r - \delta)&/u^{(m)}_1) -\delta\le \varliminf_{m \to \infty}\rf_{\uvec^{(m)}}(r) \notag \\
&\le\varlimsup_{m\to \infty} \rf_{\uvec^{(m)}}(r)\le\varlimsup_{m\to \infty} u^{(m)}_1\rf_{\mathbf e_1}(r/u^{(m)}_1) + \delta.
\label{0conv}
\end{align}
Let $r_0$ such that $\rf(r_0)=0$. Concavity of $\rf_{\mathbf e_1}$ gives a further upper bound $x_{\infty}(r- u_1r_0) \to x_{\infty}r$. For a lower bound, take a sequence $t_n \to x_{\infty}$. Then, for any $t_n$, \[u_1\rf_{\mathbf e_1}((r-\delta)/u_1) > t_n(r-\delta)-u_1\log \bE(e^{t_n\omega(\mathbf0)}).\]
Take $\varliminf_{u_1 \to 0}$ and then $t_n \to x_{\infty}$.
\end{proof} 

\begin{proof}[Proof of Proposition \ref{u0}]
We prove the second one. Let $\e>0$ and $N$ large enough, so that $\bE n^{-1} \log Z_{\fl{n\uvec}} < p(\uvec) + \e$ for $n>N$. Then, 
\begin{align*}
2 \exp\{ -c_2(\uvec)\e^2 n\} &\ge \bP\{ |\log Z_{\fl{n\uvec}} - \bE \log Z_{\fl{n\uvec}} |>n\e \}\\
&\ge \bP\{ \log Z_{\fl{n\uvec}} - \bE \log Z_{\fl{n\uvec}} >n\e \}\\
&\ge \bP\{ \log Z_{\fl{n\uvec}} >n p(\uvec) + 2n\e\}.
\end{align*}
 This suffices for the result.
\end{proof}

\begin{proof}[Proof of Proposition \ref{lvbehavior}.]
Let $m>0$ and $n>0$. Let $\beta > 0, r\in \bR$ and let $u \in \bR^d_+$ and for simplicity assume $\|\uvec\|_1=1$. Observe
\begin{align}
\log Z^{\beta}_{\fl{n\uvec}}
&\le \log \Big(|\Pi(\fl{n\uvec})|\max_{x_{\centerdot} \in \Pi(n\uvec)}\Big(\prod_{j=1}^{\|\fl{n\uvec}\|_1}e^{\om(x_j)}\Big)^{\beta}  \Big) \notag\\
&\le nC + \beta T(\fl{n\uvec})\notag
\end{align}
Hence
\be
\beta T(\fl{n\uvec}) \le \log Z^{\beta}_{\fl{n\uvec}} \le  nC + \beta T(\fl{n\uvec}).
\label{Tbounds}
\ee
Since  $I^{\infty}_{\uvec}$ is left continuous at $r$, we can estimate from below
\begin{align}
\lim_{\beta\to \infty}J^{\beta}_{\uvec}(\beta r)  &= - \lim_{\beta \to \infty}\lim_{n \to \infty} n^{-1} \log \bP\{  \log Z^{\beta}_{\fl{n\uvec}} \ge n\beta r \}  \notag \\
&\ge - \lim_{\beta\to \infty}\lim_{n \to \infty} n^{-1} \log \bP\{   nC + \beta T(\fl{n\uvec}) \ge n\beta r \}  \notag \\
&= \lim_{\beta\to \infty}I^{\infty}_{\uvec} (r-C\beta^{-1}) \notag \\
&= I^{\infty}_{\uvec} (r) \notag
\end{align}
The reverse inequality also follows with the same estimate and the left hand side of  \eqref{Tbounds}.
\end{proof}

\section{Quenched Large deviations for the polymer path and endpoint}
\label{pathLDP}

For this section, we restrict on the $\bP$ - full probability event 
\[  \Omega_0  = \Big\{\om:\lim_{n\to \infty} n^{-1}\log Z_{\fl{n\uvec}, \fl{n\vvec}}\} = p(\vvec-\uvec), \forall (\uvec, \vvec)  \in \bR^d_+ \times\bR^d_+ , \uvec\le \vvec\Big\}\] 
whose existence is given by Corollary \ref{lasterror}. We make no special mention of this fact in the proofs that follow.

\begin{proof}[Proof of Theorem \ref{qend}]
 Let $\|\uvec\|_1 = 1$. 
We compute
\begin{align}
- n^{-1}\log Q^{\om}_n\{ x_{n} = [n\uvec]\} 
&= - n^{-1}\log \frac{\sum_{x_{\centerdot} \in \Pi([n\uvec])}\prod_{j=1}^n e^{\beta \om(x_j)}}{\sum_{x_{0,n} \in \Pi_{\dtot}(n)}\prod_{j=1}^n e^{\beta \om(x_j)}}\notag\\
&= - n^{-1}\log \frac{Z_{[n\uvec]}}{Z^{\dtot}_n}\notag\\
&\stackrel{n\to \infty}{\longrightarrow}p(d^{-1}\mathbf{1}) - p(\uvec).
\end{align} \qedhere
\end{proof}

\begin{proof}[Proof of Theorem \ref{qpath}] We only treat the free endpoint case. The constrained endpoint result follows by similar arguments.

Fix $L \in \bR_+$ and let $\gamma:  [0,1] \to \bR^d_+$ be a curve such that  
 each coordinate $\gamma_j(t)$, $1\le j \le d$,  is non-decreasing and $L$-Lipschitz. Since $\gamma$ is Lipschitz, it has a derivative almost everywhere. For the upper bound, this is the only fact of $\gamma$ that we are going to use.

Pick $\e>0$ and $\cN_{\e}(\gamma)$ an $\e -$neighborhood of $\gamma$ in the $\|\cdot\|_1$ norm. (For the definition of this neighborhood, consider $\gamma$ as a set in $\bR^2$.) For this choice of $\e>0$,  let $M$ sufficiently large to define a partition of the time interval $[0,1]$ 
\[ \pi_M = \{ 0 = t_1 < t_2 < \cdots < t_M =1:  \gamma_j(t_{i+1}) - \gamma_j(t_{i}) < \e/2d, \,\, 1\le j \le d\}, \]
and assume without loss that for all partition points $t_i$, $\gamma'(t_i)$ exists.



Abbreviate $\gamma(t_i) = \uvec_i$ and define the rectangles 
\begin{align}
R(\fl{n\uvec_i}, \fl{n\uvec_{i+1}}) &= \{ \vvec \in \bZ^d : \fl{n\uvec_i} \le \vvec \le \fl{n\uvec_{i+1}},\,  \notag\\ 
&\phantom{xxxxxxxxxxx} 0\neq \|\vvec - \fl{n\uvec_i}\|_{\infty} \le \| \fl{n\uvec_{i+1}} - \fl{n\uvec_i}\|_{\infty}\}.\notag
\end{align}
The definition of $\pi_M$ implies that the disjoint rectangles $ R(\fl{\uvec_i}, \fl{\uvec_{i+1}})\subseteq n\cN_{\e}(\gamma)$. Microscopically, for any path $x_{0,n}$ define by $x^{i}_{0,n}$ to be the piece of the original path that lies in $ R(\fl {n\uvec_i}, \fl{n\uvec_{i+1}})$. To get an estimate for the quenched rate function from above, 
\begin{align}
Q^{\omega}_n\{ x_{0,n} \in n\cN_{\e}(\gamma)\} &\ge Q_n^{\omega}\{ x_{0,n} \text{ goes through } \fl{n\uvec_1}, \fl{n\uvec_2}, \cdots \fl{n\uvec_M}.\} \notag \\
&= Q_n^{\omega} \{x_{0,n}^{i} \in  R(\fl{n\uvec_i}, \fl{n\uvec_{i+1}}), x_{0,n}^i \text{ starts at } \fl{n\uvec_i},\, 1\le i \le M\} \notag \\
&= \frac{\prod_{i = 1}^{M-1}Z_{\fl{n\uvec_i}, \fl{n\uvec_{i+1}}}}{Z^{\dtot}_n} \label{pth1}
\end{align}
Take logarithms in \eqref{pth1}, divide by $-n$ and let $n\to \infty$ to conclude that 
\begin{align}
-\varliminf_{n\rightarrow \infty} n^{-1}\log Q_n^{\omega}\{ x_{0,n} \in n\cN_{\e}(\gamma)\} &\le p(d^{-1}\mathbf1) -\sum_{i=1}^{M-1} p(\gamma(t_{i+1}) - \gamma(t_{i})) \notag\\
&= p(d^{-1}\mathbf1) -\sum_{i=1}^{M-1} p\Big(\int_{t_i}^{t_{i+1}}\gamma'(s)\,ds\Big) \notag\\
&\le p(d^{-1}\mathbf1) -\sum_{i=1}^{M-1} \int_{t_i}^{t_{i+1}}p(\gamma'(s))\,ds \label{jen},
\end{align}
where \eqref{jen} is the result of Jensen's inequality applied on the concave function $p$ and of the fact that $p$ is 1-homogeneous, by Proposition \eqref{substd}. Equation \eqref{jen} gives the upper bound.

\begin{figure}
\begin{center}

\includegraphics[trim = 100 450 100 100, scale=0.8 ]{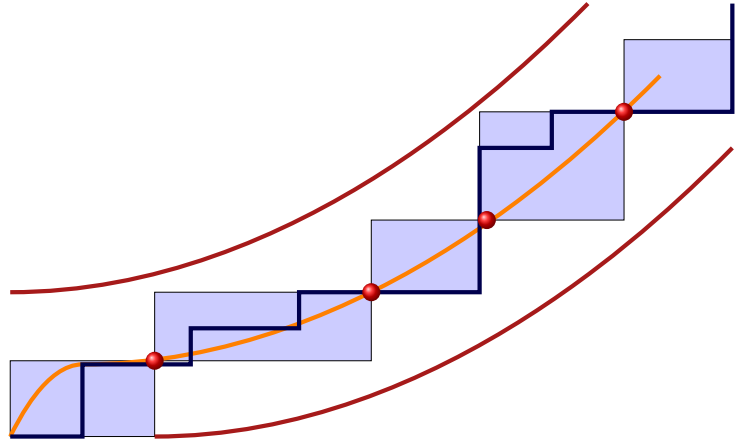}
\caption[Illustration of the upper bound in Theorem \ref{qpath}]{The idea of the proof of the upper bound in Theorem \ref{qpath}. The rectangles are defined using the the partition points (the circles). They are disjoint and inside the $\e-$ neighborhood, so any polymer chain inside those rectangles is inside the neighborhood.}
\label{QuenchedLB}
\end{center}
\end{figure}

\begin{figure}
\begin{center}
\includegraphics[trim = 200 460 200 200, scale=0.8 ]{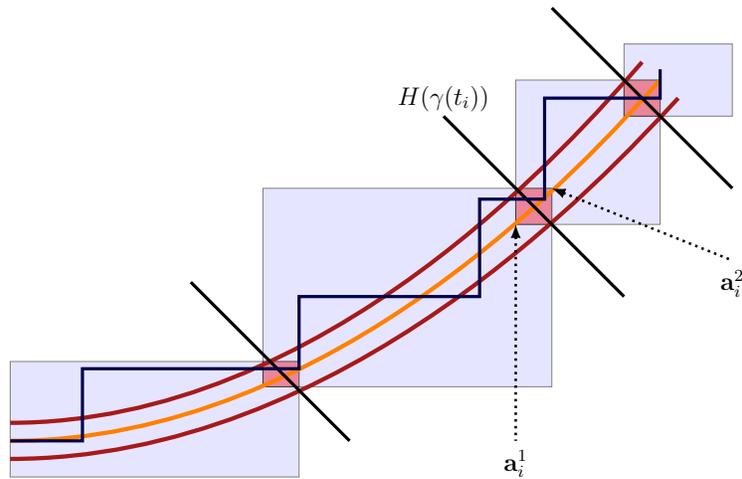}
\caption[ Illustration of the lower bound in Theorem \ref{qpath}]{ The idea of the proof of the lower bound in Theorem \ref{qpath}. The diagonal lines are the $\|\gamma(t_i)\|$ hyperplanes. The extra error (compared to proof of the lower bound) comes from the common small rectangles with the hyperplanes as diagonals. That error can become small if the $\e$-neighborhood is sufficiently narrow. }
\label{QuenchedUB}
\end{center}
\end{figure}

The remaining proof is about the lower bound. A word on the arrangement of the proof: The first part of what follows is used to specify an $\e_0 > 0$ such that certain macroscopic uniformity and continuity conditions are satisfied (here we use the fact that $\gamma$ is Lipschitz).  After $\e_0$ is specified we define a neighborhood $\cN_{\e}$ for $\e< \e_0$ and work microscopically to bound the quenched probabilities. That part of the proof works for any $\e<\e_0$. The need for the conditions that specify $\e_0$ become apparent after taking the limits (calculations \eqref{RS5}-\eqref{RS7}). 

Fix $L \in \bR_+$ and let $\gamma:  [0,1] \to \bR^d_+$ be a curve such that  
 each coordinate $\gamma_j(t)$, $1\le j \le d$,  is non-decreasing and $L$-Lipschitz.  An immediate consequence is that $0\le \gamma_j'(t)\le L$ for all $t$ where the derivative is defined. Since $\gamma$ is Lipschitz, the difference quotients are bounded
\be
\Big|\frac{\gamma_j(x) - \gamma_j(y)}{x- y}\Big| \le L \quad x,y \in [0,1],\,\, 1\le j \le d.
\ee

Let $\e'>0$ and restrict the limiting point-to-point free energy $p(\cdot)$ on the set $A_{L}=\{ \|\uvec\|_1 \le d L\}$. The function $p$ is uniformly continuous on $A_L$, so it admits a modulus of continuity and we can specify $\delta'>0$ so that
\be
\sup_{\substack{\|\uvec - \vvec\|_1 < \delta' \\ \uvec, \vvec \in A_L}}| p(\uvec) - p(\vvec)| < \e' .
\label{modulus1}
\ee
For a given $\e'>0$, denote  by $\Omega(\e')$ the supremum of $\delta'>0$ such that \eqref{modulus1} holds. As soon as $\delta'$ is specified, define a partition
\[ \lambda_M(\e') = \{0= t_1 <t_2 < \cdots < t_M =1\} \]
so that for all partition points $t_i$, the derivative $\gamma'(t_i)$ exists, and so that for all $1\le i \le M$ and $1 \le j \le d $,
\be
\Big|\frac{\gamma_j(t_{i+1}) - \gamma_j(t_{i})}{t_{i+1}-t_i} - \gamma_j'(t_i)\Big| < \frac{\delta'}{2d}.
\label{modulus2}
\ee

Pick $\e>0$ and $\cN_{\e}(\gamma)$ an $\e -$neighborhood of $\gamma$ in the $\|\cdot\|_1$ norm. Assume $\e$ is small enough so that 
\be 
\e< \e_0 =\min\Big\{ \Omega(\e'/2M), \,\e'/2M, \,\min_{1\le i \le M-1}\{ \| \gamma(t_{i+1})-\gamma(t_i)\|_1/ M \}\Big\}. \label{e-size}
\ee

For any vector $\uvec \in \bR^d_+$ define the hyperplane $ H(\uvec) = \{\xvec \in \bR^d_{+} : \| \xvec \|_1 = \| \uvec \|_1 \}$ and the positive half-space 
$H_{+}(\uvec) = \{\xvec \in \bR^d_{+} : \| \xvec \|_1 \ge  \| \uvec \|_1 \}$. 
We also denote  $H_{-}(\uvec) =  H_{+}^c(\uvec)$.

Observe that there exist vectors $\{\avec^j_{i}\}_{1\le i \le M-1}^{ j=1,2}$ such that for each $i,j$ the following hold:
\begin{enumerate}
\item $\avec_i^1 \le \gamma(t_i) \le \avec_i^2, \quad 1\le i \le M$,
\item $\| \gamma(t_i) - \avec_i^j \|_1 = \e, \quad j=1,2$,
\item $ R(\fl{n\avec_i^1 }, \fl{n\avec_{i+1}^2}) \supseteq  n\cN_{\e}(\gamma{[t_i, t_{i+1} )})\cap H_+(\fl{n\gamma(t_i)}) \cap H_{-}(\fl{n\gamma(t_{i+1})})=\cH^n_i.$
\end{enumerate}

The set $\cH^n_i$ is the set between two consecutive hyperplanes and in $n\cN_\e$ (see Figure \ref{QuenchedUB}). Denote the partition function $\cH^n_i$ by 
\[ Z_{\cH_i^n} = \sum_{x^{(i)} \in\cH_i^n } \prod_{j= \fl{n\gamma(t_i)}+1}^{\fl{n\gamma(t_{i+1})}}e^{\beta \om(x^{(i)}_j)},\]
where $x^{(i)}$ is any up-right path living in $\cH^n_i$.
 
Define $\Pi_{i, n\e}^{1} = \{ x_{0, n\e} : x_0 = \fl{n\avec_i^1}\}$, and $\Pi_{i, n\e}^{2} = \{ x_{0, n\e} : x_0 \in  H(\fl{n\gamma(t_i)}), x_{n\e} = \fl{n\avec_i^2}\}$, the set of all path of length $n\e$ that start from  $\fl{n\avec_{i}^j}$ and the set of paths that start somewhere on the hyperplane $H(\fl{n\gamma(t_i)})$ and end after $\fl{n\e}$ + O(1) steps at $ \fl{n\avec_i^2}$, respectively. The error comes from the integer parts, but is eventually immaterial so we ignore it for convenience.

Observe that any polymer chain that starts at $\fl{n\avec^1_i}$ and ends at $\fl{n\avec^2_{i+1}}$ has to cross the hyperplanes
$H(\fl{n\gamma(t_i)})$ and $H(\fl{n\gamma(t_{i+1})})$ at some points $\alpha$ and $\beta$ respectively.

The three conditions, along with \eqref{e-size} guarantee that $\bigcup _{1\le i \le M-1} R(\fl{n\avec_i^1}, \fl{n\avec_{i+1}^2}) \supseteq  n\cN_{\e}(\gamma)$ while the common  $R(\fl{n\avec_i^1}, \fl{n\avec_i^2})$ (the red-shaded rectangles in Figure \ref{QuenchedUB}) are pairwise disjoint and ``small'', in the sense of the following bound:

\begin{align}
 Z&_{\fl{n\avec_i^1}, \fl{n\avec_{i+1}^2}}
= \sum_{(\alpha, \beta) \in H(\fl{n\gamma(t_i)}) \times H(\fl{n\gamma(t_{i+1})})}Z_{\fl{n\avec_i^1},\alpha}Z_{\alpha,\beta} Z_{\beta, \fl{n\avec_{i+1}^2}} \notag \\
&\ge \sum_{(\alpha, \beta) \in H(\fl{n\gamma(t_i)}) \times H(\fl{n\gamma(t_{i+1})})}Z_{\alpha,\beta}\min_{x_{0,n\e} \in \Pi_{i, n\e}^{1}}\prod_{j=1}^{\fl{n\e}} e^{\beta \om(x_j)} \min_{x_{0, n\e} \in \Pi_{i+1, n\e}^{2}}\prod_{j=1}^{\fl{n\e}} e^{\beta \om(x_j)} \notag \\
& \ge Z_{\cH^n_i}\min_{x_{0, n\e} \in \Pi_{i, n\e}^{1}}\prod_{j=1}^{\fl{n\e}} e^{\beta \om(x_j)}\min_{x_{0,n\e} \in \Pi_{i+1, n\e}^{2}}\prod_{j=1}^{\fl{n\e}} e^{\beta \om(x_j)}\label{5am}.
\end{align}
Hence we can bound the probabilities
\begin{align}
Q^{\om}_n\{ x_{0, n} &\in n\cN_{\e}(\gamma)\}= Q_n^{\om}\{ x_{0,n} \in \cup_{i=1}^{M-1}\cH^n_i \} \notag\\
 &= (Z_{n})^{-1}{\sum_{x_{0,n} \in \cN_{\e}(\gamma)}\prod_{j=1}^{n} e^{\beta \om(x_j)}}\notag\\
 &\le (Z_{n})^{-1}\prod_{i=1}^{M-1}  Z_{\cH^n_i}\notag\\
 &\le (Z_{n})^{-1}\prod_{i=1}^{M} Z_{\fl{n\avec_i^1}, \fl{n\avec_{i+1}^2}}\notag \\
 &\phantom{xxxxxxxx} \times \Big(\min_{x_{0, n\e} \in \Pi_{i, n\e}^{1}}\prod_{j=1}^{\fl{n\e}} e^{\beta \om(x_j)}\times \min_{x_{0,n\e} \in \Pi_{i+1, n\e}^{2}} \prod_{j=1}^{\fl{n\e}} e^{\beta \om(x_j)}\Big)^{-1}\label{timcor}.
\end{align}
We are going to take logarithms on both sides of \eqref{timcor}. On the right-hand side, we have  $\log \min_{x_{0, n\e} \in \Pi_{i,n\e}^{1}}\prod_{j=1}^{\fl{n\e}} e^{-\beta| \om(x_j)|}$  which is the first passage percolation time  of $n\e$ steps if the weights were negative (so the negative of last passage percolation time if the weights where positive) with limiting constant $c_{-|\om|}(\e)$. 

Then, we conclude
\begin{align}
-\varlimsup_{n\rightarrow \infty} n^{-1}\log Q^{\omega}_n\{ x_{0,n} \in n\cN_{\e}(\gamma)\} 
&\ge p(d^{-1}\mathbf1) -\sum_{i=1}^{M-1} p(\avec^2_{i+1} - \avec^1_{i})+ 2Mc_{-|\om|}(\e) \notag\\
&\ge p(d^{-1}\mathbf1) -\sum_{i=1}^{M-1} p(\avec^2_{i+1} - \avec^1_{i})+ 2\e'd^{-1}c_{-|\om|}(\mathbf1) \notag\\
&\ge p(d^{-1}\mathbf1) -\sum_{i=1}^{M-1} p(\gamma(t_{i+1}) - \gamma (t_{i})) - C\e' \label{RS5}\\
&= p(d^{-1}\mathbf1) -\sum_{i=1}^{M-1} p\Big(\frac{\gamma(t_{i+1}) - \gamma (t_{i})}{t_{i+1}-t_i}\Big)(t_{i+1}-t_i) - C\e'\label{RS6}\\
&\ge p(d^{-1}\mathbf1) -\sum_{i=1}^{M-1} p(\gamma'(t_i))(t_{i+1}-t_i) - C\e'\label{RS7}.
\end{align}
Equation \eqref{RS5} is the result of  \eqref{e-size} and \eqref{modulus1}. Homogeneity of $p$ gives \eqref{RS6} and then \eqref{e-size} and \eqref{modulus2} give \eqref{RS7}. The bound in \eqref{RS7} is independent of $\e$ so letting $\e \to 0$ does not affect it. To get the result, let the mesh of the partition tend to $0$ and then let $\e' \to 0$.
\end{proof}

\chapter{The log-gamma Polymer Model}\label{log-gamma model}   

\section{The model}\label{logG}

\subsection{Introduction and results}

In the present chapter, we derive an explicit expression for the upper tail rate function for $\log Z_{\fl{n\uvec}}$  in the case of the $1+1$ dimensional log-gamma polymer. The computations are tractable exactly because of the Burke property. More detailed information about the model and basic properties can be found in Section \ref{loggpol}.

 The results for the particular $1+1$ dimensional  log-gamma model. The distributions of the $\om$ weights are i.i.d.
\be
\om(i,j) \sim \log Y, \quad \text{ where } Y^{-1} \sim \text{Gamma}(\mu),
\label{burkwait}
\ee
where the density of the Gamma$(\mu)$ is given by \eqref{density}. 
For this choice of i.i.d.~ weights, denote by 
\be
\lmgf_{\mu}(\xi) = \log \bE\big( e^{\xi \om(0,0)}\big)= \log \Gamma(\mu-\xi) - \log\Gamma(\mu)
\label{lmgfom}
\ee
the logarithmic moment generating function.

It is convenient for the proof of these results to write the vectors in $\bR_+^{2}$ using both their coordinates. The main result is an explicit formula for the upper tail large deviation rate function for the logarithm of the partition function. 

\begin{theorem} 
Let $r\in \bR$ and let $\om(i,j)$ to be distributed as in \eqref{burkwait}. The function $\rf_{s,t}(r)$ defined  by \eqref{psidef} is given by
\be
 \rf_{s,t}(r) 
                    = \sup_{\xi \in[0, \mu)} \Big\{ r\xi -\inf_{\xi < \theta < \mu} \bigl( t\lmgf_{\theta}(\xi)-s\lmgf_{\mu-\theta}(-\xi)\bigr)\Big\}.
\label{phiexact}
\ee
\label{up-rate}
\end{theorem}

\begin{remark}
While the symmetry of $\rf_{s,t}$ is clear by definition \eqref{psidef}, it is not immediately obvious from \eqref{phiexact}. From the proof of the theorem for $s\le t$ one can check that we can restrict the set where the inner infimum is taken to $\theta \in [(\mu+ \xi)/2, \mu)$. Under the assumption that $s\ge t$ it is easy to check that the infimum  is now attainable for $\theta \in (\xi, (\mu+ \xi)/2]$. If $s\le t$ and $\theta_0$ gives the infimum $\gamma_{*}$, it is easy to check that interchanging the roles of $s,t$ will give the same infimum $\gamma_{*}$ at $\theta_1 = \mu + \xi - \theta_0$. These symmetries will be explained in detail in the proof the theorem. 
\end{remark}
The next result is about the free-endpoint case.
 
\begin{theorem} Let $s>0$, $r\in \bR$. The large deviation rate function \eqref{psidef2}, for $\beta=1$ is
\be
-\lim_{n\rightarrow \infty}n^{-1}\log\bP\{ \log Z_{\fl{ns}}\ge nr \} =  \rf_{s/2,s/2}(r)
\ee
where $ \rf_{s,t}(r)$ is given by \eqref{phiexact}.
\label{free-rate}
\end{theorem}.

\subsection{The model and the Burke property}
\label{loggpol}

For the rest of this chapter we also the the parameter $\beta$ to equal $1$. Henceforth we adopt the multiplicative notation for the polymer measure. It is also convenient to adjust the notation for the partition functions and redefine the rate functions: 

On each site $(i,j)$ of  $\bZ_+^2$ we assign weights $Y_{i,j}$. For any $(k,\l)< (m,n)$ define the partition function for paths that start from $(k,\l)$ and whose endpoint is constrained to be $(m,n)$, by
\be
Z_{(k,\l),(m,n)} = \sum_{x_{\cdot}\in \Pi_{(k,l),(m,n)}}\prod_{j=k+\l+1}^{m+n}Y_{x_j},
\label{partition function}
\ee
 where $ \Pi_{(k,l),(m,n)}$ is the collection of up-right paths $x_{\cdot} = (x_j)_{k+\l \le j \le m+n }$ inside the rectangle $R_{k,\l}^{m,n} = \{k,k+1,...,m\}\times\{\l,...,n\}$ that go from $(k,\l)$ to $(m,n)$: $x_{k+\l}=(k,\l)$, $x_{m+n}=(m,n)$. We adopt the convention that $Z_{(k,\l),(m,n)}$ does not include the weight at the starting point. In the case that the weight at the starting point needs to be considered we also define  
\be
Z^{\square}_{(k,\l),(m,n)} = \sum_{x_{\cdot}\in \Pi_{(k,l),(m,n)}}\prod_{j=k+\l}^{m+n}Y_{x_j} = Y_{k,\l}Z_{(k,\l),(m,n)},
\label{squareparf}
\ee
If a value is needed, then assume that $Z_{k,\l,(k,\l)} = 1.$  In the special case where $(k,\l) = (0,0)$ we omit the subscript from the above notation and we also set $Y_{0,0}=1.$

We assign distinct  weight distributions on the  boundaries $\bN\times \{0\}$, $\{0\}\times \bN$ and in the bulk $\bN^2$.  To emphasize this, the symbols $U$ and $V$ will denote the weights on the horizontal and vertical boundary respectively:
\be
U_{i,0} = Y_{i,0}\quad \textrm{and} \quad V_{0,j} = Y_{0,j}.
\label{boundary letters}
\ee 

Our results depend on the explicit distribution of the weights; all weights are reciprocals of gamma variables. To be precise, here are the assumptions on the distributions:
\begin{assumption}
Let $0< \theta <\mu < \infty$. We assume that the weights $\{U_{i,0}, V_{0,j}, Y_{i,j}: i,j \in \bN\}$ are independent with distributions 
\be
U_{i,0}^{-1}\sim \text{Gamma}(\theta), V_{0,j}^{-1}\sim \text{Gamma}(\mu-\theta), \textrm{ and } Y_{i,j}^{-1}\sim \text{Gamma}(\mu).  
\label{weight distributions}
\ee
\label{weights2}
\end{assumption}
Given the initial weights $\{U_{i,0}, V_{0,j}, Y_{i,j}: i,j \in \bN\}$ and starting from the lower left corner of $\bN^2$, define inductively for $(i,j) \in \bN^2$
\be
U_{i,j} = Y_{i,j}\Big(1+\frac{U_{i,j-1}}{V_{i-1,j}}\Big),\,  V_{i,j} = Y_{i,j}\Big(1+\frac{V_{i-1,j}}{U_{i,j-1}}\Big)\, \textrm {and } X_{i-1,j-1} = \Big( \frac{1}{U_{i,j-1}}+\frac{1}{V_{i-1,j}}\Big)^{-1}.
\label{propagation}
\ee 

The partition function for the model with the boundary condition is denoted by $Z^{(\theta)}_{m,n}$  satisfies 
\be 
Z^{(\theta)}_{m,n}= Y_{m,n}(Z^{(\theta)}_{m-1, n}+Z^{(\theta)}_{m,n-1}) \textrm{ for } (m,n)\in \bN^2
\label{lptform}
\ee
and one can check inductively that
\be
U_{m,n}= \frac{Z^{(\theta)}_{m,n}}{Z^{(\theta)}_{m-1,n}}\,\, \textrm{ and }\,\,  V_{m,n}= \frac{Z^{(\theta)}_{m,n}}{Z^{(\theta)}_{m,n-1}}.
\label{fractionid}
\ee 
Equations \eqref{lptform} and \eqref{fractionid} are also valid for $\zsq$ since the weight at the origin is canceled from the equations. 

The key result that allows explicit calculations for this model is the Burke-type Theorem 3.3 in \cite{sepp-poly}.

Let $z_{\cdot} = (z_k)_{k\in \bZ}$ be a nearest-neighbor down-right path in $\bZ^2_{+}$, that is, $z_k \in \bZ^2_+$ and $z_k-z_{k-1}=e_1$ or $-e_2$. Denote the undirected edges of the path by  $f_k=\{z_{k-1},z_k\}$ and let 
\be
T_{f_k} = \begin{cases}
                   U_{z_k} & \text{if } f_k \text{ is a horizontal edge}\\
                   V_{z_{k-1}}  & \text{if } f_k \text{ is a vertical edge}.
\end{cases}
\notag
\ee   
Let the lower left interior of the path be the vertex set $\mathcal{I}=\{(i,j) \in \bZ^2_+: \exists \, m \text{ so that } (i+m,j+m)\in \{z_k\}\}.$ 

Recall the definition of $X_{i,j}$ from  \eqref{propagation}.
\begin{theorem}[\cite{sepp-poly}-Burke Property]
Under the assumption \eqref{weights2}, and for any down-right path $(z_k)_{k\in \bZ}$ in $\bZ_+^2$, the variables $ \{T_{f_k}, X_z: k \in \bZ, z\in \mathcal{I}\}$ are mutually independent with marginal distributions
\be
U^{-1}\sim \text{Gamma}(\theta), V^{-1}\sim \text{Gamma}(\mu-\theta), \textrm{ and } X^{-1}\sim \text{Gamma}(\mu).  
\label{wd2}
\ee
\end{theorem} 
From this, one can compute 
\be
\bE(\log Z^{(\theta)}_{m,n})=m\bE(\log U)+n\bE(\log V) = -m\Psi_0(\theta) - n \Psi_0(\mu-\theta). 
\label{mean log}
\ee

In \cite{sepp-poly} a law of large numbers is proved for the limiting free energy in the case of no boundary weights. Let $(s,t) \in \bR^{2}_{+}$ and observe that there exists a unique $\theta = \theta_{s,t} \in (0, \mu)$ such that $t\Psi_1(\mu-\theta) = s \Psi_1(\theta).$ Define
\be
p_{\mu}(s,t)= - (s\Psi_0(\theta_{s,t}) + t \Psi_0(\mu-\theta_{s,t})).
\label{b-limit}
\ee

For the model without boundaries with $Y^{-1}_{i,j} \sim \textrm{Gamma}(\mu),\, i\ge0, j\ge 0$, the limiting free energies can be evaluated explicitly to be 
\be
\lim_{n\rightarrow \infty} n^{-1}\log Z_{\fl{ns}, \fl{nt}} = p_{\mu}(s,t) \quad  \bP - a.s.
\label{lawcons} 
\ee
and
\be
\lim_{n\rightarrow \infty} n^{-1}\log Z^{\dtot}_{\fl{ns}} = -s\Psi_0(\mu/2)\quad  \bP - a.s.
\label{lawtotal} 
\ee

\section{Continuity of the rate function on the boundary}

The logarithmic moment generating function of the bulk weights
$\log Y_{i,j}\sim$ $-\log$ Gamma$(\mu)$ is
\[
\lmgf_\mu(\xi)=\log \bE\bigl( e^{\xi\log Y}\bigr) = 
\begin{cases}   \log \Gamma(\mu-\xi)-\log \Gamma(\mu), &\xi < \mu\\
\infty,  &\xi\ge\mu. \end{cases} 
\]
Its convex dual is the Cram{\'e}r rate function (recall \eqref{c1})  $I_\mu$
of this distribution,
given by 
\begin{align} 
I_{\mu}(r) 
                 &=-r\Psi^{-1}_0(-r)-\log\Gamma(\Psi^{-1}_0(-r))+\mu r+ \log\Gamma(\mu), \quad r\in \bR.  
\label{logcram}
\end{align}

From Theorem \ref{phidef} we know that the rate function for the model with i.i.d\ weights is continuous on the boundary, and in this case, its value equals the Cram{\'e}r rate function for the sum of i.i.d.\ weights with inverse Gamma$(\mu)$ distribution.  


The rate function on the boundary is given by
\be
 \rf_{0,x}(r) =  \rf_{x,0}(r) =  \begin{cases}
                                                   xI_{\mu}(rx^{-1}), &\quad r\ge  -x\Psi_0(\mu),\, x > 0\\
                                                   r\mu , &\quad r\ge 0,\, x=0 \\ 
                                                  0,   &\quad r <  -x\Psi_0(\mu),\, x \ge 0.
                                                  \end{cases}
\label{s,t=0}
\ee

For the second branch of \eqref{s,t=0} we used 
\be 
\lim_{x\searrow 0} x\log \Gamma(\Psi_0^{-1}(-rx^{-1})) =0, \quad r> 0.
\label{L'Hlimit}
\ee 
Note that $\rf_{0,0}(r)=r\mu$ as given by \eqref{j0def}. (Here $\mu=x_{\infty}$). One can also interpret the second branch as the large deviation rate function for a single $\log Y$  random variable, where $Y^{-1}\sim $ Gamma($\mu$).
 
The strong law of large numbers for the limiting constant of the free energy from \eqref{lawcons} and continuity of $ \rf$ give the support described in the theorem. The fact that $\rf$ is strictly increasing for $r \ge p_{\mu}(s,t)$  can follow independently after finding the rate function explicitly in Theorem \ref{up-rate}.

\section{Exact point-to-point rate function}
\begin{lemma}[Varadhan's lemma]
Let $Z_{\fl{ns},\fl{nt}}$ the partition function given by \eqref{partition function} with weights $Y$ such that $Y^{-1}\sim \text{Gamma}(\mu)$. Also let $\rf_{s,t}(r)$ the upper tail large deviation for $\log Z_{\fl{ns},\fl{nt}}$. Then for $0\le\xi < \mu$
\be
\lim_{n\to \infty} n^{-1}\log\bE e^{\xi \log Z_{\fl{ns},\fl{nt}}} = \sup_{r\in \bR}\{ r\xi - \rf_{s,t}(r) \} = \rf^*_{s,t}(\xi).
\notag
\ee 
\label{Varadhan's lemma}
\end{lemma}

\begin{proof}
Let $\gamma_{\inf} = \varliminf_{n\to \infty} n^{-1}\log\bE e^{\xi \log Z_{\fl{ns},\fl{nt}}}$ and $\gamma_{sup} = \varlimsup_{n\to \infty} n^{-1}\log\bE e^{\xi \log Z_{\fl{ns},\fl{nt}}}.$ First an exponential Chebychev argument for a lower bound:
\begin{align}
n^{-1}\log \bP\{ \log Z_{\fl{ns},\fl{nt}} \ge nr \} \le -\xi r + n^{-1}\log \bE e^{\xi \log Z_{\fl{ns},\fl{nt}}}.
\notag
\end{align}
By letting $n\to \infty$ on a suitable subsequence we get that for all $r\in \bR$
\[
\gamma_{\inf} \ge \xi r - \rf_{s,t}(r).
\]
Optimizing over $r$ we get $\gamma_{\inf}\ge  \rf_{s,t}^*(\xi).$

For the upper bound, first note that there exists $\alpha>1$ such that $\alpha \xi < \mu$,
\be
\sup_{n} \Big( \bE  e^{\alpha\xi \log Z_{\fl{ns},\fl{nt}}} \Big)^{1/n} < \infty.
\label{supcond}
\ee 
To see this, distinguish cases where $\alpha \xi <1$ or otherwise. For $\alpha\xi <1$,
\begin{align}
\Big(\bE  e^{\alpha\xi \log Z_{\fl{ns},\fl{nt}}} \Big)^{1/n} 
&=\Big(\bE\Big(\sum_{x \in \Pi(\fl{ns},\fl{ nt})} \prod_{i=1}^{\fl{ns}+\fl{nt}} Y_{x_j}\Big)^{\alpha\xi} \Big)^{1/n} \notag \\
&\le\Big(\sum_{x \in \Pi(\fl{ns},\fl{ nt})} \prod_{i=1}^{\fl{nt}+\fl{ns}}\bE Y^{\alpha\xi}\Big)^{1/n}\notag \\  
&\le e^{F(s,t) + o(1)}(\lmgf_{\mu}(\alpha\xi))^{t+s}.\notag  
\end{align} 
For $\alpha\xi \ge 1$, we use Jensen's inequality. Let N denote the number of paths.
\begin{align}
\Big(\bE  e^{\alpha\xi \log Z_{\fl{ns},\fl{nt}}} \Big)^{1/n} 
&=\Big(\bE\Big(\sum_{x \in \Pi(\fl{ns},\fl{ nt})} \prod_{i=1}^{\fl{ns}+\fl{nt}} Y_{x_j}\Big)^{\alpha\xi} \Big)^{1/n} \notag \\
&\le\Big(N^{\alpha\xi}\sum_{x \in \Pi(\fl{ns},\fl{ nt})} \prod_{i=1}^{\fl{nt}+\fl{ns}}\bE Y^{\alpha\xi}\Big)^{1/n}\notag \\  
&\le e^{\alpha\xi F(s,t) + O(1)}(\lmgf_{\mu}(\alpha\xi))^{t+s}.\notag  
\end{align}
Now, we can show that 
\be
\lim_{r \to \infty}\varlimsup_{n\to \infty}n^{-1}\log \bE \Big(e^{\xi \log Z_{\fl{ns},\fl{nt}}}\mathbf{1}\{ \log Z_{\fl{ns},\fl{nt}} \ge nr\}\Big)= - \infty.
\label{T-F}
\ee
Using H{\"o}lder's inequality,
\begin{align}
n^{-1}\log \bE \Big(&e^{\xi \log Z_{\fl{ns},\fl{nt}}}\mathbf{1}\{ \log Z_{\fl{ns},\fl{nt}} \ge nr\}\Big)  \notag\\
&\le \alpha^{-1}\log \sup_{n} (\bE e^{\alpha\xi \log Z_{\fl{ns},\fl{nt}}})^{1/n} + \frac{\alpha-1}{\alpha}n^{-1}\log \bP\{\log Z_{\fl{ns},\fl{nt}} \ge nr\}.\notag
\end{align}
Taking a limit $n\to \infty$ we conclude
\be
\varlimsup_{n\to \infty}n^{-1}\log \bE \Big(e^{\xi \log Z_{\fl{ns},\fl{nt}}}\mathbf{1}\{ \log Z_{\fl{ns},\fl{nt}} \ge nr\}\Big) \le C - \frac{\alpha-1}{\alpha} \rf_{s,t}(r).
\ee
Letting $r$ to $\infty$ concludes the proof, since $\rf_{s,t}(r)$ is a non-constant increasing convex function.

To establish an upper bound let $\delta >0$ and partition $\bR$ so that for $i\in \bZ$, $r_{i}= i\delta$. Then for any $m$
\begin{align}
n^{-1}\log\Big( \bE e^{\xi \log Z_{\fl{ns},\fl{nt}}} \Big)& = n^{-1}\log \Big(\sum_{i=-\infty}^{\infty} \bE e^{\xi \log Z_{\fl{ns},\fl{nt}}}\mathbf{1}\{nr_i \le \log Z_{\fl{ns},\fl{nt}} < nr_{i+1}\}\Big) \notag\\
&\le n^{-1}\log \Big( \sum_{i=-m}^{m} e^{n \xi r_{i+1}}\bP\{ \log Z_{\fl{ns},\fl{nt}} \ge nr_i\} \notag\\ 
&\phantom{xxxxxxx}+ e^{n\xi r_{-m}} + \bE \Big(e^{\xi \log Z_{\fl{ns},\fl{nt}}}\mathbf{1}\{ \log Z_{\fl{ns},\fl{nt}} \ge nr_m\} \Big)\Big)
\end{align} 
A limit along a suitable subsequence yields
\begin{align}
\gamma_{\sup} &\le \max\big\{\max_{-m\le i \le m}\{ \xi r_{i+1} - \rf_{s,t}(r_i)\}, \xi r_{-m},  C - \frac{\alpha-1}{\alpha} \rf_{s,t}(r_m)\}  \big\}\notag \\
&\le \max\{ \sup_{r}\{ \xi r - \rf_{s,t}(r)\} -\delta\xi,  \xi r_{-m},  C - \frac{\alpha-1}{\alpha} \rf_{s,t}(r_m)\}  \big\}\notag
\end{align}
To finish the proof, let $\delta \to 0$ and $m\to \infty$.
\end{proof}

An immediate consequence is that $\log Z^{\square}$ and $\log Z$ have the same rate functions. This follows from having the same convex dual for $0\le \xi < \mu$ (for $\xi \ge \mu$ both rate functions are $\infty$):
\begin{align}
\lim_{n \to \infty}n^{-1} \log \bE e^{\xi\log Z^{\square}_{\fl{ns}, \fl{nt}}}&= \lim_{n \to \infty}n^{-1} \log \bE e^{\xi\log Z_{\fl{ns}, \fl{nt}} + \xi\log Y_{1,1}}\notag \\ 
&= \rf^{*}_{s,t}(\xi) + \lim_{n \to \infty}n^{-1} \log \bE Y_{1,1}^{\xi}\notag \\
&= \rf^{*}_{s,t}(\xi).\notag    
\end{align} 
We are using this fact in the remaining part of the section without alerting the reader.

 For what follows we need some notational conventions. Assume the polymer lives in the rectangle $\{0,\dotsc,\fl{ns}\}\times \{0,\dotsc,\fl{nt} \}$ and let $-\fl{nt} \le k \le \fl{ns}$. Let $Y_{0,0} = 1$ and  define
\be
\eta_k = \begin{cases}
         \displaystyle \prod_{j=-k+1}^{\fl{nt}} V_{0,j}^{-1}, &\text{ for } -\fl{nt} \le k \le -1,\\
          \eta_{-1}, &k=0 \\
          \displaystyle \eta_0 \prod_{i=1}^{k}U_{i,0}   , &\text{ for } 0<  k \le \fl{ns},\\
          \end{cases}
\label{sums}
\ee
and 
\be
\mathbf{k}=\begin{cases}
                     (1,-k),& -\fl{nt} \le k \le -1,\\
                     (1,1), & k= 0, \\
                     (k,1), &  0<  k \le \fl{ns},
                     \end{cases}
\quad\text{and}\quad 
\fl{n\mathbf{a}}=\begin{cases}
                     (1,\fl{-na}) ,&  -t \le a < 0,\\
                     (1,1), & a= 0,\\
                     (\fl{na}, 1), &  0<  a \le s.
                     \end{cases}
\label{exitvec}
\ee

It is going to be notationally convenient to assume that $k = \fl{na}$ for some $a\in [-t,s]$. Whenever this happens, we identify $\mathbf{k}=\fl{n\mathbf{a}}$ and we assume that $n$ is large enough so that $\fl{na} \neq 0 $. When we take the limit as $n \rightarrow \infty$ to compute the various rate functions, we will need a continuous (and scaled) version of \eqref{exitvec}. For this reason we abuse this notation by writing
\be
\mathbf{a}= \lim_{n\rightarrow\infty}n^{-1}\fl{n\mathbf{a}} = \begin{cases}
                     (0,-a) ,&  -t \le a < 0,\\
                     (0,0), & a= 0\\
                     (a, 0), &  0<  a \le s.
                     \end{cases}
\label{exitvec2}
\ee
Observe that for all $a \in [-t,s]$, the r.v.\ $\log \eta_{\fl{na}}$ is a sum of independent $\log$-gamma random variables: 
 For $a < 0$, $\log \eta_{\fl{na}}$ is just a sum of i.i.d.\ $\log$ Gamma$(\mu-\theta)$, the Cram{\'e}r rate function exists and so are the lower and upper large deviations rate functions. 

For $a>0$, 
\be
\log \eta_{\fl{na}} = \sum_{j=1}^{\fl{nt}} \log V_{0,j}^{-1} -\sum_{i=1}^{\fl{na}}\log U_{i,0}^{-1}.
\notag
\ee
 Setting $L_{n} =  \sum_{j=1}^{\fl{nt}} \log V_{0,j}^{-1}$,  $Z_{n} =  -\sum_{i=1}^{\fl{na}}\log U_{i,0}^{-1}$, we appeal to Lemma \ref{sumsind} so the upper large deviation rate function 
\be
 -\lim_{n\rightarrow \infty}n^{-1}\log \bP\{\log \eta_{\fl{na}} \ge nr \} =  \kappa_a(r)
\label{liadef}
\ee
exists and is convex and continuous. 
With these definitions we can have the following computational lemma that we are using throughout.

\begin{lemma} Fix $a \in [-t,s]$ and let $\kappa_{a}(r)$ defined by \eqref{liadef}. Then 
\be
\kappa^*_a(\xi) =  \begin{cases}
                                       (t+a)\big(\log\Gamma(\mu-\theta + \xi) - \log \Gamma(\mu - \theta)\big), \quad -t\le a\le0,\,\, \xi \geq 0, \vspace{0.1 in}\\
                                        t\big(\log\Gamma(\mu-\theta + \xi) - \log \Gamma(\mu - \theta)\big) + a\big(\log \Gamma(\theta - \xi) - \log \Gamma(\theta)\big)\\
\phantom{xxxxxxxxxxxxxxxxxxxxxxxxxxxxxxxxx} 0< a \le s, \,\, 0 \le \xi < \theta, \\
\infty, \phantom{xxxxxxxxxxxxxxxxxxxxxxxxxxxxxx1}\text{otherwise}.
                                       \end{cases}
\label{l*}
\ee
\end{lemma}

\begin{proof}
Fix $a \in [-t,s].$ For $-t\le a\le 0$, the first branch of \eqref{l*} is the logarithmic moment generating function for log-gamma weights. The second branch comes by taking the limiting logarithmic moment generating function for $\xi>0$ of $\eta_{\fl{na}}$. 
\end{proof}




Recall \eqref{exitvec}. Let $(a, b) \in [-t,s]^2$ and let $m_{\kappa,a}$ and $m_{\rf, b}$ the zeros of $\kappa_a$ and $J_{(s,t) - \mathbf b}$ respectively. Define
\begin{align}
H^{a,b}_{s,t}(r) &= \lim_{n\to \infty}n^{-1}\log \bP\{ \log \eta_{\fl{na}} + \log Z^{\square}_{\fl{n\mathbf{b}}, (\fl{ns},\fl{nt})} \ge nr\}    \notag\\
                            & = \begin{cases}
0, & r <  m_{\kappa,a}+ m_{\rf, b}\\
\displaystyle\inf_{m_{\kappa,a}\le x \le r- m_{\rf, b}}\{ \kappa_a(x) + \rf_{(s,t) - \mathbf{b}}(r-x)\}, & \text{otherwise}.
\end{cases}
 \label{Hdef} 
\end{align}
where $\kappa_a(x)$ is given by \eqref{liadef}. The existence of $H^{a,b}_{s,t}(r)$ is established by Lemma \ref{sumsind} and continuity in the $b$ argument (when $a,s,t,r$ are fixed) follows directly from the continuity of $\rf$ in $b$, the fact that $\rf$ is always finite and that $x$ can be restricted in a compact set. In the case where $a=b$ we define $H^a_{s,t}(r)= H^{a,a}_{s,t}(r)$.

Let $s,t> 0$, $r\in \bR$.  Let $\big\{U_{i, \fl{nt}}\big\}_{1\le i\le \fl{ns}}$ be the weights as defined by \eqref{propagation}. Define
\be
R_s(r) = -\lim_{n \rightarrow \infty} n^{-1}\log \bP\Bigg\{ \sum_{i=1}^{\fl{ns}}\log U_{i,\fl{nt}}\ge nr \Bigg\}=\begin{cases} s I_{\theta}(rs^{-1}), &\quad r\ge -s\Psi_0(\theta),\\
                                    0, &\quad \text{otherwise}, 
\end{cases}
\label{JJ}
\ee 
with convex dual
\be
R^*_s(\xi) = \begin{cases} s\log \Gamma(\theta - \xi)-s\log \Gamma(\theta),\quad 0\le \xi< \theta,
\\
\infty, \quad \textrm{otherwise},
\end{cases}
\label{J*}
\ee
 where we use \eqref{stardef} and \eqref{logcram} to obtain the first branch.


\begin{figure}
\begin{center}
\includegraphics{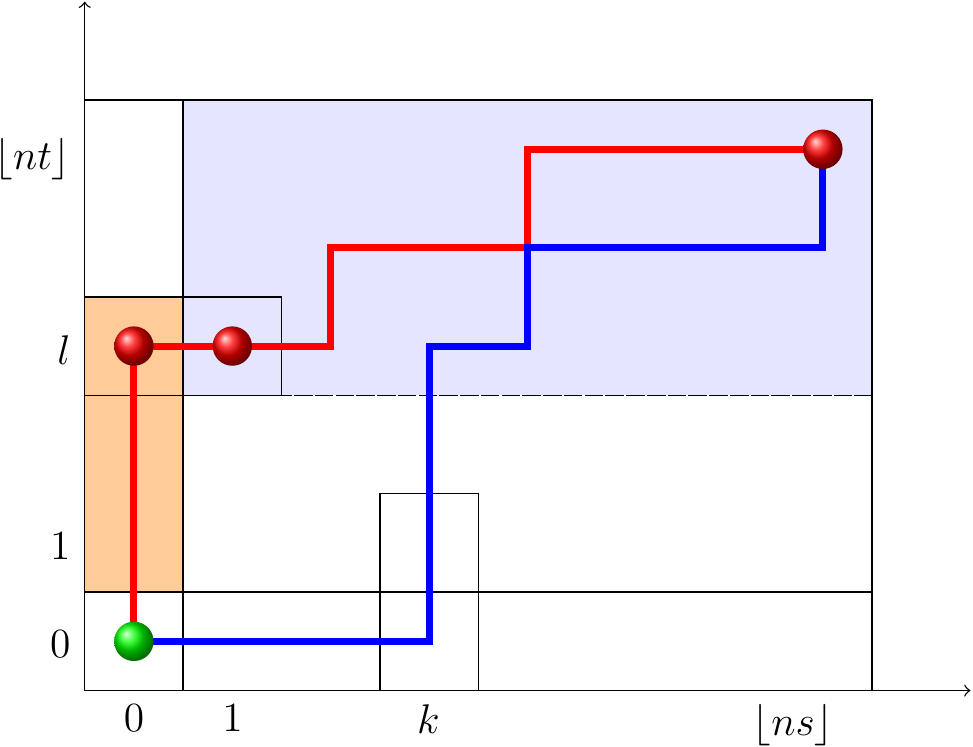}
\caption[Decomposition of $Z_{\fl{ns},\fl{nt}}$ according to the polymer exit point]{Two possible polymer paths with weights described by \eqref{weight distributions}. The shaded parts explain the decomposition of the partition function $Z_{\fl{ns},\fl{nt}}$  needed for Lemma \ref{decomp}.}
\label{polypic}
\end{center}
\end{figure}

\begin{lemma}
Let $s,t > 0$ , $a \in [-t,s]$ and $r\in \bR$. Let (for fixed $s,t, a$) $  \kappa_a$, $H^a_{s,t} $ and $R_s$ be defined by \eqref{liadef}, \eqref{Hdef} and \eqref{JJ}. Then
\be
R_s(r) =\inf_{-t\le a \le s }\{H^a_{s,t}(r)\} =  \inf_{-t\le a \le s} \,\,\inf_{m_{\kappa,a}\le x \le r- m_{\rf, a}}\{  \kappa_a(x) +  \rf_{(s,t)-\mathbf{a}}(r-x)\}.
\label{Jdefi}
\ee 
\label{decomp}
\end{lemma}
\begin{proof}
We start by decomposing $ Z^{(\theta)}_{\fl{ns}, \fl{nt}}$ according to the exit point of the polymer path from the boundary:
\begin{align}
 Z^{(\theta)}_{\fl{ns}, \fl{nt}}&= \sum_{x_{\cdot}\in \Pi_{(0,0),(\fl{ns}, \fl{nt})}}\prod_{j=1}^{\fl{ns}+ \fl{nt}} Y_{x_j} \notag \\
                             &= \sum_{\l=1}^{\fl{nt}}\bigg\{\bigg(\prod_{j=1}^{\l}V_{0,j}\bigg)Z^{\square}_{(1,\l),(\fl{ns},\fl{nt})}\bigg\}+  \sum_{k=1}^{\fl{ns}}\bigg\{\bigg(\prod_{i=1}^{k}U_{i,0}\bigg)Z^{\square}_{(k,1),(\fl{ns},\fl{nt})}\bigg\}.
\label{deco}
\end{align}

 Dividing both sides of \eqref{deco} by $Z^{(\theta)}_{0,\fl{nt}} = \prod_{j=1}^{\fl{nt}}V_{0,j}$ and through multiple uses of \eqref{fractionid}, we get 
\begin{align}
\prod_{i=1}^{\fl{ns}}U_{i,\fl{nt}} &=\sum_{\l=1}^{\fl{nt}}\bigg\{\bigg(\prod_{j=\l+1}^{\fl{nt}}V^{-1}_{0,j}\bigg)Z^{\square}_{(1,\l),(\fl{ns},\fl{nt})}\bigg\} \notag\\ &\phantom{xxxxxxxxxxxxxxxxxx}+\sum_{k=1}^{\fl{ns}}\bigg\{\bigg(\prod_{j=1}^{\fl{nt}}V^{-1}_{0,j}\prod_{i=1}^{k}U_{i,0}\bigg)Z^{\square}_{(k,1),(\fl{ns},\fl{nt})}\bigg\}\notag\\
&=\sum_{\substack{ k = -\fl{nt} \\[1pt] k \neq 0}}^{\fl{ns}}\eta_k Z^{\square}_{\mathbf{k},(\fl{ns},\fl{nt})}.
\label{L-Zdec}
\end{align}
Since all terms are nonnegative we get the bounds 
\begin{align}
 \eta_kZ^{\square}_{\mathbf{k},(\fl{ns},\fl{nt})} \le \prod_{i=1}^{\fl{ns}}U_{i,\fl{nt}}\le \sum_{k = -\fl{nt}}^{\fl{ns}}\eta_k &Z^{\square}_{\mathbf{k},(\fl{ns},\fl{nt})}\le \notag \\ 
&\le \fl{n(s+t)}\max_{-\fl{nt}\le k \le \fl{ns}} \eta_kZ^{\square}_{\mathbf{k},(\fl{ns},\fl{nt})}. \notag 
\end{align}
Take logs of these inequalities to get
\begin{align}
\log \eta_k + \log &Z^{\square}_{\mathbf{k},(\fl{ns},\fl{nt})} \le\sum_{i=1}^{\fl{ns}}\log U_{i,\fl{nt}}\le\log\bigg\{ \sum_{k = -\fl{nt}}^{\fl{ns}}\eta_k Z^{\square}_{\mathbf{k},(\fl{ns},\fl{nt})}\bigg\}\notag \\ 
&\phantom{xxxxxxxxxxxxx}\le \max_{-\fl{nt}\le k \le \fl{ns}} \big\{\log \eta_k + \log Z^{\square}_{\mathbf{k},(\fl{ns},\fl{nt})}\big\} + \log( n(s+t)). \label{UB}
\end{align}

For any $a \in [-t,s]$, 
\begin{align}
\lim_{n \rightarrow \infty}n^{-1}\log \bP\bigg\{ \sum_{i=1}^{\fl{ns}}& \log U_{i,\fl{nt}} \ge nr \bigg\} \notag \\ 
&\ge \lim_{n \rightarrow \infty}  n^{-1} \log \bP\big\{ \log \eta_{\fl{na}} + \log Z^{\square}_{\fl{n\mathbf{a}},(\fl{ns},\fl{nt})} \ge nr \big\} \notag \\
&\ge - H^{a}_{s,t}(r).\label{bo}
\end{align}
Equation \eqref{bo} is valid for all $a$, so optimizing over $a\in [-t,s]$ gives
\[
\lim_{n \rightarrow \infty}n^{-1}\log \bP\bigg\{ \sum_{i=1}^{\fl{ns}}\log U_{i,\fl{nt}} \ge nr \bigg\} \ge - \inf_{-t\le a \le s} \,\,\inf_{a_{  \kappa} \le x \le r- a_{ \rf}}\{  \kappa_a(x) +  \rf_{(s,t)-\mathbf{a}}(r-x)\}.
\]
This proves the upper bound in \eqref{Jdefi}.

The remaining of the proof is about the lower bound. Let $\e > 0$.  Fix a sufficiently small $\delta > 0$ and  let $-t = a_0 < a_1 < \dotsm < a_q=0 < \dotsm< a_m = s$  be a partition of the interval $-[t,s]$ so that $|a_{i+1}-a_i| < \delta$. For a given $\xi>0$ assume $\delta = \delta(\e,\xi)$ is sufficiently small so that 
$\delta\int_{\mu}^{\mu+ \xi}\Psi_0(x)\,dx < \e/4$. 

Without loss of generality assume $a_{i} \ge 0$.  For any integer $k \in [\fl{na_i}, \fl{na_{i+1}}]$, we can estimate

\begin{align}
\bP&\big\{ \log \eta_{k} + \log Z^{\square}_{\mathbf{k},(\fl{ns},\fl{nt})} \ge nr \big\}\notag\\
&\le\bP\bigg\{ \log \eta_{\fl{na_{i+1}}} + \log Z^{\square}_{\fl{n\mathbf{a}_i},(\fl{ns},\fl{nt})} -\sum_{i=k+1}^{\fl{na_{i+1}}}\log U_{i,0}  -\sum_{j=\fl{na_i}}^{k-1}\log  Y_{j,1} \ge nr  \bigg\} \label{sam} \\
&\le\bP\big\{ \log \eta_{\fl{na_{i+1}}} + \log Z^{\square}_{\fl{n\mathbf{a}_i},(\fl{ns},\fl{nt})} \ge n(r-\e)\big\} \notag \\
&\phantom{xxxxxxxxxxxxxxx}+ \bP\Bigg\{ -\sum_{i=k+1}^{\fl{na_{i+1}}}\log U_{i,0} -\sum_{j=\fl{na_i}}^{k-1} \log Y_{j,1} \ge n\e \bigg\} \notag \\
&\le\bP\big\{ \log \eta_{\fl{na_{i+1}}} + \log Z^{\square}_{\fl{n\mathbf{a}_i},(\fl{ns},\fl{nt})} \ge n(r-\e)\big\} \notag \\
&\phantom{xxxxxxxxxxxxxxx}+ e^{-\xi n\e}\big(\bE e^{\xi\log U^{-1}}\big)^{n(a_{i+1} -k-1)}\big(\bE e^{\xi\log Y^{-1}}\big)^{k-na_i-1} \notag\\
&\le\bP\big\{ \log \eta_{\fl{na_{i+1}}} + \log Z^{\square}_{\fl{n\mathbf{a}_i},(\fl{ns},\fl{nt})} \ge n(r-\e)\big\} + e^{-n\big(\xi \e -  2\delta\int_{\mu}^{\mu+ \xi}\Psi_0(x)\,dx\big)}\notag\\
&\le\bP\big\{ \log \eta_{\fl{na_{i+1}}} + \log Z^{\square}_{\fl{n\mathbf{a}_i},(\fl{ns},\fl{nt})} \ge n(r-\e)\big\} +  e^{-n\xi \e/2} \label{tiny}.
\end{align}

The key to the bound is the fact that the upper tail large deviation rate function for sums of  $\log Y^{-1}$ random variables has unbounded slope. This is not true for $\log Y$. This is why $\eta_k$ changes to $\eta_{\fl{na_{i+1}}}$ and  $ Z^{\square}_{\mathbf{k},(\fl{ns},\fl{nt})} $ to $ Z^{\square}_{\fl{n\mathbf{a}_i},(\fl{ns},\fl{nt})}$ in \eqref{sam}. For the case $a_i < 0$, the corresponding changes will be $\eta_k$ to $\eta_{\fl{na_i}}$ because the weights are reciprocals and  $ Z^{\square}_{\mathbf{k},(\fl{ns},\fl{nt})} $ to $ Z^{\square}_{\fl{n\mathbf{a}_i},(\fl{ns},\fl{nt})}$ as before.

Now for the actual error estimate. Assume $n$ is large enough so that $n\e  > \log(ns+nt)$. Equation \eqref{UB} implies
\begin{align}
n^{-1}&\log \bP\bigg\{ \sum_{i=1}^{\fl{ns}} \log U_{i,\fl{nt}} \ge nr \bigg\} \notag \\ 
& \le  n^{-1} \log \sum_{k=-\fl{nt}}^{\fl{ns}}\bP\big\{ \log \eta_{k} + \log Z^{\square}_{\mathbf{k},(\fl{ns},\fl{nt})} \ge nr - \log (ns+nt)\big\} \notag \\ 
&\le \max_{0\le i \le m-1}\,  \max_{ \fl{na_i} \le k \le \fl{na_{i+1}}} n^{-1} \log \bP\big\{ \log \eta_{k} + \log Z^{\square}_{\mathbf{k},(\fl{ns},\fl{nt})} \ge n(r - \e) \big\}\notag \\
&\phantom{xxxxxxxxxxxxxxxxxxxxxxxxxxxxxxxxxxxxxxxx} +n^{-1}\log(ns+nt) \notag\\
&\le \max_{0\le i \le m-1}n^{-1} \log \Big(\bP\big\{ \log \eta_{\fl{na_{i+1}}} + \log Z^{\square}_{\fl{n\mathbf{a}_i},(\fl{ns},\fl{nt})} \ge n(r-2\e)\big\} +  e^{-n\xi \e/2}\Big) + \e \label{UBSTEP}\\
&\le \max_{0\le i \le m-1}n^{-1} \log\bP\big\{ \log \eta_{\fl{na_{i+1}}} + \log Z^{\square}_{\fl{n\mathbf{a}_i},(\fl{ns},\fl{nt})} \ge n(r-2\e)\big\} \vee (-\xi \e/2)  + 2\e. \label{bp}
\end{align}
Equation \eqref{UBSTEP} follows from \eqref{tiny}. Take a limit $n \to \infty$ in equation \eqref{bp} to conclude 
\begin{align}
-R_s(r)&\le \max_{0\le i \le m-1}\{-H^{a_{i+1}, a_i}_{s,t}(r-2\e)\}\vee (-\xi \e/2)  + 2\e \label{better0}\\
&\le \max_{0\le i \le m-1}\{-H^{a_{i+1}, a_{i}}_{s,t}(r-2\e)\} + 2\e \label{better1}\\
&\le \max_{0\le i \le m-1}\{-H^{a_{i+1}, a_{i+1}}_{s,t}(r-2\e)\}+\e' + 2\e\label{better2}\\
&\le \max_{0\le i \le m-1}\{-H^{a_{i+1}, a_{i+1}}_{s,t}(r)\}+\e' \label{better3}\\
&\le \sup_{-t\le a \le s}\{-H^{a}_{s,t}(r)\}  +\e' \label{finalstep}.
\end{align}
Equation \eqref{better1} is the result of $\xi$ tending to infinity in \eqref{better0} and noting that $H^{a_{i+1}, a_i}_{s,t} < \infty$.

Equation \eqref{better2} requires explanation. Observe that for $r, s,t, \e$ fixed there exists a compact set $K$ that depends only the fixed parameters, that contains all intervals $K_{a,b} =[m_{\kappa, a}, r-2\e-m_{\rf,b}]$, for which the $x$ variable ranges over in definition \eqref{Hdef} for $H^{a, b}_{s,t}(r-2\e)$. This follows from the continuity of $m_{\kappa, a}$ and $m_{\rf, b}$ in $a, b$ respectively, and from the fact that $a,b$ range over compact sets. Note that by enlarging the compacts $K_{a,b}$ to $K$ we do not change the value of $H^{a,b}_{s,t}(r-2\e)$. 

For $x$ restricted in $K$, $\rf_{(s,t) - \mathbf b}(r-2\e +x)$ is uniformly continuous in $(b,x)$. Then, for $\e'>0$ we can assume the mesh of the partition is small enough so that for any fixed $x$
\[ -\e'\le \rf_{(s,t) - \mathbf a_i}(r-2\e +x) - \rf_{(s,t) - \mathbf a_{i+1}}(r-2\e +x) \le \e'.\]   

Hence, in \eqref{better2}, $\e'$ is the error that comes from changing the superscript in $H$ from $a_i$ to $a_{i+1}$. Equation \eqref{better3} comes from letting $\e \to 0$ and the continuity of $H$ in the $r$ variable that follows from arguments similar to those that justified \eqref{better2}.
Letting $\e'\to 0$ in \eqref{finalstep} gives the lemma. 
\end{proof}

We will also need the following lemma.
\begin{lemma}
For a fixed $\xi \in [0, \mu)$ the function 
\be
G_{\xi}(a) = \begin{cases}
                             - \rf_{(t,t)-\mathbf{a}}^*(\xi),&\textrm{for }   0\le a\le t\\
                             \infty, &\text{otherwise}
                            \end{cases} 
\ee
is  convex, lower semi-continuous  on $\bR$ and continuous on $[0,t]$. In particular, $G^{**}_{\xi}(a) = G_{\xi}(a)$
for $a\in \bR$.
\label{Gisconvex}
\end{lemma}

\begin{proof}  To show convexity on $[0,t]$ let $\lambda \in (0,1)$ and 
  $a = \lambda a_1+(1-\lambda)a_2$: 
\begin{align}
   - \rf_{(t,t)-\mathbf{a}}^*(\xi)  &= -\sup_{r\in \bR}\{\xi r -  \rf_{(t,t)-\mathbf{a}}(r)\}\notag \\ 
                                  &=\inf_{r\in \bR}\{ \rf_{t-a,t}(r)-\xi r\}\notag \\ 
                                  &\le \inf_{r\in \bR} \inf_{\substack{(r_1,r_2): \\ \lambda r_1 + (1-\lambda)r_2=r}}\{\lambda( \rf_{t-a_1,t}(r_1)-\xi r_1)+(1-\lambda) ( \rf_{t-a_2,t}(r_2)-\xi r_2)\} \notag\\
                                  &=  \inf_{\substack{(r_1,r_2)\in \bR^2}}\{\lambda( \rf_{t-a_1,t}(r_1)-\xi r_1)+(1-\lambda) ( \rf_{t-a_2,t}(r_2)-\xi r_2)\} \notag\\
                                  &=\lambda \inf_{r_1 \in \bR}\{ \rf_{t-a_1,t}(r_1)-\xi r_1\}+(1-\lambda) \inf_{r_2\in \bR}\{ \rf_{t-a_2,t}(r_2)-\xi r_2\} \notag\\
                                  &=-\lambda \rf_{t-a_1,t}^*(\xi)-(1-\lambda) \rf_{t-a_2,t}^*(\xi). \label{Gconv}
\end{align}
The inequality comes from the convexity of $ \rf$ in the variable $(t-a,t,r)$. 

For finiteness on $[0,t]$ it is now enough to show that $G_{\xi}(a)$ is finite at the endpoints. 
Continuity then  follows in the interior $(0,t)$. 
  First  take $a=t$.   Then $\rf_{0,t}^*$ is the dual of  a   Cram{\'e}r rate
function, and  for $\xi > 0$   
\be
G_{\xi}(t) = - \rf_{0,t}^*(\xi) = - t\log \bE e^{ \xi  \log Y_{1,0} } 
\label{va}
\ee
which is finite for $\xi<\mu$. 

Convexity of $ \rf_{s,t}(r)$  and symmetry $ \rf_{s,t}(r)= \rf_{t,s}(r)$ imply   $ \rf_{t,t}(r) \le  \rf_{0,2t}(r)$. From this  
\begin{align}
G_{\xi}(0) &= - \rf_{t,t}^*(\xi) = \inf_{r\in \bR}\{ \rf_{t,t}(r) - \xi r\} \notag\\
                      &\le \inf_{r\in\bR}\{  \rf_{0,2t}(r)-\xi r \} = -   \rf_{0,2t}^*(\xi) < \infty. 
\end{align}

\textit{Continuity at $a=0$ and $a=t$}.
Case $1$: $a=0$. To show that $G_{\xi}$ is also continuous at the endpoints, we first obtain a lower bound. For any $r \in \bR$,
\[
\rf^{*}_{t-a,t}(\xi) \ge r\xi - \rf_{t-a,t}(r)
\]
hence, by continuity of $\rf_{s,t}$ in the $(s,t)$ argument,
\be
\lim_{a\to 0}\rf^{*}_{t-a,t}(\xi) \ge r\xi - \rf_{t,t}(r).
\ee
Optimize over $r$ to conclude $\lim_{a\to 0}\rf^{*}_{t-a,t}(\xi) \ge \rf^{*}_{t,t}(\xi).$ 

For the upper bound, let $n\in \bN$. Then we estimate, using Lemma \ref{Varadhan's lemma}
\begin{align}
\rf^{*}_{t,t}(\xi)&= \lim_{n\to \infty}n^{-1}\log \bE e^{\xi \log Z_{\fl{nt},\fl{nt}}}\notag \\   
                            &\ge \lim_{n\to \infty}n^{-1}\log \bE e^{\xi \log Z_{\fl{n(t-a)},\fl{nt}}}+  \lim_{n\to \infty}n^{-1}\log \bE e^{\xi \sum_{i=\fl{n(t-a)}+1}^{\fl{nt}}\log Y_{i, \fl{nt}} }\notag \\   
                            &=\rf_{t-a,t}^*(\xi)+ a \log \bE Y ^{\xi}.   
\end{align}
Taking $a \to 0$ yields the result.

Case $2$: $a=t$. The lower bound follows as in the previous case. For the upper bound we use a path counting argument. Consider first the case where $0\le\xi<1$.
Then,  
\begin{align}
\rf_{t-a,t}^*(\xi) &= \lim_{n\to \infty}n^{-1}\log \bE\Big(\sum_{x \in \Pi(\fl{n(t-a)}, \fl{nt})} \prod_{i=1}^{\fl{nt}+\fl{n(t-a)}} Y\Big)^\xi\notag \\
                             &\le  \lim_{n\to \infty}n^{-1}\log\sum_{x \in \Pi(\fl{n(t-a)}, \fl{nt})} \prod_{i=1}^{\fl{nt}+\fl{n(t-a)}}\bE(Y)^\xi\notag \\  
&=  F(t-a,t) + (2-a/t)\rf_{0,t}^*(\xi)\notag  
\end{align}
For $1\le \xi <\mu$, Jensen's inequality yields
\be
\rf_{t-a,t}^*(\xi)\le   \xi F(t-a,t) +(2-a/t)\rf_{0,t}^*(\xi).
\notag
\ee
Let $a\to t$ to get the result.

The fact $G^{**}_{\xi} = G_{\xi}$ is now a direct consequence of lower semicontinuity, by\cite[Thm.~12.2]{rock-ca}
\end{proof}

\begin{proof}[Proof of Theorem \ref{up-rate}.]  We begin by expressing the 
explicitly known dual $R^{*}_s(\xi)$ from \eqref{J*} in terms of the unknown function
$\rf_{(s,t)-\mathbf{a} }$.  Recall that $a \in [-t,s]$ is the macroscopic exit point of the polymer chain from the boundary and 
$\avec$ is given by \eqref{exitvec2}.

Fix $0\le \xi <\mu$.
By \eqref{Hdef} and \eqref{convo} we can write
$
H^a_{s,t}(r) =  (  \kappa_{a}\square \rf_{(s,t)-\mathbf{a} })(r).
$
Then by  \eqref{stardef} and  \eqref{Jdefi} 
\begin{align}
R^{*}_s(\xi) 
                        & = \sup_{-t\le a\le s}\sup_{r}\big\{ r\xi -  (  \kappa_{a}\square \rf_{(s,t)-\mathbf{a} })(r) \big\} \notag \\
                        & = \sup_{-t\le a\le s}  (  \kappa_{a}\square \rf_{(s,t)-\mathbf{a} })^*(\xi) \notag \\
                        & = \sup_{-t\le a\le s} \big\{   \kappa_{a}^*(\xi)+  \rf_{(s,t)-\mathbf{a}}^*(\xi) \big\} \quad \text{by \eqref{addual}}. \label{stepp1}
\end{align} 
Equations \eqref{J*} and \eqref{stepp1} give
\be
 s\log \Gamma(\theta - \xi)-s\log \Gamma(\theta) =  \sup_{-t\le a\le s} \big\{   \kappa_{a}^*(\xi)+  \rf_{(s,t)-\mathbf{a} }^*(\xi) \big\},\quad  0 \le \xi < \theta.
\ee

Define 
\be
u(\theta) =  \begin{cases}
                                     -\hks{(\theta)}=\lmgf_{\mu-\theta}(-\xi)  =\log\Gamma(\mu-\theta + \xi) - \log \Gamma(\mu - \theta), &\quad -t\le a\le0 \vspace{0.1 in}\\
                                       d_{\xi}{(\theta)}=\lmgf_{\theta}(\xi)  = \log \Gamma(\theta - \xi) - \log \Gamma(\theta), &\quad 0< a\le s
                                       \end{cases}
\label{wow}
\ee
and substitute  \eqref{l*}, \eqref{wow} into equation \eqref{stepp1} to get
\be
 s\log \frac{\Gamma(\theta - \xi)}{ \Gamma(\theta)} -t\log\frac{\Gamma(\mu-\theta + \xi) }{ \Gamma(\mu - \theta)}=  \sup_{-t\le a\le s} \big\{ au(\theta)+ \rf_{(s,t)-\mathbf{a} }^*(\xi) \big\}
\label{almost}
\ee

We now specialize to the case $s=t$ and we treat $\theta$ as a variable in the remaining part of the proof. We assume that $\theta \in (\xi, \mu)$ and $\xi \ge 0$ is fixed. When $s=t$, symmetry of $\rf_{s,t}$ allows us to write \eqref{almost} as
\be
 t(\dks(\theta)+\hks(\theta))=  \sup_{0\le a\le t} \big\{ a \max\{\hks(\theta),\dks(\theta)\}+ \rf_{t-a,t}^*(\xi) \big\}.
\label{almost1}
\ee

Assume first that $(\mu+\xi)/2 \le \theta<\mu$. This is equivalent to $\hks(\theta) \ge \dks(\theta)$ and equation
\eqref{almost1} turns into  
\be
 t(\dks(\theta)+\hks(\theta))=  \sup_{0\le a\le t} \big\{ a \hks(\theta)+ \rf_{t-a,t}^*(\xi) \big\}.
\label{almost2}
\ee
Let $\hks(\theta)= v$ and $G_{\xi}(a) = - \rf_{t-a,t}^*(\xi)$. This notation makes \eqref{almost2}
\be
 t((d_{\xi}\circ \hks^{-1})(v)+ v)=  \sup_{0\le a\le t} \{ av- G_{\xi}(a) \} = G^*_{\xi}(v), \quad \hks(\tfrac{\mu + \xi}2)\le v < + \infty.
\label{almost3}
\ee

Now assume that $\xi < \theta \le (\mu +\xi)/2$. Then, equation \eqref{almost1} becomes 
\be
 t(\dks(\theta)+\hks(\theta))=  \sup_{0\le a\le t} \big\{ a \dks(\theta)+ \rf_{t-a,t}^*(\xi) \big\}.
\label{almost4}
\ee

Let $\psi_{\mu,\xi}:(\xi, (\mu +\xi)/2] \longrightarrow [ (\mu +\xi)/2, \mu),\,\, \theta\mapsto\psi_{\mu, \xi}(\theta) = \mu-\theta+\xi$ is a homeomorphism between the intervals $(\xi, (\mu +\xi)/2]$ and $[ (\mu +\xi)/2, \mu)$. It has the following properties: First, $\dks(\theta) = \hks(\mu - \theta + \xi) = \hks(\psi_{\mu,\xi}(\theta))$ and 
$\dks(\mu-\theta+\xi) = \dks(\psi_{\mu,\xi}(\theta)) = \hks(\theta)$, hence it fixes the sum 
\[ \hks(\theta) + \dks(\theta) = \hks(\psi_{\mu, \xi}(\theta)) + \dks(\psi_{\mu, \xi}(\theta)). \] 
Equation \eqref{almost4} can then be re-written as
\be
 t(\dks(\psi_{\mu,\xi}(\theta))+\hks(\psi_{\mu,\xi}(\theta)))=  \sup_{0\le a\le t} \big\{ a \hks(\psi_{\mu,\xi}(\theta))+ \rf_{t-a,t}^*(\xi) \big\},
\label{almost5}
\ee
where $\hks(\psi_{\mu,\xi}(\theta)) \in [ \hks(\tfrac{\mu + \xi}2),  \infty).$ This shows that equations \eqref{almost4} and \eqref{almost3} are equivalent, so we can restrict $\theta \in [\tfrac{\mu + \xi}2, \mu)$ without loss of generality. We will work only with equation \eqref{almost3} from now on. 

We compute
\begin{align}
 \rf_{(t,t)-\mathbf{a}}(r) &= \sup_{\xi \in[0,\mu)}\{r\xi- \rf_{t-a,t}^*(\xi)\} \quad \text{by \eqref{dcd} and Theorem \ref{phidef}},\notag \\
                         &= \sup_{\xi\in[0,\mu)}\{r\xi + G_{\xi}(a)\}\notag\\
                         &= \sup_{\xi\in[0,\mu)}\big\{r\xi + \sup_{v\in \bR}\{ av-  G^*_{\xi}(v)\}\big\}\quad \text{by \eqref{dcd} and Lemma \ref{Gisconvex}},\notag\\
                         &=\sup_{\xi\in[0,\mu)} \sup_{v\in \bR} \{r\xi + av -  G^*_{\xi}(v)\}\label{F-dual0}.
\end{align}
We now argue that $ \sup_{v\in \bR}\{ av- G^*_{\xi}(v)\}$ can be achieved when $\hks(\tfrac{\mu + \xi}2)\le v < + \infty$. 

Equation \eqref{almost3} gives the values of $G^*_{\xi}(v)$ for $v \in  [\hks(\tfrac{\mu + \xi}2), \infty)$ and we see that $G^{*}_{\xi}(v)$ is differentiable for $v \in (\hks(\tfrac{\mu + \xi}2), \infty)$ values. The derivative (using \eqref{almost3} and \eqref{wow}) is   
\be
\frac{d}{dv}G^{*}_{\xi}(v) = t\bigg(1 + \frac{\Psi_0(\hks^{-1}(v)-\xi) - \Psi_0(\hks^{-1}(v))}{\Psi_0(\mu +\xi - \hks^{-1}(v)) - \Psi_0(\mu - \hks^{-1}(v)) }\bigg).
\label{Gder}
\ee 

The right derivative of $G^*_{\xi}$ as $v\to \hks(\tfrac{\mu + \xi}2)$ is 
\begin{align}
\Big({\frac{d}{dv}G^{*}_{\xi}}\Big)_{+}(\hks(\tfrac{\mu + \xi}2)) 
&= t\bigg(1 + \lim_{\hks^{-1}(v)\searrow \tfrac{\mu+\xi}2}\frac{\Psi_0(\hks^{-1}(v)-\xi) - \Psi_0(\hks^{-1}(v))}{\Psi_0(\mu +\xi - \hks^{-1}(v)) - \Psi_0(\mu - \hks^{-1}(v)) }\bigg)\notag\\
&= t\bigg(1 +\frac{\Psi_0({\tfrac{\mu-\xi}2}) - \Psi_0( {\tfrac{\mu+\xi}2})}{\Psi_0( {\tfrac{\mu+\xi}2}) - \Psi_0( {\tfrac{\mu-\xi}2}) }\bigg)\notag\\
&=0.
\label{Gder2}
\end{align}
Recall that  $G^*_{\xi}(v)$ is a convex function of $v$ and note that its subdifferential set (see \cite{rock-ca}) is a subset of $[0,t]$. Since $a \ge 0$, this implies that  the supremum cannot be attained for $v < \hks(\tfrac{\mu + \xi}2)$. 

Therefore,
\be
\rf_{(t,t)-\mathbf{a}}(r) = \sup_{\xi\in[0,\mu)}\,\,\, \sup_{v\in [\hks((\mu + \xi)/2),\infty)} \{r\xi + av -  G^*_{\xi}(v)\}\label{F-dual}.
\ee

To compute $ \rf_{s,t}(r)$ we can assume without loss of generality that $s < t$ (otherwise observe that  $ \rf_{s,t}(r)=  \rf_{t,s}(r)$). Then let $a = t-s $, observe that $  \rf_{s,t}(r) =   \rf_{t-(t-s),t}(r)$ and use \eqref{F-dual}. 

\end{proof}
\begin{remark}
For $s=t=1$, the rate function has a unique $0$ at $r_0 = -2\Psi_0(\mu/2)$. Even though an explicit Taylor expansion is not easily computable, we can still show the first term in the expansion around $r_0$ has order $3/2$. To simplify the calculations that follow, assume that $\mu = 2$. 

Let $r = r_0 + \e$. Using \eqref{phiexact}, it is not hard to verify that the maximizing $v$ is $\hks(1+\xi/2)$. Then
\be
\rf_{1,1}(r) = \sup_{0<\xi< 2}\{ (r_0 +\e)\xi - 2(\log\Gamma(1 - \xi/2)) - \log\Gamma(1 + \xi/2))\}.
\label{fo}
\ee 
Take the $\xi$-derivative of the expression in the braces of \eqref{fo} to show (also by using \eqref{telehar}) that the maximizing $\xi$, $\xi_{max}$,  solves the equation
\be
\e = 2\gamma - \Psi_{0}(1-\xi_{max}/2)-\Psi_0(1+\xi_{max}/2) = \sum_{k=1}^{\infty}\bigg( -\frac{2}{k} + \frac{2}{2k -\xi_{max}}+\frac{2}{2k+\xi_{max}}  \bigg).
\label{seriessol}
\ee
All terms in \eqref{seriessol} are positive, so $\xi_{max}$ must be such so that the first term of the series satisfies
\be
 -2 + \frac{2}{2-\xi_{max}}+\frac{2}{2+\xi_{max}} <\e. 
\label{ineq}
\ee 
On the other hand, since $\xi_{max} < 1$, one can easily bound the sum in \eqref{seriessol} from above, to obtain
\be
\e < 4\xi^2_{max}\sum_{k=1}^{\infty}\frac{1}{2k(2k-1)(2k+1)}<  {4\xi^2_{max}}\int_{1}^{\infty}\frac{1}{(2x-1)^3}\,dx=\xi_{max}^2
\label{xilb}
\ee
Solve \eqref{ineq},\eqref{xilb} to obtain 
\be
\e^{1/2}<\xi_{max}< 2\e^{1/2}.
\label{xiub}
\ee 
We now bound \eqref{fo}in the following manner:
\begin{align}
\rf_{1,1}(r) &=r_0\xi_{max} +\e\xi_{max} - 2(\log\Gamma(1 - \xi_{max}/2)) - \log\Gamma(1 + \xi_{max}/2)) \notag \\
                   &= \sup_{\e^{1/2}<\xi<2\e^{1/2}} \Big\{r_0\xi +\e\xi +2\Psi_0(1)\xi+\frac{\Psi_2(1)}{12}\xi^3 + \mathcal{O}(\xi^5)\Big\},\notag          
\end{align}
since by \eqref{xiub}, we can Taylor expand the braces for small values of $\xi$ . Recall that $r_0 = -2\Psi_0(1)$ and that $\Psi_2$ is finite away from $0$, to conclude that there exist positive constants $C_1$, $C_2$ so that 
\be
C_1(r-r_0)^{3/2}+\mathcal{O}\big((r-r_0)^{5/2}\big)< \rf_{1,1}(r)<C_2(r-r_0)^{3/2}+\mathcal{O}\big((r-r_0)^{5/2}\big).
\label{convspeed}
\ee         
\end{remark}

\section{Exact free-endpoint rate function}

\begin{figure}
\begin{center}

\includegraphics[trim = 150 250 100 250,scale=0.8 ]{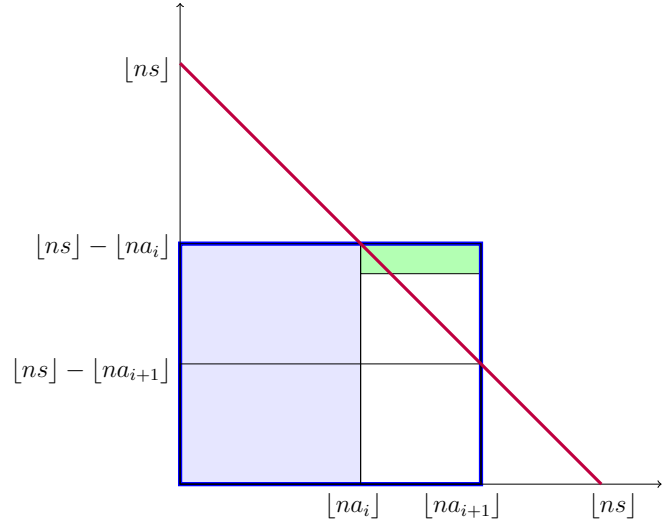}
\caption[Illustration of the bounds for Theorem \ref{free-rate}]{ The partition function in the thick-set square bounds from above the product of the partition functions of the shaded parts. 
This is used in the proof of Theorem \ref{free-rate}. 
}
\label{boundpart}
\end{center}
\end{figure}


We now turn our attention to the free endpoint case with no boundary. The conclusion of Theorem \ref{up-rate} suffices to prove Theorem \ref{free-rate}.

\begin{proof}[Proof of Theorem \ref{free-rate}.]
It is easy to check the following bounds for $0\le k \le \fl{ns}$:
\begin{align}
 \log\big(Z_{\fl{n \tfrac s2}, \fl{ns} - \fl{n\tfrac s2}}\big)& \le \log Z^{\dtot}_{\fl{ns}} \le\label{ublb1} \\
&\le \log (ns+1) +  \log\big(\max_k  Z_{k, \fl{ns} - k}\big).  
\label{ublb2}
\end{align}
To get the upper bound let $n\rightarrow \infty$ using \eqref{ublb1}.

We show the lower bound. Let $\e > 0$ and let $n$ large enough so that $n^{-1} \log( ns+1) <\e$.
\begin{align}
\bP\{ \log Z^{\dtot}_{\fl{ns}}\ge nr  \} &\le \bP\{ \log ns +  \log\big(\max_k  Z_{k, \fl{ns} - k}\big) \ge nr  \} \notag \\
                                                               &\le \bP\{ \max_k\,\, \log Z_{k, \fl{ns} - k} \ge n(r-\e)  \} \notag \\
                                                                &\le \sum_{k=0}^{\fl{ns}}\bP\{\log Z_{k, \fl{ns} - k} \ge n(r-\e)  \}. \label{fbub1}
\end{align}
After taking logarithms on both sides of \eqref{fbub1}, we have the upper bound 
\be
\log \bP\{ \log Z^{\dtot}_{\fl{ns}}\ge nrs  \} \le \max_k\,\, \log \bP\{\log Z_{k, \fl{ns} - k} \ge n(r-\e)  \} + \log( ns+1).
\ee
Now take $\e' > 0$ and a partition $0= a_1 < a_2<  ... < a_m = s$ of the interval $[0,s]$ with small enough mesh so that for $r$ fixed, 
$|\rf_{a_i, s- a_i}(r) -\rf_{a_{i+1}, s- a_i}(r)|<\e  $ .
We conclude that for $n$ large enough \eqref{ublb2} implies
\begin{align}
n^{-1}&\log\bP\{ \log Z^{\dtot}_{\fl{ns}}\ge nr  \} \le \notag \\ 
                      &\le \max_{1\le i \le m}n^{-1} \log \bP\{\log Z_{\fl{na_i}, \fl{ns} - \fl{na_i}} \ge n(r-\e)  \} + \e \notag\\
 &\le \max_{1\le i \le m}n^{-1} \log \bP\Big\{\log Z_{\fl{na_{i+1}}, \fl{ns} - \fl{na_i}} - \sum_{i=\fl{na_i}+1}^{\fl{na_{i+1}}}\log Y_{i, \fl{ns}-\fl{na_i}} \ge n(r-\e)  \Big\} + \e.
\label{b4l}
\end{align}
The same arguments as in \eqref{UBSTEP}-\eqref{better2} turn \eqref{b4l} into
\begin{align}
\lim_{n\to \infty}n^{-1}\log \bP\{ \log Z^{\dtot}_{\fl{ns}}\ge nr  \}  &\le -\min_{1\le i\le m-1}\rf_{a_{i+1}, s-a_i }(r- 2\e)+ \e\notag\\
                      &\le -\min_{1\le i\le m}\rf_{a_i, s-a_i }(r)+ \e'
\label{lnow}\\
&\le -\inf_{0\le a\le s}\rf_{a, s-a }(r) + \e'. \label{woot}
\end{align}  
Equation \eqref{lnow} is a result of letting $\e \to 0$ and after that adjusting the index of the rate function. Finally, let $\e' \rightarrow 0$, to conclude
\be
-\lim_{n\rightarrow \infty}n^{-1}\log \bP\{ \log Z^{\dtot}_{\fl{ns}}\ge nr  \} \ge \inf_{0\le a\le s} \rf_{a,s-a}(r).
\label{flb}
\ee
By Theorem \ref{phidef}, we have convexity of $ \rf_{s,t}(r)$ in the $(s,t)$ argument. This gives
\be
 \rf_{s/2,s/2}(r)\le \tfrac 12  \rf_{a, s-a}(r)+ \tfrac 12  \rf_{s-a, a}(r) =  \rf_{a, s-a}(r),
\ee
where the last equality follows from the symmetry relation $ \rf_{s,t}(r) =  \rf_{t,s}(r)$. This concludes the proof.
\end{proof}



\chapter{Multiphase TASEP-Introduction and Results}\label{introtasep}

\section{Introduction}
The last two chapters  study hydrodynamic limits of totally asymmetric simple
exclusion processes (TASEPs)
 with spatially inhomogeneous jump rates given by functions 
that are allowed to have discontinuities.  We
prove  a general hydrodynamic limit
and compute some explicit solutions, even though information
about invariant distributions is not
available.  The results come through a variational formula that 
 takes 
advantage of the known behavior of the homogeneous TASEP.  This way we
are able to get explicit formulas, even though the usual scenario in
hydrodynamic limits is that explicit equations and solutions require 
explicit computations of expectations under invariant distributions.  
Together with explicit hydrodynamic profiles we can present 
 explicit limit shapes 
for the related last-passage growth models with spatially inhomogeneous 
rates.  

The class of particle  processes we consider are defined by a positive speed function
$c(x)$ defined for $x\in\bR$, lower semicontinuous and assumed to have a
discrete set of discontinuities.  
 Particles reside at sites of $\bZ$, subject 
to the exclusion rule that admits at most one particle at 
each site. 
The dynamical rule is that a particle
jumps from site $i$ to site $i+1$ at rate $c(i/n)$ provided site 
$i+1$ is vacant. Space and time are both scaled by the factor 
$n$ and then we let $n\to\infty$.  We prove  the almost sure vague
convergence of the empirical measure   to a density $\r(x,t)$, assuming that
the initial particle configurations have a well-defined macroscopic density profile
$\r_0$. 

From known behavior of driven conservative particle
systems a natural expectation would be that the 
macroscopic density $\r(x,t)$  of this discontinuous  TASEP ought to
be, in some sense, the unique entropy solution
of an initial value problem  of the type 
\be \r_t + (c(x)f(\r))_x = 0, \quad \r(x,0)= \r_0(x). \label{discpde}\ee
 Our proof of the hydrodynamic limit does
not lead directly to this  scalar conservation law.  
We can make the
connection through some recent PDE theory  in the special case of the two-phase flow
where the speed function is piecewise constant with a single discontinuity.   
In this case   the discontinuous TASEP chooses the unique entropy solution.  
We would naturally expect TASEP to choose the correct entropy solution in general,
but we have not investigated the PDE side of things further   to justify   such a 
  claim.  
 
The remainder of this introduction reviews briefly some relevant literature and then
closes with an overview of the contents of these last two chapters.  The model and the results are
presented in Section \ref{results}. 

\paragraph{Discontinuous scalar conservation laws.}
The study of  scalar conservation laws  \be \r_t + F(x, \r)_x = 0\label{disc2}\ee
 whose
flux $F$ may admit discontinuities in $x$
has taken off in the last decade.  
As with the multiple weak solutions of even the simplest spatially homogeneous case,
a key issue is the identification of the unique physically relevant solution 
by means of a suitable {\sl entropy condition}.  (See Sect.~3.4 of \cite{evan} for 
textbook theory.)  
Several different entropy 
conditions for the discontinuous case have been proposed, motivated 
by  particular physical problems.  See for example  
\cite{adim-gowd-03, MR2356208, audu-pert-05, chen-even-klin-08, Diehl, Klin-Ris, ostr-02}.  
Adimurthi, Mishra and Gowda 
\cite{MR2356208} discuss   how different theories  
 lead to different choices of relevant solution.  
An interesting phenomenon is that limits of vanishing higher order 
effects can  lead to distinct choices (such as vanishing viscosity vs.\ vanishing
capillarity).   
 
However, the model we study does not offer  more than one choice. 
 In our case the graphs of the different fluxes do not intersect as they are
all  multiples  of  $f(\r)=\r(1-\r)$.  In such  cases   it is expected that all 
the entropy
criteria single out the same solution (Remark 4.4 on p.~811 of
\cite{adim-etal-05}).  By  appeal to the theory developed by 
Adimurthi and Gowda \cite{adim-gowd-03}   we show that the 
discontinuous TASEP chooses entropy solutions of   equation \eqref{discpde}
in the case where $c(x)$ takes two values separated by a 
single discontinuity   

Our approach to the hydrodynamic limit goes via the interface process
whose limit is a Hamilton-Jacobi equation.   
Hamilton-Jacobi equations with discontinuous spatial dependence 
have been studied by 
Ostrov  \cite{ostr-02}  via mollification. 

\paragraph{Hydrodynamic limits for spatially inhomogeneous, driven  conservative  particle 
systems.}   Hydrodynamic limits for the case where the speed function possesses
some degree of  smoothness 
  were proved over a decade ago by Covert and Rezakhanlou 
 \cite{cove-reza} and   Bahadoran 
  \cite{Bahadoran-98}. For the case where the speed function is continuous, a hydrodynamic limit was proven by Rezakhanlou
in \cite{Reza-2002}  by the method of \cite{sepp99K}.   
  
The most relevant and interesting predecessor to our work is the study of   
  Chen et al.\ \cite{chen-even-klin-08}. They combine an existence proof of
  entropy solutions for \eqref{disc2} under certain technical hypotheses on $F$ 
   with a hydrodynamic limit 
for an attractive zero-range process (ZRP)   with discontinuous 
speed function.   The hydrodynamic limit is proved through a compactness argument for
approximate solutions that utilizes measure-valued solutions.  
The approach follows   \cite{Bahadoran-98, cove-reza} by
 establishing a microscopic entropy inequality which under the limit turns into a
macroscopic entropy inequality.  

The scope of \cite{chen-even-klin-08} and  our work are significantly different.  
Our flux $F(x,\r)=c(x)\r(1-\r)$ does not satisfy the hypotheses of \cite{chen-even-klin-08}.  
 Even with spatial inhomogeneities, a ZRP has product-form invariant distributions that can be readily written down
and computed with.    This is a key distinction in comparison with exclusion processes. 
The microscopic entropy inequality in \cite{chen-even-klin-08}
 is derived by a coupling with a stationary process.  

 Finally, let us emphasize the distinction between the present work and some
 other hydrodynamic limits that feature spatial inhomogeneities.   
Random rates (as for example  in \cite{sepp99K})   
  lead to 
homogenization (averaging) and   the macroscopic   flux does not depend 
on the spatial variable.   Somewhat similar but still fundamentally different is TASEP
with a slow bond.  In this model jumps across bond $(0,1)$ occur at rate $c<1$
while all other jump rates are $1$.  The deep question is whether the slow bond
disturbs the hydrodynamic   profile for all $c<1$.  V.~Beffara, V.~Sidoravicius and M.~E.~Vares 
have announced a resolution of this question in the affirmative.  
Then the hydrodynamic
limit can be   derived  in the same way as in the main theorem here.  The solution is not entirely explicit, however: one unknown constant
  remains that quantifies the effect of the slow bond
(see    \cite{sepp01slow}).     \cite{baha-04-aop}   generalizes the hydrodynamic limit of  \cite{sepp01slow}  
  to a broad class of driven particle systems with a microscopic blockage.  

\paragraph{Organization.}   Section \ref{results} 
contains the main results for the inhomogeneous corner growth model and TASEP.  
 Sections \ref{CGM} and  
\ref{hydrolimit}   prove the   limits.   Section 
\ref{Density}  outlines  the explicit computation  of 
  density profiles for the two-phase TASEP.   
Section \ref{pdes}  discusses the  connection with PDE theory.  

\paragraph{Notational conventions.}  
The Exp($c$) distribution has 
density $f(x)=ce^{-cx}$ for $0<x<\infty$. 
Two last passage time models appear:  the   corner growth model 
whose last-passage times are denoted by $G$,  
and the equivalent wedge growth model with last-passage times $T$.  
 $H(x) = \mathbf{1}_{[0,\infty)}(x)$ is the Heavyside function. 
 $C$ is a constant that may change from line to line.

\section{Results}

\label{results}
The corner growth model connected with TASEP has been a central object of study
in this area since the seminal 1981 paper of Rost \cite{rost}.  So let us begin 
 with an explicit description of the limit shape for a
  two-phase  corner growth model with a boundary along the diagonal.   
Put  independent 
exponential random variables $ \{ Y_v\}_{v\in\bN^2}$ on the points of the lattice 
with distributions 
\be
Y_{(i,j)}\sim\left\{
\begin{array}{lll}

\vspace{0.1 in}
\textrm{Exp}(c_1), & \textrm {if } i < j \\

\vspace{0.1 in}
\textrm{Exp}(c_1\wedge c_2), & \textrm {if } i=j \\

\textrm{Exp}(c_2), &\textrm {if } i > j.
\end {array}
\right.
\ee 
We  assume that the rates satisfy $c_1 \geq c_2.$  

Define the last passage time 
\be G(m,n) = \max_{\pi \in \Pi(m,n)}\sum_{v\in \pi}Y_v, \quad (m,n)\in \bN^2, \label{basicref}\ee 
where $\Pi(m,n)$ is the collection of  weakly increasing nearest-neighbor
 paths in the rectangle $[m]\times[n]$ that start from $(1,1)$ and go up to $(m,n).$
 That is, elements of $\Pi(m,n)$ are sequences 
$\{(1,1)=v_1, v_2, \dotsc, v_{m+n-1}=(m,n)\}$ such that 
$v_{i+1}-v_i=(1,0)$ or $(0,1)$.   
 
\begin{theorem}
\label{two-phaselastpassagetime}
Let the rates $c_1 \geq c_2>0$.  Define  $c =  {c_1}/{c_2}\geq 1$
 and $b  = 2c-1-2\sqrt{c(c-1)}.$ Then the a.s.\ limit 
\[ \Phi (x,y) =\dlim n^{-1}G (\fl{nx},\fl{ny})\] exists
for  all $(x,y)\in(0,\infty)^2$  and is given by
\[
\Phi (x,y) = 
 \begin{cases} c_1^{-1}\left( \sqrt{x}+\sqrt{y}\right)^2 , &\text{if } 0 < x \leq b^2y \\[11pt]
x\dfrac{4c-(1+b)^2}{c_1(1-b^2)}   +y\dfrac{(1+b)^2-4cb^2}{c_1(1-b^2)} , 
&\text{if } b^2y< x < y \\[11pt]
c_2^{-1}\left( \sqrt{x}+\sqrt{y}\right)^2 , & \text{if } y \leq x < +\infty. 
\end{cases}
 \] 
\end{theorem} 

This theorem will be obtained as a side result of the development in
Section \ref{CGM}.

\begin{figure}
\begin{center}
\includegraphics{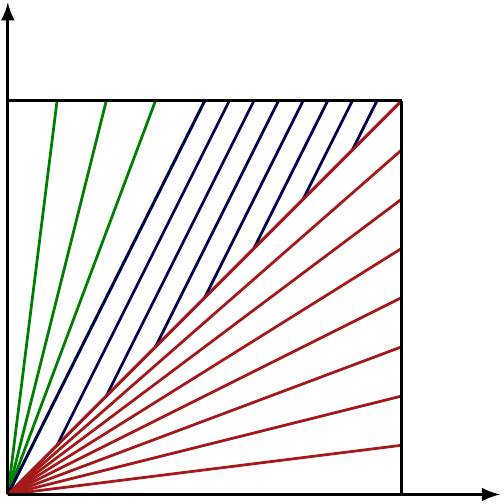}
\caption[Optimal two-phase macroscopic paths]{Optimal macroscopic paths that give the last 
passage time constant described in Theorem 
\ref{two-phaselastpassagetime}.} 
\label{Optimal paths}
\end{center}
\end{figure}

We turn to the general hydrodynamic limit.  The variational description needs the
following ingredients.  Define the wedge 
 $$\cW = \{ (x,y)\in \bR^2: y\geq0, \, x\geq -y \}$$ 
 and on $\cW$  the last-passage  function of homogeneous TASEP by 
  \be\gamma(x,y) =(\sqrt{x+y} +\sqrt{y})^2  .\label{gammanot}\ee 
Let $\textbf{x}(s) = (x_1(s), x_2(s))$ denote a path in $\bR^2$ and set 
\begin{align}
\cH(x,y) =\{ \textbf{x}\in C( [0,1], \cW) &: \textbf{x} \text{ is   piecewise }C^1, 
\textbf{x}(0)=(0,0),\notag \\ 
             &\textbf{x}(1)=(x,y), \, \textbf{x}'(s)\in \cW \textrm{ wherever the derivative is defined} \}. \notag
\end{align}  

The  {\sl speed function} $c$ of our system is by assumption 
  a positive lower semicontinuous function on $\bR$. We assume that at each $x\in \bR$ 
\be 
c(x) = \min\Big\{ \lim_{y \nearrow x}c(y),  \lim_{y \searrow x}c(y)  \Big\}.
\label{llrl}
\ee  In particular we assume that the limits in  \eqref{llrl}  exist. We also assume that $c(x)$  has only 
   finitely many discontinuities in any compact set, hence  it
  is bounded away 
  from $0$ in any compact set.   
  
For the hydrodynamic limit   consider
 a sequence of exclusion processes $\eta^n = (\eta^n_i(t): i \in \bZ, \, t\in\bR_+)$ 
  indexed by $n \in \bN$.
These processes
are constructed on a common probability 
space that supports the initial configurations $\{\eta^n(0)\}$ 
and the Poisson clocks of each process.  As always, the clocks of
process $\eta^n$ are independent of its initial state $\eta^n(0)$.
The joint distributions across the index $n$ are immaterial,
except for the assumed initial law of large numbers \eqref{weaklaw}
below.       In the process $\eta^n$ 
a particle at site $i$ attempts  
 a  jump  to $i+1$ with rate $c(i/n)$. Thus the generator of   $\eta^n$ is 
\be
L_nf(\eta) = \sum_{x\in \bZ} c(xn^{-1})\eta(x)(1-\eta(x+1))(f(\eta^{x,x+1})-f(\eta))  
\ee 
for  cylinder functions $f$ on the state space $\{0,1\}^\bZ$.  The usual notation is 
that particle configurations are denoted by $\eta=(\eta(i):i\in\bZ)\in\{0,1\}^\bZ$ and 
\be \eta^{x,x+1}(i) =
\begin{cases}
0 & \text{when } i=x\\
1 & \text{when } i=x+1 \\
\eta(i) &\text{when } i\neq x,x+1 
\end{cases}
\notag
\ee
is the configuration that results from moving a particle from $x$ to $x+1$. 
Let $J_i^n(t)$ denote the number of particles that have made the jump from site $i$ to site $i+1$ in  time interval $[0,t]$ in the process
 $\eta^n.$

An initial macroscopic profile $\rho_0$ is a   measurable function on $\bR$ such that $0\leq \rho_0(x)\leq 1$ for all real $x$, with antiderivative $v_0$ satisfying
\be
v_0(0)=0, \quad v_0(b) - v_0(a) = \int_{a}^{b}\rho_0(x)\,dx.
\label{vdef}
\ee

The macroscopic flux function of the constant rate 1 TASEP  is   \be
f(\r) =\left\{
\begin{array}{ll}

\vspace{0.1 in}
\r(1-\r), & \textrm {if } 0 \leq \r \leq 1 \\ 

-\infty, & \textrm{otherwise. }
\end {array}
\right.
\label{fnot}
\ee 
Its Legendre conjugate   $$f^*(y) = \inf_{0\leq\rho\leq1}\{y\r - f(\r)\}$$ 
represents the limit shape in the wedge. 
We orient our model so that
 growth in the wedge proceeds upward, and so 
we use  $g(y) = -f^*(y)$. It is explicitly given by
\be
g(y) = \sup_{0\leq\rho\leq1}\{f(\rho) - y\rho\} = \left\{
\begin{array}{lll}

\vspace{0.1 in}
-y, & \textrm {if } y \leq -1  \\

\vspace{0.1 in}
\frac1{4}{(1-y)^2}, & \textrm {if } -1 \leq y \leq 1 \\ 

0, & \textrm{if } y \geq 1.
\end {array}
\right.
\ee

For $x \in \bR$  define $v(x,0) = v_0(x)$, and for $t>0,$
\be
v(x,t)=\sup_{w(\cdot)}\left\{  v_0(w(0)) - \int_{0}^{t}c(w(s))\,g\left( \dfrac{w'(s)}{c(w(s))} \right)\,ds \right\}
\label{velocityversion1}
\ee 
where the supremum is taken over continuous piecewise $C^1$ paths $w:[0,t]\longrightarrow \bR$ that satisfy $w(t)=x.$ 
The function $v(x,t)$ is Lipschitz continuous jointly in $(x,t).$
 (see Section \ref{hydrolimit}) and it has a derivative almost everywhere. The macroscopic density is defined by  $\r(x,t) = v_x(x,t).$

The initial distributions of the processes
$\eta^n$  are arbitrary subject to the condition that the following strong law of large numbers holds
at time $t=0$: 
for all real $a<b$
\be
\dlim \frac{1}{n}\sum_{i=\fl{na}+1}^{\fl{nb}}\eta^n_i(0)=\int_a^b\rho_0(x)\,dx \quad \text{a.s.}  
\label{weaklaw}
\ee  

The second theorem gives the hydrodynamic limit of current and particle density for TASEP with discontinuous jump rates.

\begin{theorem}
Let $c(x)$ be a lower semicontinuous positive function satisfying \eqref{llrl}, with finitely many discontinuities in any compact set. Under assumption \eqref{weaklaw}, these strong laws of large numbers hold at each   $t>0$: for all real numbers $a<b$
\be
\dlim n^{-1}J^n_{\fl{na}}(nt) = v_0(a)-v(a,t)\quad a.s.
\label{macrocurrent}
\ee 
and 
\be
\dlim \frac{1}{n}\sum_{i=\fl{na}+1}^{\fl{nb}}\eta^n_i(nt)=\int_a^b\rho(x,t)\,dx \quad a.s.  
\label{macrospeed}
\ee 
where $v(x,t)$ is defined by \eqref{velocityversion1} and $\rho(x,t)=v_x(x,t).$
\label{hydro}
\end{theorem}

\begin{remark} In a totally asymmetric
$K$-exclusion with speed function $c$ the state space would be $\{0,1,\dotsc,K\}^\bZ$
with $K$ particles   allowed at each site, and  one particle   moved from site $x$
to $x+1$ at rate $c(x/n)$ whenever such a move can be legitimately completed.  
 Theorem \ref{hydro} can   be proved for this process with the same method of
 proof. The definition of the limit  \eqref{velocityversion1} would be the same, 
 except that   the explicit  flux $f$ and wedge shape $g$ would be replaced by the
 unknown functions $f$ and $g$ whose existence was proved in \cite{sepp99K}. 
\end{remark}

 To illustrate Theorem \ref{hydro} we compute the macroscopic 
density profiles $\r(x,t)$  
  from constant   initial conditions
in the  two-phase  model with
speed function 
\be     c(x) = c_1(1-H(x))+
c_2H(x) \label{defca}\ee 
where $H(x) = \mathbf{1}_{[0,\infty)}(x)$ is the Heavyside function
and  $c_1 \geq c_2$.  (The case $c_1<c_2$ can then be deduced from 
 particle-hole duality.)  The particles hit 
the region of lower speed as they pass the origin from left to right.  Depending on the
initial density $\rho$, we see the system adjust to this 
discontinuity  in  different ways to match the actual throughput
of particles on either side of the origin.   The maximal
flux on the right is $c_2/4$ which is realized on the left at 
densities $\r^*$ and $1-\r^*$ with 
\[\r^* = \tfrac 12 - \tfrac 12 \sqrt{1-c_2/c_1}.\]
 
\begin{corollary} Let $c_1 \geq c_2$ and
the speed function as in \eqref{defca}.   Then the macroscopic density profiles with initial conditions $\r_0(x,0)=\r$ are given as follows. 

{\rm (i)} Suppose $0< \r < \r^*$. Define $r^*=r^*(\r)= \tfrac 12 - \tfrac 12 \sqrt{1-4\r(1-\r)c_1/c_2} $.  Then 
\be
\r(x,t) = 
\left\{
\begin{array}{llll}

\vspace{0.1 in}
\r & \textrm {if } -\infty \leq x \leq 0\\

\vspace{0.1 in}
r^*  &\textrm{if } 0\leq x \leq c_2(1-2r^*)t \\

\vspace{0.1 in}
\dfrac{1}{2}\left(1- \dfrac{x}{tc_2}\right) &\textrm{if } c_2(1-2r^*)t\leq x \leq c_2(1-2\r)t\\

\r &\textrm{if } (1-2\r)tc_2\leq x < +\infty
\end{array}
\right.
\ee

\begin{figure}
\includegraphics{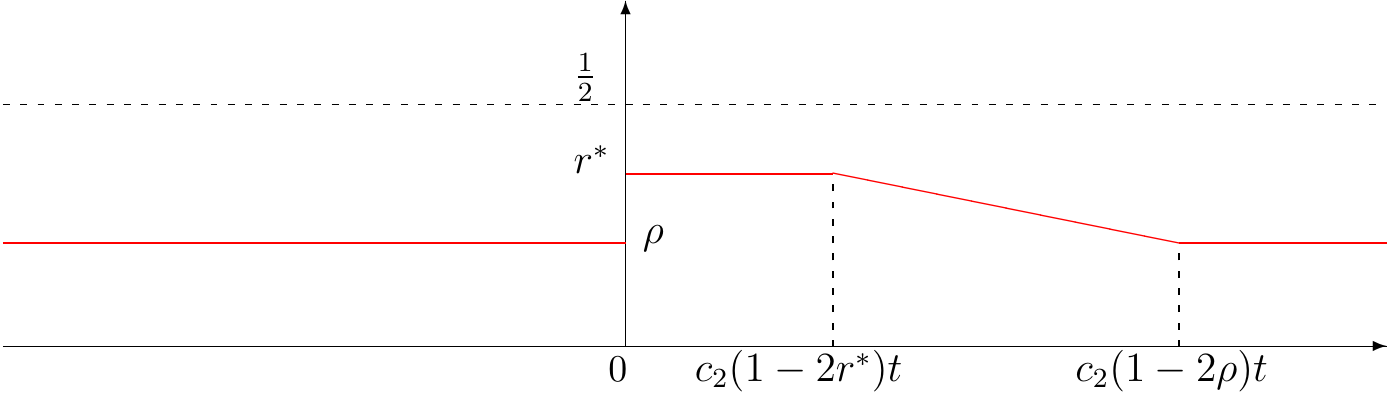}
\caption[Macroscopic two-phase density profile (Case 1) ]{Density profile $\r(x,t)$ in the two-phase ($c_1 > c_2$) TASEP when we start from constant initial configurations $\r_0(x)\equiv \r \in (0,\r^*).$}
\label{fig1}
\end{figure}

{\rm (ii)} Suppose $\r^* \leq \r\leq \frac{1}{2}. $ Then 
\be
\r(x,t) = 
\left\{
\begin{array}{llll}

\vspace{0.1 in}
\r & \textrm {if } -\infty \leq x \leq -tc_1(\r-\r^*) \\

\vspace{0.1 in}
1-\r^*  &\textrm{if } -tc_1(\r-\r^*)\leq x \leq 0 \\

\vspace{0.1 in}
\dfrac{1}{2}\left(1-\dfrac{x}{tc_2}\right) &\textrm{if } 0 \leq x \leq (1-2\r)tc_2\\

\r &\textrm{if } (1-2\r)tc_2\leq x < +\infty
\end{array}
\right.
\ee

\begin{figure}
\includegraphics{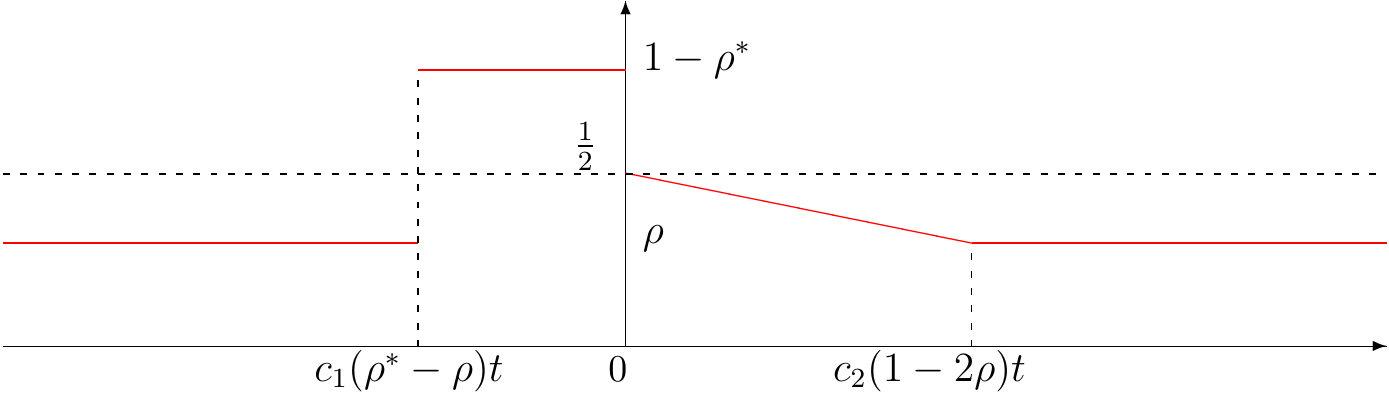}
\caption[Macroscopic two-phase density profile (Case 2) ]{Density profile $\r(x,t)$ in the two-phase ($c_1 > c_2$) TASEP when we start from constant initial configurations $\r_0(x)\equiv \r \in [\r^*, \tfrac 12].$}
\label{fig2}
\end{figure}

{\rm (iii)} Suppose $ \r\ge \frac{1}{2}.$  Define $r^*=r^*(\r)= \tfrac 12 - \tfrac 12 \sqrt{1-4\r(1-\r)c_2/c_1} $.
Then
\be
\r(x,t) = 
\left\{
\begin{array}{lll}

\vspace{0.1 in}
\r & \textrm {if } -\infty \leq x \leq -tc_1(\r-r^*) \\

\vspace{0.1 in}
1-r^*  &\textrm{if } -tc_1(\r-r^*)\leq x \leq 0 \\

\r &\textrm{if } 0<x <+\infty
\end{array}
\right.
\ee
\label{densityprofiles}

\begin{figure}
\includegraphics{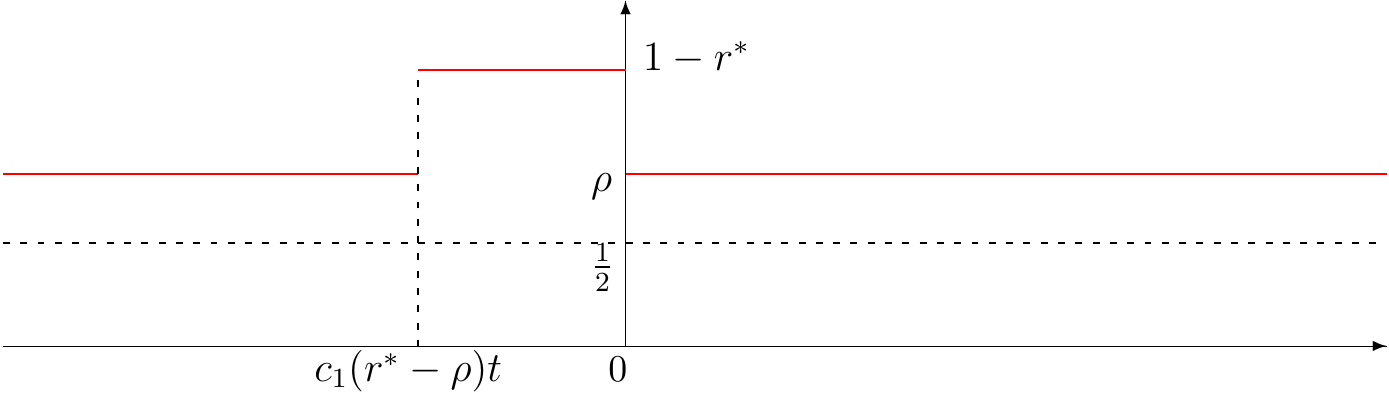}
\caption[Macroscopic two-phase density profile (Case 3)]{Density profile $\r(x,t)$ in the two-phase ($c_1 > c_2$) TASEP when we start from constant initial configurations $\r_0(x)\equiv \r \in (\tfrac 12,1).$} 
\label{fig3}
\end{figure}
\end{corollary} 

\begin{remark}
Taking $t\to\infty$ in the three cases of Corollary 
\ref{densityprofiles} gives a family of macroscopic profiles that are
fixed by the time evolution.  A natural question to investigate would be the
existence and uniqueness of invariant distributions that correspond
to these macroscopic profiles.  
\end{remark}

Next we relate the density profiles picked by the discontinuous TASEP
to entropy conditions for  scalar conservation laws with discontinuous 
fluxes.
The entropy conditions 
defined by Adimurthi and Gowda \cite{adim-gowd-03} are particularly suited to
our needs. Their results give   
uniqueness of the solution for the scalar conservation law 
\be
\left\{
\begin{array}{ll}

\vspace{0.1 in}
\r_t+\left( F(x,\r) \right)_x = 0, &  x\in \bR, t > 0  \\

\r(x,0) = \r_0(x), &  x \in  \bR
\end {array}
\right.
\label{scl}
\ee with distinct fluxes on the half-lines:
\be
F(x,\r) = H(x)f_r(\r)+(1-H(x))f_{\ell}(\r)  \label{Disflux}
\ee 
where  $f_r,f_{\ell} \in C^{1}(\bR)$ are strictly concave with superlinear decay to $-\infty$
as $\abs{x}\to\infty$.
A solution of \eqref{scl}   means a weak solution, that is,  
$\r \in L^{\infty}_{\text{loc}}(\bR\times\bR_+)$ such that for all continuously differentiable, compactly supported test functions $\phi \in C_c^1\left(\bR\times\bR_+\right)$, 
\be
\int_{-\infty}^{+\infty}\int_{0}^{+\infty}\left( \r\frac{\partial\phi}{\partial t}+F(x,\r)\frac{\partial\phi}{\partial x} \right)\,dt\,dx + \int_{-\infty}^{+\infty}\r(x,0)\phi(x,0)\,dx=0.
\label{weakformulation}
\ee  
\eqref{weakformulation} 
is the weak formulation of the problem
\be
 \begin{cases} 
 \r_t+f_r(\r)_x =0, & \textrm {for } x > 0, t>0  \\[5pt]
 \r_t+f_{\ell}(\r)_x =0, & \textrm{for } x <0, t>0 \\[5pt]
f_r(\r(0+,t)) = f_{\ell}(\r(0-,t))   &\text{for a.e.\ $t>0$} \\[5pt] 
\r(x,0)=\r_0(x). 
\end{cases}
 \label{Breakscl}
\ee 
 
The   entropy conditions used in   \cite{adim-gowd-03} come in two sets and assume
the existence of certain one-sided limits: 

\medskip

\textsl{$(E_i)$ Interior entropy condition, or Lax-Oleinik entropy condition:}  
\be
  \r(x+,t) \geq  \r(x-,t) \quad \text{for $x\ne0$ and for all $t>0$.}  
\label{Ei}
\ee  

\textsl{$(E_b)$  Boundary entropy condition at $x=0$:}
for almost every $t$,    the limits $\r(0\pm ,t)$ exist  and  one of the following   holds: 
\be
f_r'(\r(0+,t))\geq 0 \quad\textrm{and}\quad  f_{\ell}'(\r(0-,t))\geq 0,
\label{Eb1}
\ee 
\be
f_r'(\r(0+,t))\leq 0 \quad\textrm{and}\quad  f_{\ell}'(\r(0-,t))\leq 0,
\label{Eb2}
\ee 
\be
f_r'(\r(0+,t))\leq 0 \quad\textrm{and}\quad  f_{\ell}'(\r(0-,t))\geq 0.
\label{Eb3}
\ee 

\smallskip

Define 
\[G_x(p)= \mathbf{1}\{x> 0\}f_r^*(p)  +\mathbf{1}\{x< 0\} f_{\ell}^*(p) + 
\mathbf 1\{x=0 \} \min\bigl( f^*_r(0), f^*_{\ell}(0)\bigr),\]
 where $f_r^*$ and  $f_{\ell}^*$ are the convex duals of $f_r$ and $f_{\ell}$.
Set $V_0(x) = \int_0^x \r_0(\theta)\,d\theta$ and define
\be
V(x,t) = \sup_{w(\cdot)}\left\{ V_0(w(0)) + \int_0^t G_{w(s)}\bigl( w'(s) \bigr)\,ds \right\}
\label{defV} \ee
where the supremum is   over continuous, piecewise linear paths $w: [0,t] \longrightarrow \bR$ with $w(t) = x.$  
 
 \begin{theorem} \cite{adim-gowd-03}
Let $\r_0\in L^{\infty}(\bR)$ and define $V$ by \eqref{defV}.
 Then   $V$ is a uniformly Lipschitz continuous function and 
 $\r(x,t)=V_x(x,t)$ is the unique weak solution of  \eqref{Breakscl} 
 that satisfies the entropy assumptions $(E_i)$ and  
 $(E_b)$ in the class $L^{\infty}\cap BV_{\text{\rm loc}}$ and with discontinuities 
 given by a discrete set of Lipschitz curves. 
 \label{Uniq}
 \end{theorem}

It is easy to check that the two-phase density profile $\r(x,t)$ in Corollary \ref{densityprofiles} is a weak solution (in the sense of \eqref{weakformulation}) to the scalar conservation law \eqref{scl} with flux function $F(x,\r) = c(x)\r(1-\r)$. However we cannot immediately apply this theorem in our case since the two-phase flux function $\widetilde{F}(x,\r) =(1- H(x))c_1f(\r) +  H(x)c_2f(\r)$ is finite only for $\r \in [0,1]$ and in particular is not $C^1.$ We show how we can replace $F(x,\r)$ with $\widetilde{F}(x,\r)$
in the above theorems in Section \ref{pdes}. In particular, we prove the following.

\begin{theorem}
For $\r \in \bR$ define $f_r(\r) = c_2(1-\r)\r$ and $f_{\ell}(\r) = c_1(1-\r)\r$ to be the flux functions for the scalar conservation law \eqref{Breakscl}. 
Let  the initial macroscopic profile 
for the hydrodynamic limit be a measurable function
 $0\le\r_0(x)\le 1$. 
  Then the  macroscopic density profile $\r(x,t)$ from the 
hydrodynamic limit in Theorem \ref{hydro} is the unique
solution described in 
 Theorem \ref{Uniq}.  
\label{pdecor}
\end{theorem} 

\chapter{Multi-phase TASEP}\label{Multi-phase TASEP}   

\section{Wedge last passage time}
\label{CGM}
The strategy of the proof of the hydrodynamic limit is the one from \cite{sepp99K}
and \cite{sepp01slow}.
Instead of the particle process we work with the height process.  
The limit is first proved for the jam initial condition of TASEP (also called step initial
condition) which for the height process is an initial wedge shape.  
This process can be equivalently represented by the wedge  last-passage model.
 Subadditivity 
gives the limit.   The general case then follows from an envelope 
property   that also leads to the variational representation of the limiting height profile.
In this section we treat the wedge case, and the next section puts it all together.  

Recall the notation and conventions introduced in the previous section.  In particular, 
$c(x)$ is  a positive, lower semicontinuous speed function with only   finitely many discontinuities in any compact set. Define a lattice analogue of the wedge $\cW$ by 
\be\cL=\{(i,j) \in \bZ^2: j\geq1, i\geq -j+1\}\label{defLL}\ee
 with boundary $\partial\cL=\{(i,0):i\geq 0\}\cup\{(i,-i): i<0\}$.

For each $n\in \bN$   construct a last-passage growth model on $\cL$ that represents the TASEP height function in the wedge. Let $\{ \tau^n_{i,j}: (i,j) \in \cL \}_{n\in \bN}$ denote a sequence of independent collections of i.i.d.\ exponential rate $1$ random variables. We need an extra index $\ell$ to denote the shifting. Define weights
\be
\omega_{i,j}^{n,\ell} = c\left(\dfrac{i - \ell}{n}\right)^{-1}, \quad (i,j)\in \mathcal{L}. 
\label{weights}
\ee 
For $\ell \in \bZ$ and  $n \in \bN$ assign to   site $(i,j) \in \cL$ the random variable $\omega_{i,j}^{n,\ell}\tau^n_{i,j}.$ 
Given lattice points $(a,b), (u,v) \in \cL$,   $\Pi((a,b),(u,v))$ is the set of lattice paths $\pi = \{(a,b) = (i_0,j_0),(i_1,j_1),...,(i_p,j_p) = (u,v)\}$ whose admissible steps satisfy 
\be
(i_l,j_l) - (i_{l-1},j_{l-1}) \in \{ (1,0), (0,1), (-1,1)\}.
\label{wedgepaths}
\ee
In the case that $(a,b)=(0,1)$ we simply denote this set by $\Pi(u,v)$. 
For $(u,v)\in \cL$, $\ell \in \bR$ and $n\in \bN$ denote the \textsl{wedge last passage time}
\be
T^{n,\ell}(u,v) = \max_{\pi \in \Pi(u,v)} \sum_{(i,j)\in \pi}\omega^{n,\ell}_{i,j}\tau^n_{i,j}
\label{lastpassagetime}
\ee with boundary conditions
\be
T^{n,\ell}(u,v)=0 \quad \text{for}\quad (u,v)\in \partial\cL.
\label{boundaryoflastpassagetime}
\ee

Admissible 
steps \eqref{wedgepaths}  come from the properties of the TASEP height function.  
Notice that $(0,1)$ is in fact  never used in a maximizing path.   

To describe macroscopic last passage times define,  for $(x,y)\in \cW$ and $q \in \bR$, 
\be
\Gamma^q(x,y)=\sup_{\textbf{x}(\cdot)\in \cH(x,y)}\Big\{ \int_{0}^{1}\frac{\gamma(\textbf{x}'(s))}{c(x_1(s)-q)}\,ds \Big\}.
\label{gammaq}
\ee
 
 \begin{theorem}
For all $q\in \bR$ and $(x,y)$ in the interior of $\cW$
\be
\dlim n^{-1}T^{n,\fl{nq}}(\fl{nx},\fl{ny})=\Gamma^q(x,y) \quad \text{a.s.}
\label{Tlimeq} \ee\label{Tlimit}
\end{theorem} 

\begin{remark} In a constant 
 rate $c$ environment the wedge last passage limit is
\be
\lim_{n\rightarrow\infty}\frac1{n} {T^{n}(\fl{nx},\fl{ny})} = c^{-1}\gamma(x,y)  = c^{-1}\left(\sqrt{x+y}+\sqrt{y}\right)^2.
\label{chomogeneous}
\ee 
The limit  $\gamma(x,y)$
  is   concave, but this is  not true 
 in general for   $\Gamma^0(x,y)$.
In some special cases concavity still holds,
such as   if the function $c(x)$ is nonincreasing if $x<0$ and nondecreasing if $x>0$. 
\end{remark} 

To prove Theorem \ref{Tlimit} we approximate $c(x)$ with step   functions. Let $-\infty = a_1 <a_2< ... < a_{L-1} <a_L = +\infty,$ and consider the lower semicontinuous step function 
\be 
c(x) = \sum_{m=1}^{L-1}r_m \mathbf{1}_{(a_m, a_{m+1})}(x) + \sum_{m=2}^{L-1}\min\{r_{m-1}, r_{m}\}\mathbf{1}_{\{ a_m \}}(x).
\label{simpleratefunction}
\ee    

\begin{proposition}
Let $c(x)$ be given by \eqref{simpleratefunction}. Then limit \eqref{Tlimeq} holds. 
 \label{simpleTlimit}
\end{proposition}

On the way to Proposition \ref{simpleTlimit} we state preliminary lemmas that will be used for
pieces of paths.  We write $c_i$ for the rate values instead of $r_i$ 
to be consistent with the notation in Theorem \ref{two-phaselastpassagetime}. 

\begin{lemma}
Assume that there is a unique discontinuity $a_2 = 0$ for the speed function $c(x)$ in \eqref{simpleratefunction}. Then for $y>0$
\[
\dlim n^{-1}T^{n,0}(0,\fl{ny}) = \dfrac{4y}{\min\{c_1, c_2\}} = \int_{0}^{1}\frac{\gamma(0,y)}{c(0)}\,ds\quad \text{a.s.} 
\]
\label{Verticalpassage}
\end{lemma}

\begin{proof} The upper bound in the limit is immediate from domination with  
constant rates $c(0)$. 

 For the lower bound we spell out the details for the case
 $c_1 \geq c_2$.
  Let $\varepsilon >0$.  To bound $T^{n,0}(0,\fl{ny})$ from below 
  force  the  path  to  go through  points  
$(0,1)$, $ \{(\fl{ny\e}, (k-1)\fl{ny\e}):  k=1,\dotsc, \fl{\e^{-1}}\}$ and  $(0,\fl{ny})$. 
For  $1\le k< \fl{\e^{-1}}$  let $T^n(R^n_k)$ be the last passage time from $(\fl{ny\e}, (k-1)\fl{ny\e})$  to 
  $(\fl{ny\e}, k\fl{ny\e})$.   $R^n_k$ refers to  the parallelogram that contains all the
  admissible paths   from $(\fl{ny\e}, (k-1)\fl{ny\e})$  to 
  $(\fl{ny\e}, k\fl{ny\e})$.  Each $R^n_k$ lies to the right of $x=0$ and therefore 
   in the $c_2$-rate area.  (See Fig.~\ref{lowrectangles}.)

\begin{figure}
\begin{center}
\includegraphics{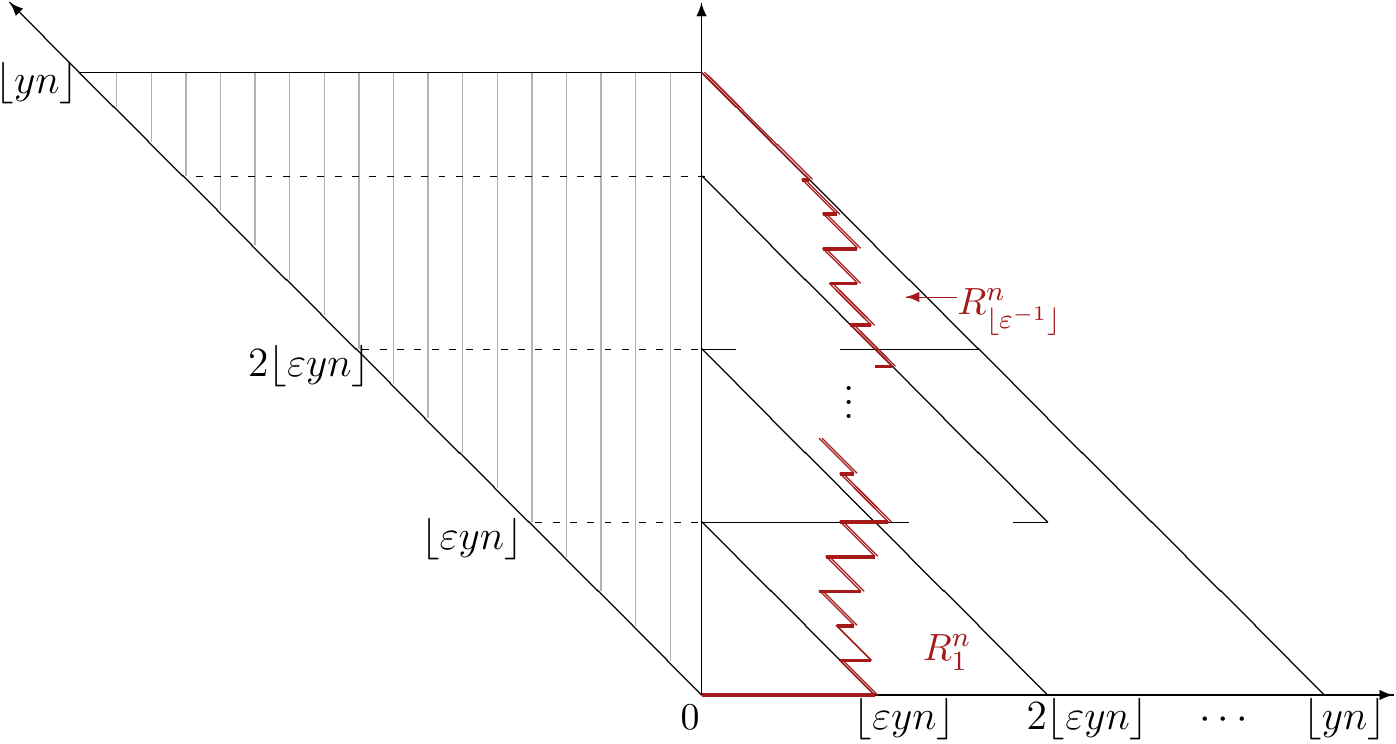}
\caption[Approximation of the last passage time with homogeneous rectangles]{A possible microscopic path   forced to go through opposite corners of the parallelograms $R^n_k$. The striped area left of   $x=0$ is the $c_1$-rate region. }
\label{lowrectangles}
\end{center}
\end{figure}

Let $0< \delta < \e {c_2}^{-1}\gamma(0,y)$. 
A large deviation estimate (Theorem 4.1 in \cite{sepp-large-deviations}) 
gives  a constant $C = C(c_2,y,\e,\delta)$ such that 
\be
\bP\left\{ T^n_{c_2}(R^n_k)\leq n(\e {c_2}^{-1}\gamma(0,y) - \delta) \right\}\leq e^{-Cn^2}.
\label{largedeviations}
\ee
 By a   Borel-Cantelli argument, for large $n$, 
\begin{align*}
T^{n,0}(0,\fl{ny}) &\ge \sum_{k=1}^{\fl{\e^{-1}}-1}T^n(R^n_k) \ge n(\fl{\e^{-1}}-1)
(\e {c_2}^{-1}\gamma(0,y) - \delta).  
\end{align*}
This suffices for the conclusion.  
 \end{proof}
\begin{remark} 
This lemma shows why it is convenient to use a lower semi-continuous speed function. A path that starts and ends at the same discontinuity stays mostly in the low rate region to maximize its weight. This translates macroscopically to the formula for the limiting time constant obtained in the lemma, involving only the value of $c$ at the discontinuity. If the speed function is not lower semi-continuous, we can state the same result using left and right limits.
\end{remark}

\begin{lemma}
Let  $ a = 0 < b< +\infty$
be discontinuities  for the step speed function $c(x)$ and $c(x)=r$ for $a<x<b$.  
Take  $z\in[0,b]$.  Let 
$\wt T^n(\fl{nz}, \fl{ny})$ be the wedge last passage time from $(0,1)$ to $(\fl{nz},\fl{ny})$ subject to the constraint that the   path has to stay in the $r$-rate region $(a,b)\times(0,+\infty)$,
except possibly for the initial and final steps. Then  
\be
\dlim n^{-1}\wt T^n(\fl{nz},\fl{ny}) = r^{-1}{\gamma(z,y)} \quad \text{a.s.} 
\ee
Same statement holds if $b\le z\le 0$.  
\label{Conditionalpassage}
\end{lemma}

\begin{proof}
The upper bound  $\varlimsup n^{-1}\wt T^n(\fl{nz},\fl{ny}) 
\leq {r}^{-1}{\gamma(z,y)}$  is immediate
by  putting constant rates $r$ everywhere and dropping the restrictions
  on the path. For the lower bound
adapt   the steps of the proof in Lemma \ref{Verticalpassage}. \end{proof}

Lemma \ref{Conditionalpassage} is a place where we cannot allow accumulation of discontinuities for the speed function.

\begin{figure}
\begin{center}
\includegraphics{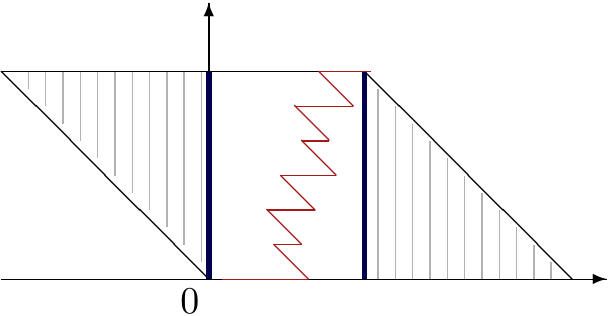}
\caption[Restricted paths in a homogeneous strip ]{A possible microscopic path described in Lemma \ref{Conditionalpassage}. The path has to stay in the unshaded region. }
\label{paths}
\end{center}
\end{figure}

 Before proceeding to the proof of Proposition \ref{simpleTlimit} we make
  a simple but important  observation  about the macroscopic paths 
  $\mathbf{x}(s)= (x^1(s), x^2(s))$, $s\in[0,1]$, in $\cH(x,y)$   
  for the case where  $c(x)$ is a step function \eqref{simpleratefunction}.   
    
\begin{lemma}
There exists a constant $C=C(x,y,c(\cdot-q))$ such that
the supremum in \eqref{gammaq} comes from 
 paths in  $\cH(x,y)$ that consist of 
at most $C$  line segments.  Apart from the first and last
segment, these segments can be of two types:  segments that 
go from one discontinuity of $c(\cdot-q)$ to a neighboring 
discontinuity, and vertical segments along a discontinuity. 
\label{segments}
\end{lemma}

 \begin{proof}   
Path $\mathbf x$  is  a union of subpaths $\{\bx_j\}$  along which  
$c(x_j^1(s)-q)$ 
is constant, except possibly at the endpoints.
  Given such a subpath $(\mathbf x_j(s): {t_j}\le s\le {t_{j+1}})$,  concavity of $\gamma$ and Jensen's inequality imply that the   line segment  
  $\phi_j$ that connects $\bx_j(t_j)$ to  $\bx_j(t_{j+1})$ dominates:  
\be
\int_{t_j}^{t_{j+1}} \frac{\gamma(\bx_j '(s))}{c(x_j^1(s)-q)}\,ds \leq \int_{t_j}^{t_{j+1}} \frac{\gamma(\phi_j '(s))}{c(\phi_j^1(s)-q)}\,ds.
\notag 
\ee
Consequently we can restrict to paths that are unions of line segments.  

To bound the number of line segments, observe first that 
the number of segments that go from one discontinuity to a neighboring
discontinuity is bounded.  The reason is that the restriction
$\bx'(s)\in\cW$ forces  such a segment  to
increase at least one of the coordinates by the distance between
the discontinuities.  

Additionally there can be subpaths that  touch
 the same discontinuity more than once 
without touching a different discontinuity. Lower semi-continuity
of $c(\cdot)$  and Jensen's inequality
show again that the vertical line segment  that stays on the discontinuity 
dominates such a subpath.  
Consequently there can be at most one (vertical) line segment 
between two line segments that connect distinct discontinuities.  
\end{proof} 
 
Next a lemma about the continuity of $\Gamma^q$. We write 
$\Gamma^q((a,b),(x,y))$ for the value in \eqref{gammaq} 
when the paths go from $(a,b)$ to $(x,y)\in (a,b)+\cW$. 

\begin{lemma}
Fix $z, w >0$. Then there exists 
a constant $C= C(z,w, c(\cdot-q))<\infty$ such  that for all 
$0< \delta \leq 1$ and $0 \leq a \leq z$
\be
\Gamma^q((a,0),(z,\delta)) - \Gamma^q((a,0),(z,0)) \leq C\sqrt{\delta},
\label{deltaparineq1}\ee
and for $0 \leq b \leq w$
\be
\Gamma^q((-b,b),(-w,w+\delta)) - \Gamma^q((-b,b),(-w,w)) \leq C\sqrt{\delta}.
\label{deltaparineq2}\ee
\label{deltapar}
\end{lemma}

\begin{proof}
Pick $\delta\in(0,1]$ and consider the point
 $(z,\delta)$ in $\cW$.  For any $\mathbf{x}= (x^1(s),x^2(s)) \in \cH(z,\delta)$ set 
\be 
I(\bx,q) =\int_{0}^{1}\dfrac{\gamma(\bx'(s))}{c(x^1(s)-q)}\,ds.  
\ee 
Let $\e>0 $ and assume that $\phi=(\phi^1,\phi^2)\in \cH(z,\delta)$ is a path such that $\Gamma^q(z,\delta) - I(\phi,q) < \e.$
 Lemma \ref{segments} implies that we can decompose $\phi$ into 
 disjoint linear segments $\phi_j$ so that $\phi= \sum_{j=1}^M \phi_j$
and $\phi_j:[s_{j-1},s_{j}]\to\cW$. Here $\sum_{j}\phi_j$ means path concatenation.

We can find segments $\phi_{j(k)}$, $1\le k\le N$,  such that 
\[\phi^1_{j(k)}(s_{j(k)-1})< \phi^1_{j(k)}(s_{j(k)}), \quad  
\phi^1_{j(k)}(s_{j(k)})= \phi^1_{j(k+1)}(s_{j(k+1)-1}), \]
$\phi^1_{j(1)}(s_{j(1)-1})=0$, and $\phi^1_{j(N)}(s_{j(N)})=z$. 
 In other words, the projections of the  segments  $\phi_{j(k)}$ cover the interval
 $[0, z]$ without overlap and without backtracking.  

We bound the contribution of the remaining path segments to $I(\phi, q)$. 
Let $J$ $=$ \break $\bigcup_{k=1}^{N-1} [s_{j(k)}, s_{j(k+1)-1}]$ be the 
leftover portion of the time interval $[0,1]$.  
The subpath $\phi(s)$, $s\in [s_{j(k)}, s_{j(k+1)-1}]$,
  (possibly) eliminated from between 
$\phi_{j(k)}$ and $\phi_{j(k+1)}$  satisfies 
$\phi^1(s_{j(k)})=\phi^1(s_{j(k+1)-1})$.   
Note that $\gamma(a,b)\le 2a+4b$ for $(a,b)\in\cW$ and 
$\int_0^1 (\phi^2)'(s)\,ds = \delta$. 
  We
 can bound as follows: 
  \begin{align}
  \int_{J}\frac{\gamma((\phi^1)'(s),(\phi^2)'(s))}{c(\phi^1(s)-q)}\,ds 
       &\leq C\int_{J}\gamma((\phi^1)'(s),(\phi^2)'(s))\,ds \notag\\
       &\leq C\int_{J} \big(2(\phi^1)'(s)+4(\phi^2)'(s)\big)\,ds \notag\\
       &\leq C\int_{J} 2(\phi^1)'(s)\,ds +  C\int_{0}^{1}4(\phi^2)'(s)\,ds \notag\\
       &= 0 + 4C\delta.       
  \end{align}
 
Set $t_k = s_{j(k)-1} < u_k = s_{j(k)}$. 
Define a horizontal path $w$ from $(0,0)$ to $(z,0)$ 
with segments 
\be
w_{k}(s)= \big(\phi_{j(k)}^{1}(s), 0\big), \quad \text{for } 
\  t_k \leq s \leq u_k, 
\ee
and constant on the complementary time set $J$.

To get the lemma, we estimate \normalsize{
\begin{align}
\Gamma^q(z&,\delta) -\e \leq I(\phi,q) = \int_J \frac{\gamma(\phi'(s))}{c(\phi^1(s)-q)}\,ds +  \int_{[0,1]\setminus J} \frac{\gamma(\phi'(s))}{c(\phi^1(s)-q)}\,ds\notag \\ 
                   &\leq C\delta + \sum_{k = 1}^{N} \Big(I(\phi_{j(k)},q) - I(w_{k}, q)\Big)+ \Gamma^q(z,0) \notag \\
                   &\leq C\delta + C'\sum_{k=1}^{N}\int_{t_{k}}^{u_{k}} \big(\gamma(\phi'_{j(k)}(s)) - \gamma(w'_{k}(s))\big)\,ds+  \Gamma^q(z,0)\notag \\
                   &\leq C\delta + C'\sum_{k=1}^{N}\bigg(\int_{t_{k}}^{u_k} (\phi^2)'_{j(k)}(s)\,ds +\notag \\ &\phantom{xxxxxxxxxxxxxxx}+2\int_{t_{k}}^{u_k}\sqrt{(\phi^2)'_{j(k)}(s)}\sqrt{(\phi^1)'_{j(k)}(s)+(\phi^2)'_{j(k)}(s)} \,ds\bigg)+\Gamma^q(z,0)\notag \\ 
                    &\leq C\delta + C'\sum_{k=1}^{N}\bigg(\int_{t_{k}}^{u_k} (\phi^2)'_{j(k)}(s)\,ds\bigg)^{\frac{1}{2}}\bigg(\int_{t_{k}}^{u_k}\big((\phi^1)'_{j(k)}(s)+(\phi^2)'_{j(k)}(s)\big) \,ds\bigg)^{\frac{1}{2}}+  \Gamma^q(z,0)\notag \\ 
                    &\leq C\delta + C'\bigg(\sum_{k=1}^{N}\int_{t_{k}}^{u_k} (\phi^2)'_{j(k)}(s)\,ds\bigg)^\frac 12\times \notag \\
                    &\phantom{xxxxxxxxxxxxx}\times\bigg(\sum_{k=1}^{N}\int_{t_{k}}^{u_k}\big((\phi^1)'_{j(k)}(s)+(\phi^2)'_{j(k)}(s)\big) \,ds\bigg)^{\frac{1}{2}}+  \Gamma^q(z,0)\notag \\ 
                    &\leq C\delta + C'\sqrt{\delta}\sum_{k=1}^{N}\int_{t_{k}}^{u_k}\big((\phi^1)'_{j(k)}(s)+(\phi^2)'_{j(k)}(s) \big)\,ds+\Gamma^q(z,0)\notag \\ 
                    &\leq C\delta + C'\sqrt{\delta}\sqrt{z+\delta}+  \Gamma^q(z,0)\notag \\
                    &\leq C\delta + C'\sqrt{\delta}\sqrt{z}+ \Gamma^q(z,0).\notag  
\end{align}}
 The first
 inequality \eqref{deltaparineq1} follows
 for $a = 0$ by letting $\e$ go to 0.  It also follows
 for all $a\in[0, z]$ by shifting the origin to $a$ which 
replaces $z$ with $z-a$.

For the second inequality \eqref{deltaparineq2} the arguments are 
analogous, so we omit them.
\end{proof}

\begin{figure}
\begin{center}
\includegraphics{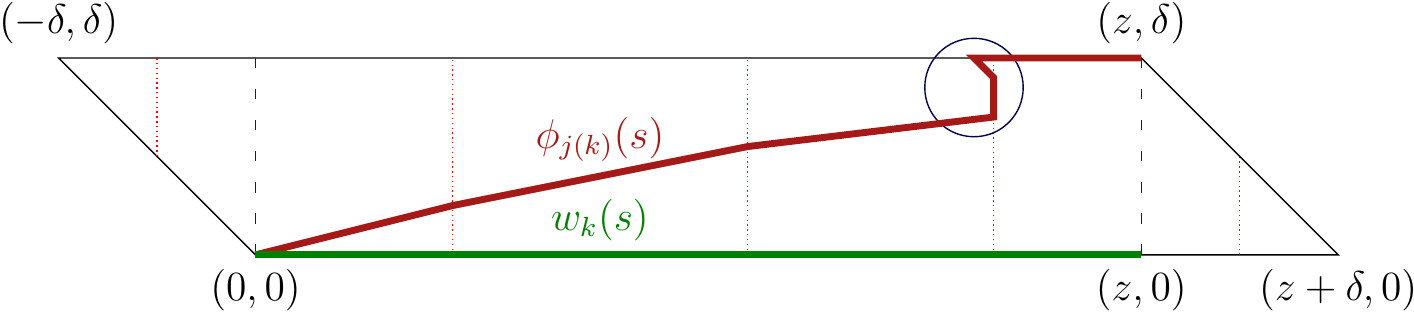}
\caption[Pathological paths in thin strips]{ A possible macroscopic path from $(0,0)$ to $(z,\delta)$. 
 The dotted vertical lines  are discontinuity columns of $c(\cdot-q).$ The error from eliminating the segments outside the two vertical dashed lines and from eliminating  pathologies (like the circled part) is of order $\delta$ and a comparison with the horizontal path leads to an error of order $\sqrt{\delta}$.}
\label{thinstrip}
\end{center}
\end{figure}

\begin{corollary}
Fix $(x,y) \in \cW$.  Then there exists  $C=C(x, y, c(\cdot - q))<\infty$
such that for all
 $0<\delta \le 1 $ 
\be
\Gamma^q(x,y+\delta) - \Gamma^q(x,y) < C\sqrt{\delta}.
\ee
\label{deltapar2}
\end{corollary}

\begin{proof}
Let $A((a,b),(x,y))$ be the parallelogram with sides parallel to
 the boundaries of the wedge, north-east corner the point $(x,y)$ and 
south-west corner at $(a,b)$. If $(a,b) = (0,0)$ we simply write $A(x,y)$.

Let $\e > 0.$ Let $\phi^{\e}$ a path such that $\Gamma^q(x,y+\delta) - I(\phi^{\e}, q)< \e$. Let $u$ be the point where $\phi^{\e}$ 
first intersects the north or the east  boundary of $A(x,y)$.
 Without loss of generality assume it is the north boundary
and so  $u = (a,y)$ for some $a\in[-y,x]$.  Then,
\begin{align}
\Gamma^q(x,y+\delta)-\e &\leq I(\phi^{\e}, q) \notag \\
                        &\leq \Gamma^q(a,y)+\Gamma^q((a,y),(x,y+\delta)) \notag \\
                        &=\Gamma^q(a,y)+\Gamma^q((a,y),(x,y)) + \Gamma^q((a,y),(x,y+\delta))-\Gamma^q((a,y),(x,y))\notag \\
                        &\leq \Gamma^q(x,y)+ \Gamma^q((a,y),(x,y+\delta))-\Gamma^q((a,y),(x,y)). 
\end{align}
The last inequality gives
\be
\Gamma^q(x,y+\delta) - \Gamma^q(x,y)\leq \Gamma^q((a,y),(x,y+\delta))-\Gamma^q((a,y),(x,y)) +\e \leq C\sqrt{\delta} +\e
\ee
by Lemma \ref{deltapar}. Let $\e$ decrease to $0$ to prove the Corollary.
\end{proof}

\begin{proof}[Proof of Proposition \ref{simpleTlimit}]
Fix $(x,y)$ in the interior of $\cW$.  
For   $\mathbf{x}= (x^1(s),x^2(s)) \in \cH(x,y)$ set 
\be 
I(\bx,q) =\int_{0}^{1}\dfrac{\gamma(\bx'(s))}{c(x^1(s)-q)}\,ds.  
\ee 

We prove first 
\be
\varliminf_{n\rightarrow \infty} n^{-1}T^{n,\fl{nq}}(\fl{nx},\fl{ny})\geq \Gamma^q(x,y)
\equiv \sup_{\textbf{x}(\cdot)\in \cH(x,y)} I(\bx,q).    
\ee
It suffices to consider  macroscopic paths   of the type  
\be
\bx(s) = \sum_{j=1}^{H}\bx_j(s)\mathbf{1}_{[s_j,s_{j+1})}(s)
\label{type22}  \ee 
where $H\in \bN$,   $\bx_j$ is the straight line segment 
from  $\bx(s_j)$ to  $\bx(s_{j+1})$,    $c(x_1(s) - q) = r_{m_j}$ is constant for $s\in( s_j, s_{j+1})$, 
and by continuity   $\bx_j(s_{j+1}) = \bx_{j+1}(s_{j+1})$.

 Let $\pi^n$ be the  microscopic path through points 
$(0,1)$,   $ \{\fl{n\mathbf x_j(s_j)}: 1\le j\le H\}$ and  $(\fl{nx},\fl{ny})$
constructed so that its segments $\pi^n_j$  satisfy these requirements: 

(i)   $\pi^n_j$   lies inside the region where $\omega^{n,\fl{nq}}_{i,k} = r_{m_j}^{-1}$ 
is constant, except possibly for the initial and final step;   

(ii) $\pi^n_j$ maximizes passage time between its endpoints $\fl{n\textbf{x}_j(s_j)}$
and  $ \fl{n\bx_{j+1}(s_{j+1})}$ subject to the above requirement.

Let 
\be T^{n,\fl{nq}}_j = \max_{\pi^n_j} \sum_{(i,k) \in \pi^n_j}\omega^{n,\fl{nq}}_{i,k}\tau_{i,k}
\label{Tpieces}
\ee 
denote the  last-passage time of a segment subject to these constraints.  
Observe that the proofs of  Lemmas \ref{Verticalpassage} and  \ref{Conditionalpassage} do not depend on the shift parameter $q$, therefore
\begin{align*}
\dlim n^{-1}T^{n,\fl{nq}}_j = \dfrac{\gamma( \bx_j(s_j) - \bx_{j+1}(s_{j}))}{r_{m_j}}  
                                   = \int_{s_j}^{s_{j+1}}\dfrac{\gamma( \bx'_j(s))}{c(x^1_j(s) - q)}\,ds.
\end{align*}
Adding up the segments gives the lower bound: 
\begin{align*}
\varliminf_{n\rightarrow \infty} n^{-1}T^{n,\fl{nq}}(\fl{nx},\fl{ny})
&\ge  \varliminf_{n\rightarrow \infty} \sum_{j} n^{-1} T^{n,\fl{nq}}_j \\
 &=\sum_{j} \int_{s_j}^{s_{j+1}}\dfrac{\gamma(\bx'_j(s))}{c(x^1(s) -q)}\,ds  
                                                       =I(\bx,q).   
\end{align*} 

Now for the complementary upper bound 
\be
\varlimsup_{n\rightarrow \infty} n^{-1}T^{n,\fl{nq}}(\fl{nx},\fl{ny})\leq \Gamma^q(x,y).
\label{ub7}\ee
Each microscopic path  to $(\fl{nx},\fl{ny})$ 
 is contained in $nA$ for a fixed parallelogram 
$A \subseteq \mathcal{W}$ with sides parallel to the wedge boundaries. 
Pick $\e>0$.
 Let $r_*>0$ be a lower bound on all the rate values that appear in the set $A$. 
Find $\delta>0$ such that   $\abs{\gamma(v) - \gamma(w)}  < \e r_*$ 
 for all $v,w \in A$ with $\abs{v-w} < \delta$ and $\delta \le 1$ 
 so that Corollary \ref{deltapar2} is valid.

Consider an arbitrary microscopic path from $(0,1)$ to 
$(\fl{nx},\fl{ny})$.   Given the speed function and $q$,
there is a fixed upper bound $Q=Q(x,y)$ 
on the number of segments of the path that start at one
discontinuity column $(\fl{na_i} + \fl{nq})\times\bN$ and 
end at a neighboring discontinuity column 
$(\fl{na_{i\pm 1}} + \fl{nq})\times\bN$. The reason
is that  there is an
order $n$ lower bound on the number of lattice steps it takes 
to travel between distinct discontinuities  
in $nA$.

 Fix $K\in\bN$ and partition the interval $[0,y]$ evenly by $b_j=jy/K$, $0\le j\le K$,  
so that $y/K < \delta/Q.$ Make the partition finer by adding 
the $y-$coordinates of the intersection points of discontinuity lines
$\{a_i+q\}\times\bR_+$  with the boundary of $A$. 

Let $\pi^n$ be the maximizing microscopic path.  
We decompose   $\pi^n$ into  path 
segments  $\{\pi^n_{j}: 0\le j<M_n\}$
by looking at visits to   discontinuity columns $(\fl{na_i} + \fl{nq})\times\bN$, 
both repeated visits to the same
column and visits to a column different from the previous one.  
Let  
$\{0=b_{k_0}\leq b_{k_1} \leq b_{k_2} \leq ... \leq b_{k_{M_n-1}} \leq b_{k_{M_n}}= y\}$ 
be a sequence of partition points  and   
 $\{ 0=x_0,\, x_1=a_{m_1}+q,\, x_2=a_{m_2}+q, \dotsc, x_{M_n}=x\}$ 
a sequence where $x_{j}$ for $0<j<M_n$ are discontinuity points of the 
shifted speed function $c(\cdot\,-q)$. 
We can create the path segments and these sequences  
with the property that 
segment $\pi^n_{j}$ starts at  
  $(\fl{nx_{j}}, l)$ with $l$ in the range $\fl{nb_{k_j}} \leq l \leq \fl{nb_{k_{j}+1}}$ and ends at $(\fl{nx_{j+1}}, l')$ with $\fl{nb_{k_{j+1}}} \leq l' \leq \fl{nb_{k_{j+1}+1}}.$
%
In an extreme
case the entire path $\pi^n$ can be a single segment that does  not
touch discontinuity columns.  

In order to have a  fixed  upper bound on the total number 
$M_n$ of   segments, uniformly in $n$, we insist that
for $0<j<M_n-1$ the labels satisfy:

(i) For odd $j$, $\pi^n_j$ starts and ends at the same discontinuity column $(\fl{nx_{{j}}}, \,\cdot\,)$.    The rate relevant for segment $\pi^n_j$
is  $r_{\ell_j}=c(a_{m_j})$. 

(ii) For  even $j$, $\pi^n_j$ starts and ends at different neighboring
discontinuity columns, and  except for the initial and final points,
  does not touch any discontinuity column
 and visits only points that are in a 
region of constant rate $r_{\ell_j}$.  

The above conditions may create empty segments.  This is not harmful.
Replace $Q$ with $2Q+2$ to continue having the uniform upper bound
 $M_n\le Q$.  

Let $T(\pi^n_j)$  be the total weight of segment $\pi^n_j$.
Let $\tilde\pi^n_j$ be the maximal path from  $(\fl{nx_j},\fl{nb_{k_j}})$
to $(\fl{nx_{j+1}},\fl{nb_{k_{j+1}+1}})$ in an environment with constant
weights $\omega_{i,j}=r_{\ell_j}^{-1}$ everywhere on the lattice,
with total weight $ T^n_j$.
 $ T^n_j\geq T(\pi^n_j)$, up to an error from the endpoints
of $\pi^n_j$.

Theorem 4.2 in \cite{sepp-large-deviations} gives
 a large deviation bound for  $ T^n_j$. Consider a constant 
rate $r$ environment and the maximal weight 
$T\bigl((\fl{nu_1},\fl{nv_1}), (\fl{nu_2},\fl{nv_2})\bigr)$
between two points $(u_1,v_1)$ and  $(u_2,v_2)$ such that 
their lattice versions can be connected by admissible paths for
all $n$. 
Then 
there exists a positive constant $C$ such that for $n$ large enough,
\be
\bP\Bigl\{ T\bigl((\fl{nu_1},\fl{nv_1}), (\fl{nu_2},\fl{nv_2})\bigr)
 > n r^{-1} 
\gamma( u_2-u_1 , v_2-v_1 ) + n\e  \Bigr\} < e^{-Cn}.
\label{ld1}
\ee 

There is a fixed finite collection out of which we pick 
the 
pairs $\{(x_j, b_{k_j}), (x_{j+1}, b_{k_{j+1}+1})\}$ that determine
the segments  $\tilde \pi^n_j$.
By \eqref{ld1} and  the Borel-Cantelli lemma, a.s.\ for  large enough $n$, 
\be
T^n_j \; \le\;  n r_{\ell_j}^{-1} 
\gamma( x_{j+1}-x_j , b_{k_{j+1}+1} - b_{k_j})
 + n\e \quad \text{for $0\le j<M_n$.}  
\label{ld3}
\ee  

Define $\delta_1>0$ by  $y+\delta_1=\sum_{j=0}^{M_n-1} (b_{k_{j+1}+1}-b_{k_j})$.
Since $y=\sum_{j=0}^{M_n-1} (b_{k_{j+1}}-b_{k_j})$ and by the choice
of the mesh of the partition $\{b_k\}$, we have 
$\delta_1\le  M_n\delta/Q\le \delta$.  
Think of $( x_{j+1}-x_j , b_{k_{j+1}+1} - b_{k_j})$,
$0\le j<M_n$,  as the successive
segments of a macroscopic  path from $(0,0)$  to $(x,y+\delta_1)$. 

For sufficiently large $n$ so that \eqref{ld3} is in effect,
\begin{align*} 
T^{n,\fl{nq}}(\fl{nx},\fl{ny}) &\le \sum_{j=1}^{M_n} T^n_j 
\le n \sum_{j=1}^{M_n}  r_{\ell_j}^{-1} 
\gamma( x_{j+1}-x_j , b_{k_{j+1}+1} - b_{k_j}) + nQ\e \\
&\le  n \Gamma^q(x,y+\delta_1)+ nQ\e  \\
&\le n\Gamma^q(x,y)+ nC\sqrt{\delta}  + nQ\e. 
\end{align*}
The last inequality came from Corollary \ref{deltapar2}. 
Let $\delta\rightarrow 0$. Since $\e$ was arbitrary the upper bound \eqref{ub7} holds. 
\end{proof}

 \begin{proof}[Proof of Theorem \ref{Tlimit}]  Fix $(x,y)$.  
For each  $\e  > 0$ we can find  lower semicontinuous step functions
$c_1$ and $c_2$ such that  $\norm{c_1-c_2}_\infty\le\e$ and on some
compact interval,  large enough to contain all the rates that can potentially
influence $\Gamma^q(x,y)$, 
   $c_1(x)\leq c(x)\leq c_2(x) $.    When the weights in \eqref{weights}  
   come  from speed function $c_i$  let us write 
  $T_i$ for last passage times and $\Gamma_i$ for their limits. 
An obvious
coupling using common exponential variables $\{\tau_{i,j}\}$ gives 
\[
T_1^{n,\fl{nq}}(\fl{nx},\fl{ny})\geq T^{n,\fl{nq}}(\fl{nx},\fl{ny})\geq T^{n,\fl{nq}}_2(\fl{nx},\fl{ny}).
\]   
Letting $\alpha>0$ denote a lower bound for $c(x)$ in the   compact interval relevant for $(x,y)$,
we have this bound for $\bx\in\cH(x,y)$: 
\begin{align*}
0\le \int_{0}^{1} \Bigl\{ \frac{\gamma(\textbf{x}'(s))}{c_1(x_1(s)-q)}
- \frac{\gamma(\textbf{x}'(s))}{c_2(x_1(s)-q)}\Bigr\} \,ds 
                             &\leq \e \int_{0}^{1}\frac{\gamma(\textbf{x}'(s))}{c^2_1(x_1(s)-q)}\,ds \\ 
                             &\leq \e  \alpha^{-2}\gamma(x,y).  
\end{align*} 
Therefore the limits also have the bound 
\begin{align*}
0\le \Gamma^q_1(x,y) - \Gamma^{q}_2(x,y) \leq C(x,y)\e. 
\end{align*}
From these approximations and the   limits for $T_i$  in Proposition \ref{simpleTlimit}
we can deduce Theorem \ref{Tlimit}. 
\end{proof}

\begin{proof}[Proof of Theorem \ref{two-phaselastpassagetime}]
We can construct the last passage times $G(x,y)$ of the
 corner growth model \eqref{basicref} with the same ingredients as the
wedge last passage times $T^{n,0}(x,y)$
of \eqref{lastpassagetime}, by taking $Y_{(i,j)}=\omega^{n,0}_{i-j,\,j}\tau^n_{i-j,\,j}$. 
Then  $T^{n,0}(x,y)=G(x+y,y)$ and  we can transfer the problem to the wedge.
The correct speed function to use  is now 
   $c(x) = c_1 \mathbf{1}\{ x < 0\}+c_2 \mathbf{1}\{ x\geq 0 \}$.
   In this case the limit in Theorem
\ref{Tlimit} can be solved explicitly with calculus.  We omit the details.  
 \end{proof}

\section{Hydrodynamic limit}
\label{hydrolimit}
In this section we sketch the proof of the main result Theorem \ref{hydro}. This argument is
from \cite{sepp99K, sepp01slow}.  

\subsection{Construction of the process and the variational coupling}
  For each $n \in \bN$ we construct a 
$\bZ$-valued \textsl{height process} $z^n(t)=(z^n_i(t): i \in \bZ)$.
 The height values obey the constraint 
\be
0\leq z^n_{i+1}(t)-z^n_i(t)\leq 1.
\label{constraint1}
\ee
Let $\{ \cD^n_i \}$ be a collection of mutually independent (in $i$ and $n$) Poisson processes with rates $c_i^n$ given by 
\be
c_i^n = c( n^{-1} i),
\label{discrete}
\ee
where $c(x)$ is the lower semicontinuous speed function. Dynamically, for each $n$ 
and  $i$, the height value $z^n_i$ is decreased by $1$ at  event times of $\cD^n_i$,
provided    the new configuration  does not violate \eqref{constraint1}. 
 
After we construct $z^{n}(t)$, we can define the exclusion process $\eta^{n}(t)$ by 
\be
\eta^n_i(t)= z^n_i(t)-z^n_{i-1}(t).
\label{exclusion-current}
\ee
A decrease in $z^n_i$ is associated with an exclusion  particle jump   from site $i$ to   $i+1$.  
Thus the $z^n$  process 
keeps track of the current of the $\eta^n $-process, precisely speaking  
 \be
J^n_i(t)=z^n_i(0)-z^n_i(t).
\label{zcurrent}
\ee 

Assume that the processes $z^n$ have  been constructed on a probability space that supports the initial configurations $z^n(0) = (z^n_i(0))$ and the Poisson processes $\{\cD^n_i\}$ that are independent of $(z^n_i(0))$.   Next we state the envelope property that is the key tool for the 
proof of the hydrodynamic limit.  Define a family of auxiliary height processes 
$\{\xi^{n,k}: n\in\bN,\,k\in\bZ\}$ that grow upward from  wedge-shaped initial conditions
\be
\xi_i^{n, k}(0) =\left\{
\begin{array}{ll}

\vspace{0.1 in}
0, & \textrm {if } i \geq 0  \\

-i, & \textrm{if } i < 0.
\end {array}
\right.
\label{xi-initial}
\ee
The dynamical rule for the $\xi^{n, k}$ process is that $\xi^{n,k}_i$ jumps up by 1
at the event times of $\cD^{n}_{i+k}$ provided the inequalities 
\be
\xi^{n,k}_i\leq\xi^{n,k}_{i-1} \quad \text{and} \quad \xi^{n,k}_i\leq\xi^{n,k}_{i+1}+1 
\label{xiineq}  
\ee
are not violated. In particular  $\xi^{n,k}_i$ attempts a jump at rate $c^n_{i+k}$.

\begin{lemma}[Envelope Property]
For each $n \in \bN$, for all $i\in \bZ$ and $t\geq 0$,
\be
z^n_i(t) = \sup_{k\in \bZ}  \{ z^n_k(0) - \xi^{n, k}_{i-k}(t)\} \quad a.s.
\label{envelope2}
\ee
\end{lemma}
Equation \eqref{envelope2}  holds by construction at time $t=0$, and it is
proved by induction on jumps. For  details see Lemma 4.2 in \cite{sepp99K}.

\subsection{The limit for $\xi$}
For $q,x \in \bR$, $t>0$ and for the speed function $c(x)$, define 
\be
g^{q}(x,t)= \inf\left\{ y: (x,y)\in \cW, \Gamma^q(x,y)\geq t  \right\}.
\label{gq}
\ee
$\Gamma^q(x,y)$ defined by \eqref{gammaq}  represents the macroscopic time it takes a $\xi$-type interface process 
to reach point $(x,y).$ 
The level curve of $\Gamma^q$ given by $g^q(\cdot,t)$ represents the limiting interface of a certain $\xi$-process,
as stated in the next proposition.
\begin{proposition}
For all $q,x \in \bR$ and $t>0$
\be
\dlim n^{-1}\xi^{n, \fl{nq}}_{\fl{nx}}(nt)=g^{-q}(x,t) \quad a.s.
\label{xilimit}
\ee 
\label{xilimit!}
\label{proxilimit}
\end{proposition}

Recall the lattice wedge  $\mathcal{L}$  defined by \eqref{defLL}.  
 For $(i,j)\in \cL\cup\partial\cL$, let 
\be
L^{n, k}(i,j)= \inf\{t\geq 0: \xi^{n,k}_i(t)\geq j\}
\ee
denote the time when $\xi^{n,k}_i$ reaches level $j$. The rules \eqref{xi-initial}--\eqref{xiineq} give the boundary conditions 
\be
L^{n,k}(i,j) = 0 \quad \text{for}\quad (i,j)\in \partial\cL
\label{Lboundary}
\ee and for $ (i,j) \in \cL$  the recurrence 
\be
L^{n,k}(i,j)= \max\{ L^{n,k}(i-1,j), L^{n,k}(i,j-1), L^{n,k}(i+1,j-1) \}+\beta^{n,k}_{i,j} 
\label{L}
\ee
where $\beta^{n,k}_{i,j}$ is an exponential waiting time, independent of everything else. It represents the time $\xi^{n,k}_i$ waits to jump, \textit{after} $\xi^{n,k}_i$ and its neighbors $\xi^{n,k}_{i-1}$, $\xi^{n,k}_{i+1}$ have reached positions that permit $\xi^{n,k}_i$ to jump
from $j-1$  to $j$. The dynamical rule that governs the jumps of $\xi^{n,k}_i$ implies that  $\beta^{n,k}_{i,j}$ has rate $c^n_{i+k}$.

Equations \eqref{lastpassagetime}, \eqref{boundaryoflastpassagetime}, \eqref{Lboundary},
and  \eqref{L}, together with the strong Markov property,  imply that 
\be
\{L^{n,k}(i,j):(i,j) \in \cL\cup\partial\cL \} \, \overset{\cD}= \, \{T^{n,-k}(i,j):(i,j) \in \cL\cup\partial\cL \}.
\label{L=-T}
\ee 
Consequently Theorem \ref{Tlimit}  gives 
the a.s.\ convergence $n^{-1} L^{n,\fl{nq}}(\fl{nx},\fl{ny})\to \Gamma^{-q}(x,y)$, 
and this passage time
limit gives limit \eqref{xilimit}.  

\begin{proof}[Proof of Theorem \ref{hydro}]
Given the initial configurations $\eta^n(0) = \left\{ \eta^n_i(0): i \in \bZ \right\}$ that appear in hypothesis \eqref{weaklaw}, define initial configurations $z^n(0) = \left\{ z^n_i(0): i \in \bZ \right\}$ so that $z^n_0(0)=0$ so that \eqref{exclusion-current} holds at time $t=0.$ Hypothesis \eqref{weaklaw} implies that
\be
\lim_{n\rightarrow \infty} n^{-1} z^n_{\fl{nq}} = v_0(q) \quad \textrm{a.s.} 
\label{weakz}
\ee  
for all $q \in \bR$, with $v_0$ defined by \eqref{vdef}.

Construct the height processes $z^n$ and define the exclusion processes $\eta^n$ by \eqref{exclusion-current}. Define $v(x,t)$ by \eqref{velocityversion1}. From 
\eqref{exclusion-current}--\eqref{zcurrent} we see that Theorem \ref{hydro}
 follows from proving that for all $x\in \bR, t\in \bR^+,$ 
\be
\lim_{n\rightarrow \infty}n^{-1}z^n_{\fl{nx}}(nt)= v(x,t) \quad \textrm{a.s.}
\label{z-v}
\ee  
Rewrite \eqref{envelope2} with the correct scaling:
\be
n^{-1}z^n_{\fl{nx}}(nt)= \sup_{q \in \bR}\left\{ n^{-1}z^n_{\fl{nq}}(0) - n^{-1}\xi^{\fl{nq}}_{\fl{nx}-\fl{nq}}(nt) \right\}.
\label{scaledz}
\ee

The proof of \eqref{z-v} is now to show that the right-hand side of \eqref{scaledz} converges to the right-hand side of \eqref{velocityversion1}.
 
From \eqref{weakz}, \eqref{scaledz} and \eqref{xilimit} we can prove that a.s.
\be
\lim_{n\rightarrow \infty}n^{-1}z^n_{\fl{nx}}(nt)= \sup_{q \in \bR}\left\{ v_0(q)-g^{-q}(x-q, t) \right\}\equiv \tilde{v}(x,t).
\label{tildev}
\ee 
The argument is the same as the one from equations (6.4)--(6.15) in \cite{sepp99K} so we will not repeat it here. 

Using \eqref{gammaq} and \eqref{gq} we can rewrite $\tilde{v}(x,t)$ as
\be
\tilde{v}(x,t)= \sup_{q,y \in \bR}\Big\{ v_0(q)-y: \exists \bx \in \cH(x-q,y) \textrm{ such that } \int_0^1\frac{\gamma(\bx'(s))}{c(x_1(s) + q)}\,ds \geq t \Big\}.
\label{massagedvtilde}
\ee  
The final step is to prove $v(x,t) = \tilde{v}(x,t).$ The argument is identical to the one used to prove Proposition 4.3 in \cite{sepp01slow} so we omit it.   With this we can consider 
Theorem \ref{hydro} proved.  
\end{proof}
  
\section{Density profiles in two-phase TASEP}
\label{Density}
This section proves Corollary \ref{densityprofiles}: assuming  $c(x) =(1- H(x))c_1 + H(x)c_2$, $c_1 \geq c_2$ and $\r_0(x) \equiv \r \in (0,1)$,  we   use  variational formula \eqref{velocityversion1}
to obtain   explicit hydrodynamic limits. 
\begin{remark}
In light of Theorem \ref{pdecor}, one can (instead of doing the following computations) guess the candidate solution for the scalar conservation law  \eqref{Breakscl} and then check that it verifies the entropy conditions \eqref{Eb1} - \eqref{Eb3}. The following computations do not require any knowledge of p.d.e.\ theory or familiarity with interface problems so we present them independently in this section.
\end{remark}
 Let 
\be
C^0(x,t,q) = \left\{w\in C( [0,t], \bR):  w \text{ piecewise linear, } w(0)=q, w(t)=x \right\}.  
\ee
To   optimize in \eqref{velocityversion1} we use a couple
different approaches for different cases. We outline this and omit the details.  

One approach is to separate the choice of the starting point $q$ of the path. 
By setting  
\be
I(x,t,q) = \inf_{w \in C^0(x,t,q)}\bigg\{ \int_{0}^{t}c(w(s))g\left( \dfrac{w'(s)}{c(w(s))} \right)\,ds \bigg\}
\label{Ixtq}
\ee
  \eqref{velocityversion1} becomes\be
v(x,t) = \sup_{q \in \bR}\left\{ v_0(q) - I(x,t,q) \right\}.
\label{spoptimize}
\ee
We  distinguish four cases according to the signs of $x,q$.  Set 
\be
R_{+}(x,t) = \sup_{q > 0}\left\{ v_0(q) - I(x,t,q) \right\},\textrm{ if } x>0,\quad
\ee
\be
L_{-}(x,t) = \sup_{q < 0}\left\{ v_0(q) - I(x,t,q) \right\},\textrm{ if } x<0.
\ee 
These functions are going to be used in Cases 1 and 2 below ($qx \geq 0$) where we can 
compute $I(x,t,q)$ directly. 

However, there are values  $(x,t, q)$ for which the $q$-derivative of the 
expression in   braces in \eqref{spoptimize} is a rational function with a quartic polynomial in the numerator. While an explicit formula for   roots of a quartic  exists, the solution is not attractive and
it is not clear how to pick the right root. Instead we turn the problem into a two-dimensional maximization problem. 

If $qx < 0$    the optimizing path $w$ crosses the origin:
$w(u)= 0$ for some $u$.  It turns out convenient  
to find the optimal $q$ for each crossing time $u$. For Case 3 ($q<0,x>0$) set
\be
\Phi(u, q) = q\r - c_1ug\left(\dfrac{-q}{uc_1}\right) - c_2(t-u)g\left( \dfrac{x}{(t-u)c_2} \right)
\ee 
and 
\be
L_{+}(x,t) = \displaystyle\sup_{q < 0, u \in[0,t]}\Phi(u,q).  
\ee
For Case 4 ($q>0,x<0$)   the obvious modifications are 
\be
\Psi(u, q) = q\r - c_2ug\left(\dfrac{-q}{uc_2}\right) - c_1(t-u)g\left( \dfrac{x}{(t-u)c_1} \right)
\ee 
and
\be
R_{-}(x,t) = \displaystyle\sup_{q > 0, u \in[0,t]}\Psi(u,q).  
\notag
\ee
Rewrite \eqref{spoptimize} using   functions $R_{\pm}$, $L_{\pm}$: 
\be
v(x,t)= \max\{ R_{+}(x,t), L_{+}(x,t)\}\mathbf{1}\{x\geq 0\} + \max\{ R_{-}(x,t), L_{-}(x,t)\}\mathbf{1}\{x < 0\}.
\label{R-Lvelocity} 
\ee

\begin{proof}[Proof of Corollary \ref{densityprofiles}\\]

We  compute the functions  $R_{\pm}$, $L_{\pm}$.
 The density profiles $\r(x,t)$ are given then by the $x$-derivative of $v(x,t)$.  
\bigskip

\noindent\textsl{\textbf{Case 1}: $x\geq 0$, $q \geq 0$.}
Since $c_2 \leq c_1$, the minimizing $w$ of $I(x,t,q)$ is the straight line connecting $(0,q)$ to $(t,x).$ In particular, 
\be I(x,t,q)=c_2tg\left(\frac{x-q}{tc_2}\right).\ee Then the resulting $R_{+}(x,t)$ is given by
\be
R_{+}(x,t) =\begin{cases}

-tc_2g(\frac{x}{tc_2} ) &\textrm{if } \r\leq \frac{1}{2},\quad x< tc_2(1-2\r)\\[4pt]

\r x -tc_2\r(1-\r), & \textrm{if } \r\leq \frac{1}{2},\quad x\geq tc_2(1-2\r)\\[4pt]

\r x -tc_2\r(1-\r), & \textrm{if } \r > \frac{1}{2}, 
\end{cases}
\ee 

\noindent\textsl{\textbf{Case 2}: $x\leq 0$, $q \leq 0$.}
The minimizing path $w$ can either be a straight line from $(0,q)$ to $(t,x)$ or a piecewise linear path such that the set $\{t: w(t) = 0\}$ has positive Lebesgue measure. This last statement just says that the path might want to take advantage of the low rate at $x=0$. We leave the calculus details to the reader and record the resulting minimum value of $I(x,t,q)$.  Set   $B = \sqrt{c_1(c_1-c_2)}$.  
\be
I(x,t,q) =
\begin{cases}
\vspace{0.07in}
\frac{-qc_1}{4B}\left(1-\frac{B}{c_1} \right)^2 + \left( t-\frac{|x|-q}{B} \right)\frac{c_2}{4} -\frac{xc_1}{4B}\left(1+\frac{B}{c_1} \right)^2,\\

 \quad\quad\quad\quad\quad\quad\quad\quad  \text{when } -(\sqrt{Bt}-\sqrt{|x|})^2 \leq q,\, -Bt\leq x< 0\quad\\ 

c_1tg\left(\frac{x-q}{c_1t}\right) \quad\quad\quad\text{otherwise }
\end{cases}
\ee 
 The corresponding function $L_{-}(x,t)$ is given by\be
L_{-}(x,t) =
\begin{cases}

\vspace{0.07 in}
\r x -tc_1\r(1-\r), &0< \r < \r^*, x\in \bR\\

\vspace{0.07 in}
\r x -tc_1\r(1-\r), &\r^*\leq\r\leq \frac{1}{2},\quad x\leq -tc_1(\r-\r^*)\\ 

\vspace{0.07 in}
-\left( t+\frac{x}{B} \right)\frac{c_2}{4} 
+  \frac{xc_1}{4B}\left(1+\frac{B}{c_1} \right)^2, & \r^*\leq\r\leq \frac{1}{2},\quad x > -tc_1(\r-\r^*)\\  

\vspace{0.07 in}
\r x -tc_1\r(1-\r) ,& \frac{1}{2} < \r \leq 1 - \r^*,\quad x< -tc_1(\r-\r^*) \\ 

\vspace{0.07 in}
-\left( t+\frac{x}{B} \right)\frac{c_2}{4} 
+  \frac{xc_1}{4B}\left(1+\frac{B}{c_1} \right)^2, & \frac{1}{2} < \r \leq 1 - \r^*,\quad -tc_1(\r-\r^*)\leq x \\

\vspace{0.07 in}
-\left( t+\frac{x}{B} \right)\frac{c_2}{4} 
+  \frac{xc_1}{4B}\left(1+\frac{B}{c_1} \right)^2, &1 - \r^* < \r<1,\quad -Bt\leq x \\

\vspace{0.07 in}
-tc_1g\big( \frac{x}{tc_1} \big), & 1-\r^* < \r<1, \quad  -c_1t(2\r-1) \leq x < -Bt \\

\vspace{0.07 in}
\r x - c_1t\r(1-\r), & 1-\r^* < \r<1, \quad  x < -c_1t(2\r-1) \\

\end{cases}
\ee 

\medskip

\noindent\textsl{\textbf{Case 3}: $x > 0$, $q \leq 0.$}
Abbreviate $ D = c_2^2 - 4c_1c_2\r(1-\r)$. First compute the $q$-derivative\be
\Phi_q(u,q) =
\begin{cases}

\vspace{0.07 in}
\r - \frac{1}{2} - \frac{q}{2uc_1}, & -uc_1 \leq q < 0 \\ 

\r & q < -uc_1.
\end{cases}
\ee 
If $\r\geq 1/2$ then $\Phi_q$ is positive and the maximum value is when $q=0$ so we are reduced to Case $1$. If $\r< 1/2$ the maximizing $\displaystyle q = uc_1\left(2\r - 1\right).$
Then 
\be F(u) = \Phi\bigl(u, 2uc_1(\r - \tfrac{1}{2})\bigr) = -uc_1\r(1-\r) -c_2(t-u)g\left( \dfrac{x}{(t-u)c_2} \right),\notag\ee with $u$-derivative 
\be
\frac{dF}{du} =  -c_1\r(1-\r)+ \frac{c_2}{4}\left(1-\frac{x^2}{(c_2(t-u))^2}\right).
\notag\ee 
Again we need to split two cases. 
If $\r < \r^*$ (equivalently $D>0$) and $x\leq t\sqrt{D}$, the maximizing $u = t- x/\sqrt{D},$ otherwise $u=0$. If $\r^* \leq \r < \frac{1}{2}$ the derivative is negative so the maximizing $u$ is still $u=0.$  Together,
\be
L_{+}(x,t) =
\begin{cases}

\vspace{0.07 in}
-tc_1\r(1-\r)+x\Big(\frac{1}{2}-\frac{\sqrt{D}}{2c_2} \Big), & \r<\r^*, x\leq t\sqrt{D}\\

\vspace{0.07 in}
-c_2tg\big(\frac{x}{tc_2}\big) , &  \r<\r^*, x\geq t\sqrt{D}\\ 

-c_2tg\big(\frac{x}{tc_2}\big) , & \r^*\leq\r \leq 1. 

\end{cases}
\ee 

\medskip

\noindent\textsl{\textbf{Case 4}: $x \leq 0, q \geq 0.$}
We treat this case in exactly the same way as Case 3, so we omit the details. Here we need the quantity $D_1=(c_1)^2 - 4c_1c_2\r(1-\r)$ and we compute 
\be
R_{-}(x,t) =
\begin{cases}

\vspace{0.07 in}
-tc_1g\big(\frac{x}{tc_1}\big), & \r\leq\frac{1}{2}\\

\vspace{0.07 in}
-tc_2\r(1-\r)+x\Big( \frac{1}{2}+\frac{\sqrt{D_1}}{2c_1} \Big) , & \frac{1}{2}<\r,- t\sqrt{D_1}\leq x\\

-tc_1g\big(\frac{x}{tc_1}\big) , & \frac{1}{2} \leq \r \leq 1, x<- t\sqrt{D_1}\\ 
\end{cases}
\ee
Now compute $v(x,t)$ from \eqref{R-Lvelocity}. We leave the remaining details to the reader.
\end{proof}

\section{Entropy solutions of the discontinuous conservation law}
\label{pdes}

For this section, $c(x)=(1-H(x))c_1 + H(x)c_2$, $h(\r)=\r(1-\r)$ and set $F(x,\r) = c(x)h(\r)$ for the flux function of the scalar conservation law \eqref{scl} and $\widetilde{F}(x,\r) =c(x)f(\r)$ for the flux function of the particle system, where $f$ is given by \eqref{fnot}. (The difference between
$F$ and $\wt F$ is that the latter is $-\infty$ outside $0\le \rho\le 1$.) 
 
In \cite{adim-gowd-03} the authors prove that there exists a solution to the corresponding Hamilton-Jacobi equation
\be
\left\{
\begin{array}{lll}

\vspace{0.1 in}
V_t+c_1h(V_x) =0, & \textrm{if } x <0, t>0 \\

\vspace{0.1 in}
V_t+c_2h(V_x) =0, & \textrm {if } x > 0, t>0  \\

V(x,0)=V_0(x)
\end {array}
\right.
\label{HJ}
\ee 
such that $V_x$ solves the scalar conservation law \eqref{scl} with flux function $F(x,\rho)$ and $V_x$ satisfies the entropy assumptions $(E_i), (E_b)$. $V(x,t)$ is given by 
\be
V(x,t)=\sup_{w(\cdot)}\left\{  V_0(w(0)) + \int_{0}^{t} (c(w(s))h)^*(w'(s)) \,ds \right\},
\label{HJsol}
\ee 
where the supremum is taken over  piecewise linear paths $w\in C([0,t],\bR)$ that satisfy $w(t)=x.$

To apply the results of \cite{adim-gowd-03} to the profile
$\rho(x,t)$ coming from our hydrodynamic limit, we only need
to show that the variational descriptions match, in other words
that  we can replace $F$ with $\widetilde{F}$ and 
the solution is still the same.

\begin{proof}[Proof of Theorem \ref{pdecor}]
Convex duality gives  $\left(c(x)f\right)^*(y) = 
c(x)f^*\left( y/c(x)\right)$ and 
so we can rewrite  \eqref{velocityversion1} as 
\be
v(x,t)=\sup_{w(\cdot)}\left\{  v_0(w(0)) + \int_{0}^{t} (c(w(s))f)^*(w'(s)) \,ds \right\}.
\label{TNsol}
\ee 
Observe that for all $y \in \bR$
\be
(c(x)f)^*(y) \geq  (c(x)h)^*(y),
\label{flux-comparison}
\ee
with equality if and only if $y \in [-c_1,c_2]$
Since the supremum in \eqref{HJsol} and \eqref{TNsol} is taken over the same set of paths, \eqref{flux-comparison} implies that
\be
V(x,t) \leq v(x,t).
\label{v-comparison}
\ee 
The proof of the theorem is now reduced to proving that the supremum in \eqref{TNsol} is achieved when $w'(s)c(w(s))^{-1} \in [-1,1],$ giving $V(x,t)=v(x,t)$.

To this end we rewrite $v(x,t)$ once more,
this time as
\[
v(x,t)= \max\{ R_{+}(x,t), L_{+}(x,t)\}\mathbf{1}\{x\geq 0\} + \max\{ R_{-}(x,t), L_{-}(x,t)\}\mathbf{1}\{x < 0\} 
\]
where the functions $R_{\pm}$, $L_{\pm}$ 
(as in the proof of Corollary \ref{densityprofiles})
are defined by 
\be
R_{+}(x,t) = \sup_{q > 0}\left\{ v_0(q) - I(x,t,q) \right\},\textrm{ if } x>0,\quad
\ee
\be
L_{-}(x,t) = \sup_{q < 0}\left\{ v_0(q) - I(x,t,q) \right\},\textrm{ if } x<0,
\ee 
where $I(x,t,q)$ is as in  \eqref{Ixtq}, 
and 
\be
L_{+}(x,t) = \displaystyle\sup_{q < 0, u \in[0,t]} \bigg\{ v_0(q) - c_1u g\Big(\frac{-q}{uc_1}\Big) - c_2(t-u)g\Big(\frac{x}{(t-u)c_2} \Big)\bigg\} \quad \textrm{if } x\geq 0,
\label{finalL}
\ee
and 
\be
R_{-}(x,t) = \displaystyle\sup_{q > 0, u \in[0,t]}\bigg\{ v_0(q) - c_2ug\left(\dfrac{-q}{uc_2}\right) - c_1(t-u)g\left( \dfrac{x}{(t-u)c_1} \right) \bigg\}, \quad x\le 0.  
\ee

 It suffices to show that the suprema
 that define  $R_{\pm}$, $L_{\pm}$
are achieved when \be w'(s)c(w(s))^{-1} \in [-1,1].
\label{slopecondition}
\ee 
We show this for $L_{+}$. The remaining cases are similar. In \eqref{finalL},
 as before, $u$ is the time for which $w(u) = 0.$ 
Let $\Phi(u,q)$ denote the expression in braces in \eqref{finalL}
with  $q$-derivative
\be
\Phi_q(u,q) =
\begin{cases}

\vspace{0.07 in}
\r_0(q) - \frac{1}{2} - \frac{q}{2uc_1}, &-uc_1 \leq q < 0 \\ 

\r_0(q), & q < -uc_1.
\end {cases}
\label{qder}
\ee 

Observe that if $\Phi_q(u,q) = 0$ for some $q^* = q^*(u)$ then also $q^*$ maximizes $\Phi$.
Otherwise the maximum is achieved at $0$ and we are reduced to a different case. Assume that $q^*$ exists. Then by \eqref{qder}
\be
\frac{-q^*}{u} = (1- 2\r_0(q^*))c_1 < c_1.
\label{slope1}
\ee  
Therefore, the slope of the first segment of the maximizing path $w$ satisfies \eqref{slopecondition}. 

The slope of the second segment is $x(t-u)^{-1}.$ Assume  that the piecewise linear path $w$ defined by $u$ and $q^*$ is the one that achieves the supremum. Also assume $u > t-xc_2^{-1}.$ Consider the path $\tilde{w}$ with $\tilde{w}(0) = q^*$ and $\tilde{w}(t-xc_2^{-1})= 0$. Since $g$ is decreasing, we only increase the value of $\Phi$. Hence the supremum that gives $L_{+}$ cannot be achieved on $w$ and this gives the desired contradiction.
\end{proof}

\appendix       

\chapter{Basic Facts }\label{appendix}
\setcounter{theorem}{0}

In this section we report all basics facts from analysis, special functions and probability theory used throughout the thesis.

\section{Special functions and distributions}

The gamma function is 
\be 
\Gamma(s)=\int_{0}^{\infty} x^{s-1}e^{-x}\,dx.
\notag
\ee 
We only use it for positive real values of $s$.

The logarithm $\log \Gamma(s)$ is convex and infinitely differentiable on $(0,\infty)$. The derivatives are called polygamma functions 
$\Psi_n(s)= (d^{n+1}/ds^{n+1})\log \Gamma(s)$, defined for $n\in \bZ_+$. For $n\ge 1$, $\Psi_n$ is nonzero and has sign $(-1)^{n-1}$ throughout $(0,\infty)$. In particular, $\Psi_{0}(s)$ is strictly increasing and has a vertical asymptote at $s=0$. It can be given by
\be
\Psi_0(1+x) = -\gamma + \sum_{k=1}^{\infty}\frac{x}{k(x+k)}. 
\label{telehar}
\ee

One way to compute the limit \eqref{L'Hlimit} is by multiple uses of L' Hospital's rule and then an asymptotic analysis for $\Psi_1(s)$ for $s \rightarrow 0$. For the asymptotic analysis, we need
\be
\Psi_1(s) = \sum_{k=0}^{\infty}\frac{1}{(k+s)^2}.
\notag
\ee

 The Gamma$(\theta, 1)$ distribution has density 
\be
\Gamma(\theta)^{-1}x^{\theta-1}e^{-x}, \quad \theta > 0.
\label{density}
\ee 
on $\bR_+$, mean $\theta$ and variance $\theta$.

Throughout the dissertation we make use of the digamma and trigamma functions $\Psi_0, \Psi_1$ since for $X \sim$ Gamma$(\theta, 1)$ we have
\be
\Psi_0(\theta)= \bE(\log X) \quad \textrm{and} \quad \Psi_1(\theta)= \bV ar(\log X).
\ee

\section{Convex Analysis} 

 For given functions $f(x), g(x)$ we denote the convex dual
\be
f^*(r)= \sup_{x}\{rx-f(x) \},
\label{stardef}
\ee and the infimal convolution
\be
(f\square g)(x)=\inf_{y }\{f(y)+g(x-y)\}.
\label{convo}
\ee 
For lower semi-continuous convex $f$ and $g$ we have 
\be(f\square g)^{*} = f^{*}+g^{*},
\label{addual}
\ee
and double convex duality 
\be
f^{**} = f.
\label{dcd}
\ee
Also, convexity of $f$ implies that on the set $\{|f| < \infty \}$, $f$ is a.e.\ differentiable with
\be
f'(x) = \arg\max\{xu - f^{*}(u) \}.
\label{inverse}
\ee 

\section{Large deviations}
Here are some basic theorems from the theory of large deviations that we are using throughout. The limiting log-moonet generating function is given by
\be
M(u) = \lim_{n\to \infty}n^{-1}\log \bE(e^{u \sum_{i=1}^n X_i})
\label{dianeishere}
\ee

\begin{theorem}[Cram{\'e}r's Theorem]
Let $\{X_n\}_n$ be a sequence of i.i.d. real-valued
random variables. Let $\mu_n$ be the law of the sample mean $S_n/n$. Then, the
large deviation principle $LDP(\mu_n,  n, I)$ is satisfied with $I$ defined by
\be
I(a) = \sup_{u\in \bR}\{ au - M(u)\},
\ee
where $M(u)$ is the limiting log-moment generating function given by \eqref{dianeishere}.
\end{theorem}

\begin{theorem}[One sided Cram{\'e}r's Theorem]
Let $\{X_n\}_n$ be a sequence of i.i.d. real-valued
random variables. Define the one sided rate functions by 
\be
J(a) = \lim_{n \to \infty}n^{-1} \log \bP\{ S_n \ge na \},
\ee

\be
I(a) = \lim_{n \to \infty}n^{-1} \log \bP\{ S_n \le na \}.
\ee
The the two functions are given respectively by 
\be
J(a) = \sup_{u \ge 0}\{ au - M(u)\},
\ee
\be
I(a) = \sup_{u \le 0}\{ au - M(u)\}.
\ee
where $M(u)$ is given by \eqref{dianeishere}
\end{theorem}

%
%

\begin{thebibliography}{10}
\bibitem{adim-gowd-03}
Adimurthi and G.~D.~Veerappa Gowda.
\newblock Conservation law with discontinuous flux.
\newblock {\em J. Math. Kyoto Univ.}, 43(1):27--70, 2003.

\bibitem{adim-etal-05}
Adimurthi, Siddhartha Mishra, and G.~D.~Veerappa Gowda.
\newblock Optimal entropy solutions for conservation laws with discontinuous
  flux-functions.
\newblock {\em J. Hyperbolic Differ. Equ.}, 2(4):783--837, 2005.

\bibitem{MR2356208}
Adimurthi, Siddhartha Mishra, and G.~D. Veerappa~Gowda.
\newblock Explicit {H}opf-{L}ax type formulas for {H}amilton-{J}acobi equations
  and conservation laws with discontinuous coefficients.
\newblock {\em J. Differential Equations}, 241(1):1--31, 2007.

\bibitem{audu-pert-05}
Emmanuel Audusse and Beno{\^{\i}}t Perthame.
\newblock Uniqueness for scalar conservation laws with discontinuous flux via
  adapted entropies.
\newblock {\em Proc. Roy. Soc. Edinburgh Sect. A}, 135(2):253--265, 2005.

\bibitem{Bahadoran-98}
C.~Bahadoran.
\newblock Hydrodynamical limit for spatially heterogeneous simple exclusion
  processes.
\newblock {\em Probab. Theory Related Fields}, 110(3):287--331, 1998.

\bibitem{baha-04-aop}
Christophe Bahadoran.
\newblock Blockage hydrodynamics of one-dimensional driven conservative
  systems.
\newblock {\em Ann. Probab.}, 32(1B):805--854, 2004.

\bibitem{Ben-Ari}
Iddo Ben-Ari.
\newblock Large deviations for partition functions of directed polymers in an
  {IID} field.
\newblock {\em Ann. Inst. Henri Poincar\'e Probab. Stat.}, 45(3):770--792,
  2009.

\bibitem{beze-tind-vien}
S{\'e}rgio Bezerra, Samy Tindel, and Frederi Viens.
\newblock Superdiffusivity for a {B}rownian polymer in a continuous {G}aussian
  environment.
\newblock {\em Ann. Probab.}, 36(5):1642--1675, 2008.

\bibitem{Carmona-Hu-Gaussian}
Philippe Carmona and Yueyun Hu.
\newblock On the partition function of a directed polymer in a {G}aussian
  random environment.
\newblock {\em Probab. Theory Related Fields}, 124(3):431--457, 2002.

\bibitem{chen-even-klin-08}
Gui-Qiang Chen, Nadine Even, and Christian Klingenberg.
\newblock Hyperbolic conservation laws with discontinuous fluxes and
  hydrodynamic limit for particle systems.
\newblock {\em J. Differential Equations}, 245(11):3095--3126, 2008.

\bibitem{come-shig-yosh-03}
Francis Comets, Tokuzo Shiga, and Nobuo Yoshida.
\newblock Directed polymers in a random environment: path localization and
  strong disorder.
\newblock {\em Bernoulli}, 9(4):705--723, 2003.

\bibitem{come-yosh-05}
Francis Comets and Nobuo Yoshida.
\newblock Brownian directed polymers in random environment.
\newblock {\em Comm. Math. Phys.}, 254(2):257--287, 2005.

\bibitem{Comets-Gregorio-arc}
Francis Comets and Nobuo Yoshida.
\newblock Branching random walks in space-time random environment: Survival
  probability, global and local growth rates.
\newblock {\em \tt arXiv:0907.0509}, 2009.


\bibitem{cove-reza}
Paul Covert and Fraydoun Rezakhanlou.
\newblock Hydrodynamic limit for particle systems with nonconstant speed
  parameter.
\newblock {\em J. Statist. Phys.}, 88(1-2):383--426, 1997.


\bibitem{Diehl}
Stefan Diehl.
\newblock On scalar conservation laws with point source and discontinuous flux
  function.
\newblock {\em SIAM J. Math. Anal.}, 26(6):1425--1451, 1995.

\bibitem{evan}
Lawrence~C. Evans.
\newblock {\em Partial differential equations}, volume~19 of {\em Graduate
  Studies in Mathematics}.
\newblock American Mathematical Society, Providence, RI, 1998.


\bibitem{joha}
Kurt Johansson.
\newblock Shape fluctuations and random matrices.
\newblock {\em Comm. Math. Phys.}, 209(2):437--476, 2000.

\bibitem{Klin-Ris}
Christian Klingenberg and Nils~Henrik Risebro.
\newblock Convex conservation laws with discontinuous coefficients.
  {E}xistence, uniqueness and asymptotic behavior.
\newblock {\em Comm. Partial Differential Equations}, 20(11-12):1959--1990,
  1995.

\bibitem{Liu-Watbled-2009}
Quansheng Liu and Fr{\'e}d{\'e}rique Watbled.
\newblock Exponential inequalities for martingales and asymptotic properties of
  the free energy of directed polymers in a random environment.
\newblock {\em Stochastic Process. Appl.}, 119(10):3101--3132, 2009.

\bibitem{meja-04}
Olivier Mejane.
\newblock Upper bound of a volume exponent for directed polymers in a random
  environment.
\newblock {\em Ann. Inst. H. Poincar\'e Probab. Statist.}, 40(3):299--308,
  2004.

\bibitem{ostr-02}
Daniel~N. Ostrov.
\newblock Solutions of {H}amilton-{J}acobi equations and scalar conservation
  laws with discontinuous space-time dependence.
\newblock {\em J. Differential Equations}, 182(1):51--77, 2002.


\bibitem{petermann}
Markus Petermann.
\newblock {\em Superdiffusivity of directed polymers in random environment}.
\newblock Ph.D. thesis, University of Z{\"u}rich, 2000.

\bibitem{Reza-2002}
F.~Rezakhanlou,
\newblock Continuum limit for some growth models.
\newblock {\em Stochastic Process. Appl.}, 101(1):1--41, 2002.

\bibitem{rock-ca}
R.~Tyrrell Rockafellar.
\newblock {\em Convex analysis}.
\newblock Princeton Mathematical Series, No. 28. Princeton University Press,
  Princeton, N.J., 1970.

\bibitem{rost}
H.~Rost.
\newblock Nonequilibrium behaviour of a many particle process: density profile
  and local equilibria.
\newblock {\em Z. Wahrsch. Verw. Gebiete}, 58(1):41--53, 1981.

\bibitem{sepp-large-deviations}
T.~Sepp{\"a}l{\"a}inen.
\newblock Coupling the totally asymmetric simple exclusion process with a
  moving interface.
\newblock {\em Markov Process. Related Fields}, 4(4):593--628, 1998.
\newblock I Brazilian School in Probability (Rio de Janeiro, 1997).

\bibitem{sepp99K}
Timo Sepp{\"a}l{\"a}inen.
\newblock Existence of hydrodynamics for the totally asymmetric simple
  {$K$}-exclusion process.
\newblock {\em Ann. Probab.}, 27(1):361--415, 1999.

\bibitem{sepp01slow}
Timo Sepp{\"a}l{\"a}inen.
\newblock Hydrodynamic profiles for the totally asymmetric exclusion process
  with a slow bond.
\newblock {\em J. Statist. Phys.}, 102(1-2):69--96, 2001.

\bibitem{sepp-poly}
Timo Sepp{\"a}l{\"a}inen.
\newblock Scaling for a one-dimensional directed polymer with boundary
  conditions.
\newblock {\em To appear in Ann. Probab., {\tt arXiv:0911.2446}}, 2009.

\bibitem{sepp-valk-poly}
Timo Sepp{\"a}l{\"a}inen and Benedek Valk{\'o}.
\newblock Bounds for scaling exponents for a 1+1 dimensional directed polymer
  in a {B}rownian environment.
\newblock {\em Alea}, 8:451--476, 2010.

\bibitem{stro-ca}
Karl~R. Stromberg.
\newblock {\em Introduction to classical real analysis}.
\newblock Wadsworth International, Belmont, Calif., 1981.
\newblock Wadsworth International Mathematics Series.

\bibitem{watbled-arc}
Frederique Watbled.
\newblock Concentration inequalities for disordered models.
\newblock {\em \tt arXiv:1010.4914}, 2010.

\bibitem{wuth-98aihp}
Mario~V. W{\"u}thrich.
\newblock Fluctuation results for {B}rownian motion in a {P}oissonian
  potential.
\newblock {\em Ann. Inst. H. Poincar\'e Probab. Statist.}, 34(3):279--308,
  1998.

\bibitem{wuth-98aop}
Mario~V. W{\"u}thrich.
\newblock Superdiffusive behavior of two-dimensional {B}rownian motion in a
  {P}oissonian potential.
\newblock {\em Ann. Probab.}, 26(3):1000--1015, 1998.


\end{thebibliography}

\end{document}